\documentclass[a4paper,11pt]{article} 

\usepackage{amsmath, amsthm, amssymb, amsfonts, amscd, mathrsfs, color, graphicx}
\usepackage{url,hyperref}
\usepackage{float}

\newtheorem{theo}{Theorem}
\newtheorem{defi}[theo]{Definition}
\newtheorem{lem}[theo]{Lemma}

\newcommand{\R}{\mathbb{R}} 
\newcommand{\C}{\mathbb{C}} 
\newcommand{\Z}{\mathbb{Z}} 
\newcommand{\N}{\mathbb{N}} 
\newcommand{\Sone}{{\mathbb{S}^1}} 

\newcommand{\U}{\mathcal{U}}   
\newcommand{\Proj}{\mathbb{P}} 
\newcommand{\Hil}{\mathcal{H}} 
\newcommand{\Opr}{{\mathbb{O}_\Theta}} 

\newcommand{\ol}[1]{\overline{#1}}

\newcommand{\wt}[1]{\widetilde{#1}}
\newcommand{\wh}[1]{\widehat{#1}}
\newcommand{\one}{\mathbf{1}}

\newcommand{\supp}{\textnormal{supp}}
\newcommand{\ad}{\dagger}
\newcommand{\Ran}{\textnormal{Ran}}
\newcommand{\Ker}{\textnormal{Ker}}

\title{Reconstructing group wavelet transform from feature maps\\ with a reproducing kernel iteration}
\author{Davide Barbieri}

\begin{document}

\maketitle

\begin{abstract}

In this paper we consider the problem of reconstructing an image that is downsampled in the space of its $SE(2)$ wavelet transform, which is motivated by classical models of simple cells receptive fields and feature preference maps in primary visual cortex. We prove that, whenever the problem is solvable, the reconstruction can be obtained by an elementary project and replace iterative scheme based on the reproducing kernel arising from the group structure, and show numerical results on real images.
\end{abstract}

\section{Introduction}\label{s:intro}

This paper introduces an iterative scheme for solving a problem of image reconstruction, motivated by the classical behavior of primary visual cortex (V1), in the setting of group wavelet analysis. The mathematical formulation of the problem is that of the reconstruction of a function on the plane which, once represented as a function on the group $SE(2) = \R^2 \rtimes \Sone$ of rotations and translations of the Euclidean plane via the group wavelet transform, is known only on a certain two dimensional subset of this three dimensional group. This problem is equivalent to that of filling in missing information related to a large subset of the $SE(2)$ group, and ultimately inquires about the completenss of the image representation provided by feature maps observed in V1.

One of the main motivations for the present study comes indeed from neuroscience, and the modeling of classical receptive fields of simple cells in terms of group actions, restricted to feature maps such as the orientation preference maps. The attempts of modeling mathematically the measured behavior of brain's primary visual cortex (V1) have led to the introduction of the linear-nonlinear-Poisson (LNP) model \cite{Carandini05}, which define what is sometimes called the classical behavior. It describes the activity of a neuron in response to a visual stimulus as a Poisson spiking process driven by a linear operation on the visual stimulus, modeled by the receptive field of the neuron, passed under a sigmoidal nonlinearity. A series of thorough studies of single cell behavior could find a rather accurate description of the receptive fields of a large amount of V1 cells, called simple cells, within the LNP model, in terms of integrals against Gabor functions located in a given position of the visual field \cite{Marcelja80, Ringach02}. This description formally reduces the variability in the classical behavior of these cells to a few parameters, regulating the position on the visual field, the size, the shape and the local orientation of a two dimensional modulated Gaussian. A slightly simplified description of a receptive field activity $F$ in response to a visual stimulus defined by a real function $f$ on the plane, representing light intensity, is the following: denoting by $x = \binom{x_1}{x_2}, y = \binom{y_1}{y_2} \in \R^2$, 
\begin{equation}\label{eq:receptive}
F = s \sqrt{2 \pi} \int_{\R^2} f(y) e^{-2 \pi i p \big( (x_1 - y_1) \cos \theta + (x_2 - y_2) \sin\theta \big)} e^{-2\pi^2 s^2 |x - y|^2} dy
\end{equation}
where the parameters $s, p \in \R^+$ define the local scale and spatial frequency, the angle $\theta \in [0,2\pi)$ defines the local direction and $x \in \R^2$ define the local position of the receptive field in the visual field, while the complex formulation can be formally justified by the prevalence of so-called even and odd cells \cite{Ringach02}. We will focus on on the sole parameters $x,\theta$. This can be considered restrictive \cite{HubelWiesel62, SCP08, Barbieri15}, but it is nevertheless interesting since angles represent a relevant local feature detected by V1 \cite{HubelWiesel62} whose identification has given rise to effective geometrical models of perception \cite{PT99, CS06, CS15}. In this case one can model the linear part of the classical behavior of simple cells in terms of a well-known object in harmonic analysis: by rephrasing \eqref{eq:receptive} as
$$
F = F(x,\theta)
$$
for fixed values of $p,s$ we obtain a $SE(2)$ group wavelet transform of $f$ (see \S \ref{s:continuous}). On the other hand, classical experiments with optical imaging have shown that not all parameters $\theta \in \Sone$ are available for these cortical operations \cite{BonhofferGrinvald91, Welicky96, BoskingFitzpatrick97, White07}: in many mammals, a pinwheel-shaped function $\Theta : \R^2 \to \Sone$ determines the available angles. This feature map can be introduced in the model \eqref{eq:receptive} by saying that the receptive fields in V1 record the data
\begin{equation}\label{eq:data}
\{F(x,\Theta(x)) : x \in \R^2\}.
\end{equation}
Within this setting, a natural question is thus whether this data actually contains all the sufficient information to reconstruct the original image $f$, and how can we reobtain $f$ from these data. This is the main problem we aim to address.

Before proceding, we recall that a severe limitation of the purely spatial model \eqref{eq:receptive} is that of disregarding temporal behaviors \cite{DeAngelis95, Cocci12}. Moreover, the classical behavior described by the LNP model is well-known to be effective only up to a limited extent: several other mechanisms are present that describe a substantial amount of the neural activity in V1, such as Carandini-Heeger normalization \cite{CarandiniHeeger12}, the so-called non-classical behaviors \cite{Fitzpatrick00, Carandini05}, as well as the neural connectivity \cite{Angelucci02}. However, spatial receptive fields and the LNP provide relevant insights on the functioning of V1 \cite{Ringach02}. Moreover, the ideas behind the LNP model have been a main source of inspiration in other disciplines, notably for the design of relevant mathematical and computational theories, such as wavelets and convolutional neural networks \cite{Marr, LeCun10}. We also point out that the use of groups and invariances to describe the variability of the neural activity has proved to be a solid idea to build effective models \cite{PetitotBook, CS06, PA14}, whose influence on the design of artificial learning architectures is still very strong \cite{AERP19, MBCSx, Anselmi20, Duits21},

Another motivation for studying the completeness of the data collected with \eqref{eq:data} comes from the relationship of this problem with that of sampling in reproducing kernel Hilbert spaces (RKHS). The RKHS structure in this case is that of the range the group wavelet transform, and will be discussed in detail, together with the basics on the $SE(2)$ wavelet transform, in \S \ref{s:continuous}. Reducing the number of parameters in a wavelet transform is a common operation, that one performs e.g. for discretization purposes \cite{Daubechies}, or because it is useful to achieve square-integrability \cite{AAG91}. The main issue is that these operations are typically constrained nontrivially in order to retain the sufficient information that allows one to distinguish all interesting signals. For example, when discretizing the short-time Fourier transform to obtain a discrete Gabor Transform, the well-known density theorem \cite{Heil07} imposes a minimum density of points in phase space where the values of the continuous transform must be known in order to retain injectivity for signals of finite energy. Much is known on this class of problems in the context of sampling, from classical Shannon's theorem and uncertainty-related signal recovery results \cite{DonohoStark89}, to general sampling results in RKHS \cite{FGHKR17, GRS18}. However, the kind of restriction implemented by feature maps does not seem to fit into these settings, even if some similarities may be found between the pinwheel-shaped orientation preference maps \cite{BonhofferGrinvald91} and Meyer's quasicrystals, which have been recently used for extending compressed sensing results \cite{CandesTao06, MateiMeyer2010, AACM19}.

The plan of the paper is the following. In \S \ref{s:continuous}, after a brief overview of the $SE(2)$ transform and its RKHS structure, we will formalize in a precise way the problem of functional reconstruction after the restriction \eqref{eq:data} has been performed. In \S \ref{s:algorithm} we will introduce a technique to tackle with this problem, given by an iterative kernel method based on projecting the restricted $SE(2)$ wavelet transform onto the RKHS defined by the group representation. Moreover, we will consider a discretization of the problem and, in the setting of finite dimensional vector spaces, we will give a proof that the proposed iteration converges to the desired solution. Finally, in \S \ref{s:numerics} we will show and discuss numerical results on real images.

\newpage
\section{Overview of the $SE(2)$ transform}\label{s:continuous}

The purpose of this section is to review the fundamental notions of harmonic analysis needed to provide a formal statement of the problem. We will focus on the group wavelet transform defined by the action on $L^2(\R^2)$ of the group of rotations and translation of the Euclidean plane, expressed as a convolutional integral transform.

We will denote the Fourier transform of $f \in L^1(\R^2) \cap L^2(\R^2)$ by
$$
\wh{f}(\xi) = \int_{\R^2} e^{-2\pi i x.\xi} f(x) dx
$$
and, as customary, we will use the same notation for its extension by density to the whole $L^2(\R^2)$. We will also denote by $\ast$ the convolution on $\R^2$:
$$
f \ast g (x) = \int_{\R^2} f(y) g(x - y) dy.
$$

Let $\Sone$ be the abelian group of angles of the unit circle, which is isomorphic either to the one dimensional torus $\mathbb{T} = [0,2\pi)$ or to the group $SO(2)$ of rotations of the plane $\R^2$. The group $SE(2) = \R^2 \rtimes \Sone$ is the semidirect product group with elements $(x,\theta) \in \R^2 \times \Sone$ and composition law
$$
(x,\theta) \cdot (x',\theta') = (x + r_\theta x', \theta + \theta').
$$
Its Haar measure, that is the (Radon) measure on the group that is invariant under group operations, is the Lebesgue measure on $\R^2 \rtimes \Sone$.

A standard way to perform a wavelet analysis with respect to the $SE(2)$ group on two dimensional signals is given by the operator defined as follows.

\begin{defi}\label{def:se2trans}
Let us denote by $R : \Sone \to \U(L^2(\R^2))$ the unitary action by rotations of $\Sone$ on $L^2(\R^2)$:
$$
R(\theta)f(x) = f(r_\theta^{-1}x) \, , \ f \in L^2(\R^2) \, , \ \theta \in \Sone \, , \ \textnormal{ where }
r_\theta = \left(\begin{array}{cc}
\cos\theta & -\sin\theta\\
\sin\theta & \cos\theta
\end{array}\right).
$$
Let $\psi \in L^2(\R^2)$, and denote by $\psi_\theta = R(\theta)\psi$. The $SE(2)$ wavelet transform on $L^2(\R^2)$ with mother wavelet $\psi$ is
\begin{equation}\label{eq:W}
W_\psi f (x,\theta) = f \ast \psi_\theta (x) \, , \ f \in L^2(\R^2). 
\end{equation}
\end{defi}

In terms of this definition, we can then see that if we choose $s, p \in \R^+$, and let $\psi_{s,p} \in L^2(\R^2)$ be
\begin{equation}\label{eq:psi}
\psi_{s,p}(x) = s \sqrt{2 \pi} e^{-2 \pi i p x_1} e^{-2\pi^2 s^2 |x|^2},
\end{equation}
we can write \eqref{eq:receptive} as $F = W_{\psi_{s,p}} f (x, \theta)$.

Moreover, by making use of the quasiregular representation of the $SE(2)$ group
\begin{equation}\label{eq:quasiregular}
\Pi(x,\theta) f(y) = f(r_\theta^{-1}(y - x)) \, , \ f \in L^2(\R^2)
\end{equation}
and denoting by $\psi^\ad(x) = \ol{\psi(-x)}$, we can rewrite \eqref{eq:W} as follows:
$$
W_\psi f(x,\theta) = \langle f, \Pi(x,\theta)\psi^\ad \rangle_{L^2(\R^2)}
$$
which is a standard form to write the so-called group wavelet transform (see e.g. \cite{Fuhr, DeitmarEchterhoff}). Note that, in the interesting case \eqref{eq:psi}, we have $\psi_{s,p}^\ad = \psi_{s,p}$.

The $SE(2)$ transform \eqref{eq:W}, together with the notion of group wavelet transform, has been studied in multiple contexts (see e.g. \cite{Weiss01, Antoine2D, Fuhr, Duits07, DeitmarEchterhoff, DDGL} and references therein), and several of its properties are well-known. In particular, if $W_\psi f$ is a bounded operator from $L^2(\R^2)$ to $L^2(\R^2 \times \Sone)$, which happens e.g. for $\psi \in L^1(\R^2) \cap L^2(\R^2)$ by Young's convolution inequality and the compactness of $\Sone$, it is easy to see that its adjoint reads
\begin{equation}\label{eq:adjoint}
W_\psi^* F (x) = \int_\Sone F(\cdot,\theta) \ast \psi_\theta^\ad\, (x) \,d\theta.
\end{equation}

It is also well-known \cite{Weiss01} that $W_\psi f$ can not be injective on the whole $L^2(\R^2)$, that is, we can not retrieve an arbitrary element of $L^2(\R^2)$ by knowing its $SE(2)$ transform. However, as shown in \cite{Duits04, Duits07}, and applied successfully in a large subsequent series of works (e.g. \cite{Duits10, Duits16, Duits21}), it is possible to obtain a bounded invertible transform by extending the notion of $SE(2)$ transform to mother wavelets $\psi$ that do not belong to $L^2(\R^2)$, or by simplifying the problem and reduce the wavelet analysis to the space of bandlimited functions, that are those functions whose Fourier transform is supported on a bounded set, whenever a Calder\'on's admissibility condition holds. Since our main point in this paper is not the reconstruction of the whole $L^2(\R^2)$, we will consider the $SE(2)$ transform with this second, simplified, approach. In this case, the image of the $SE(2)$ transform is a reproducing kernel Hilbert subspace of $L^2(\R^2 \times \Sone)$ whose kernel will play an important role in the next section. For convenience, we formalize these statements with the next two theorems, and provide a sketch of the proof, even if they can be considered standard material.

\begin{theo}\label{th:SE2inv}
For $R > 0$, let $B_R = \{\xi \in \R^2 : |\xi| < R\}$ and let
\begin{equation}\label{eq:PaleyWiener}
\Hil_{R} = \{f \in L^2(\R^2) : \supp \wh{f} \subset B_R\}.
\end{equation}
The $SE(2)$ wavelet transform \eqref{eq:W} for a mother wavelet $\psi \in L^2(\R^2)$ is a bounded injective operator from $\Hil_{R}$ to $L^2(\R^2 \times \Sone)$ if and only if there exist two constants $0 < A \leq B < \infty$ such that
\begin{equation}\label{eq:CalderonSE2}
A \leq \int_{\Sone} |\wh{\psi}(r_\theta^{-1}\xi)|^2 d\theta \leq B 
\end{equation}
holds for almost every $\xi \in B_R$.
\end{theo}
\begin{proof}[Sketch of the proof]
By Fourier convolution theorem and the definition of $\Hil_R$, we have
$$
\int_\Sone \int_{\R^2} |W_\Psi f(x,\theta)|^2 dx d\theta = \int_\Theta \int_{\R^2} |\wh{\psi_\theta}(\xi)|^2 |\wh{f}(\xi)|^2 d\xi d\theta = \int_{B_R} \left(\int_\Sone |\wh{\psi_\theta}(\xi)|^2 d\theta \right) |\wh{f}(\xi)|^2 d\xi.
$$
Thus, by Plancherel's theorem, \eqref{eq:CalderonSE2} is equivalent to the so-called frame condition
\begin{equation}\label{eq:frames}
A \|f\|^2_{L^2(\R^2)} \leq \|W_\psi f\|^2_{L^2(\R^2 \times \Sone)} \leq B \|f\|^2_{L^2(\R^2)}
\end{equation}
for all $f \in \Hil_R$, which is equivalent (see e.g. \cite{AAG93}) to $W_\Psi$ being bounded and invertible on $\Hil_R$.
\end{proof}

\vspace{-2ex}

Before stating the next result we repeat an observation of \cite{Weiss01} and show, using this theorem, that the $SE(2)$ transform can not be a bounded injective operator on the whole $L^2(\R^2)$. Indeed, using that
$$
\wh{\psi_\theta}(\xi) = \int_{\R^2} e^{-2\pi i x.\xi} \psi(r_{\theta}^{-1}x) dx = \int_{\R^2} e^{-2\pi i x.(r_{\theta}^{-1}\xi)} \psi(x) dx = \wh{\psi}(r_{\theta}^{-1}\xi).
$$
we can see that the Calder\'on's function in condition \eqref{eq:CalderonSE2} is actually a radial function
\begin{equation}\label{eq:polarCalderonfunction}
C_\psi(\xi) = \int_{\Sone} |\wh{\psi}(r_\theta^{-1}\xi)|^2 d\theta = \int_\Sone |\wh{\psi}(|\xi|\cos\varphi,|\xi|\sin\varphi)|^2 d\varphi = C_\psi(|\xi|)
\end{equation}
which, by Plancherel's theorem, satisfies
$
\int_0^\infty C_\psi(\rho) \rho d\rho = \|\psi\|_{L^2(\R^2)}^2.
$
Hence, the lower bound in condition \eqref{eq:CalderonSE2} can not be satisfied on the whole $\R^2$ by any $\psi \in L^2(\R^2)$.

On the other hand, for the mother wavelet \eqref{eq:psi}, since $\displaystyle \wh{\psi}(\xi) = \frac{1}{s \sqrt{2 \pi}} e^{-|\xi + \binom{p}{0}|^2/2 s^2}$, we have
$$
C_\psi(\xi) = \int_\Sone |\wh{\psi_\theta}(r_\theta^{-1}\xi)|^2 d\theta = \frac{1}{2 \pi s^2} \int_\Sone e^{-|\xi + \binom{p\cos\theta}{p\sin\theta}|^2/s^2} d\theta = \frac{e^{-\frac{|\xi|^2 + p^2}{s^2}}}{2 \pi s^2} \int_\Sone e^{-\frac{|\xi| p}{s^2}\cos\alpha} d\alpha.
$$
From here we can see that $C_\psi(\xi) > 0$ for all $\xi \in \R^2$, so, even if $C_\psi(\xi) \to 0$ as $|\xi| \to \infty$, we have that the lower bound in \eqref{eq:CalderonSE2}
is larger than zero for any finite $R$.

The next theorem shows how to construct the inverse of the $SE(2)$ wavelet transform on bandlimited functions, and what is the structure of the closed subspace defined by its image.

\begin{theo}\label{th:SE2proj}
Let $\psi \in L^2(\R^2)$ and $R > 0$ be such that \eqref{eq:CalderonSE2} holds. Let also $\gamma \in L^2(\R^2)$ be defined by
\begin{equation}\label{eq:gamma}
\wh{\gamma}(\xi) = \frac{\chi_{B_R}(\xi)}{C_\psi(\xi)} \, \wh{\psi}(\xi) 
\end{equation}
where $\displaystyle \chi_{B_R}(\xi) = \left\{\begin{array}{cl}1 & \xi \in B_R\\ 0 & \textnormal{otherwise} \end{array}\right.$, and let $\Hil_R$ be as in \eqref{eq:PaleyWiener}. Then
\begin{itemize}
\item[(i)] For all $f \in \Hil_R$ it holds $W_\gamma^* W_\psi f = f$.
\item[(ii)] The space $W_\psi(\Hil_R)$ is a reproducing kernel Hilbert subspace of continuous functions of $L^2(\R^2 \times \Sone)$, and the orthogonal projection $\Proj$ of $L^2(\R^2 \times \Sone)$ onto $W_\psi(\Hil_R)$ is
\begin{equation}\label{eq:RKHSproj}
\Proj F (x,\theta) = W_\psi W_\gamma^* F (x,\theta) = \int_\Sone F(\cdot,\theta') \ast \psi_\theta \ast \gamma_{\theta'}^\ad (x) \, d\theta' \ , \quad F \in L^2(\R^2 \times \Sone).
\end{equation}
\end{itemize}
\end{theo}
\begin{proof}[Sketch of the proof]
Observe first that
$\displaystyle
\wh{\gamma_\theta^\ad} = \ol{\wh{\gamma_\theta}} = \frac{\chi_{B_R}(\xi)}{C_\psi(\xi)} \ol{\wh{\psi_\theta}(\xi)}.
$
Thus, recalling \eqref{eq:adjoint} and using Fourier convolution Theorem, we can compute
$$
\wh{W_\gamma^*W_\psi f}(\xi) = \int_\Sone \wh{f}(\xi) \frac{\chi_{B_R}(\xi)}{C_\psi(\xi)} |\wh{\psi_\theta}(\xi)|^2 d\theta = \wh{W_\psi^*W_\gamma f}(\xi) = \wh{f}(\xi)\chi_{B_R}(\xi)
$$
which proves (i). To prove (ii), one needs to show that the elements of $W_\psi(\Hil_R)$ are continuous functions, that $W_\psi(\Hil_R)$ $W_\psi W_\gamma^*$ is selfadjoint and idempotent, and that \eqref{eq:RKHSproj}. The continuity can be obtained as a consequence of the continuity of the unitary representation \eqref{eq:quasiregular} and, by the same arguments used to prove (i) we can easily see that $W_\gamma W_\psi^* = W_\psi W_\gamma^*$. On the other hand, (i) implies that $W_\psi W_\gamma^* W_\psi f = W_\psi f$ for all $f \in \Hil_R$, hence $W_\psi W_\gamma^* F = F$ for all $F \in W_\psi(\Hil_R)$. Equation \eqref{eq:RKHSproj} can be obtained directly by \eqref{eq:adjoint} and the definition of $W_\psi$.
\end{proof}

Since $W_\psi(\Hil_R)$ is a Hilbert space of continuous functions, it makes sense to consider its values on the zero measure set provided by the graph of a function $\Theta : \R^2 \to \Sone$, as in \eqref{eq:data}. We can then provide a formal statement of the problem discussed in \S \ref{s:intro}:
\begin{equation}\label{eq:problem}
\textnormal{for } f \in \Hil_R \textnormal{ and } \Theta : \R^2 \to \Sone \textnormal{, reconstruct } f \textnormal{ using only the values } W_\psi f (x,\Theta(x)).
\end{equation}

In order for this problem to be solvable, the graph $\mathscr{G}_\Theta = \{(x,\theta) \in \R^2 \times \Sone : \theta = \Theta(x)\}$ must be a set of uniqueness for $W_\psi(\Hil_R)$. That is, the following condition must hold:
\begin{equation}\label{eq:uniqueness}
\textnormal{if } F \in W_\psi(\Hil_R) \textnormal{ and } F|_{\mathscr{G}_\Theta} = 0 \textnormal{ , then } F = 0.
\end{equation}
Indeed, if this was not the case, for any nonzero $F \in W_\psi(\Hil_R)$ that vanishes on $\mathscr{G}_\Theta$, the function $f_F = W_\gamma^* (W_\psi f + F) \in \Hil_R$ would be different from $f$ but $W_\psi f_F$ would coincide with $W_\psi f$ on $\mathscr{G}_\Theta$. That is, we could not solve problem \eqref{eq:problem}.

Condition \eqref{eq:uniqueness} is in general hard to be checked, and the formal characterization of the interplay between $\psi$ and $\Theta$ that make it hold true is out of the scopes of this paper. However, in the next section we will provide a technique for addressing \eqref{eq:problem} in a discrete setting, which will allow us to explore in \S \ref{s:numerics} the behavior of this problem for various functions $\Theta$ inspired by the feature maps measured in V1.

\section{A reconstruction algorithm}\label{s:algorithm}

In this section we describe the discretization of the problem \eqref{eq:problem} which is used in the next section. Then, we introduce a kernel based iterative algorithm, and prove its convergence to the solution whenever the solvability condition \eqref{eq:uniqueness} holds. \vspace{2ex}

\subsection{Discretization of the problem}

The setting described in \S \ref{s:continuous} can be discretized in a standard fashion by replacing $L^2(\R^2)$ with $\C^{N \times N}$, endowed with the usual Euclidean scalar product, circular convolution and discrete Fourier transform (FFT), which amounts to replacing $\R$ with $\Z_N$, the integers modulo $N$.  Explicitly, given $f, \psi \in \C^{N \times N}$, $x = \binom{x_1}{x_2}, y = \binom{y_1}{y_2}, \xi = \binom{y_1}{\xi_2} \in \Z_N \times \Z_N$, we have
$$
f \ast \psi (x) = \sum_{y_1 = 0}^{N-1} \sum_{y_2 = 0}^{N-1} f(y) g\big( (x - y) \!\!\!\!\!\mod N\big) \, , \textnormal{ and }
\wh{f}(\xi) = \sum_{x_1 = 0}^{N-1} \sum_{x_2 = 0}^{N-1} e^{-2\pi i \frac{x_1\xi_1 + x_2\xi_2}{N}} f(x).
$$
With a uniform discretization of angles, by replacing $\Sone$ with $\frac{2\pi}{M}\Z_M = \{0, \frac{2\pi}{M}, \frac{4\pi}{M}, \dots, 2\pi\frac{M-1}{M}\}$ we obtain the following discretization of the $SE(2)$ transform with the mother wavelet \eqref{eq:psi}:
$$
W_\psi f(x,j) = f \ast \psi_{\theta_j} (x) \, , \textnormal{ where } \wh{\psi_{\theta_j}}(\xi) = e^{-\big|\xi + \binom{p\cos\theta_j}{p\sin\theta_j}\big|^2/2 s^2}
$$
for $x, \xi \in \Z_N \times \Z_N$ and $\theta_j = \frac{2\pi}{M} j$, $j = 0,\dots,M$. Thus, in particular, $W_\psi f \in \C^{N \times N \times M}$. Note that here, for simplicity, we have removed the normalization used in \eqref{eq:psi}.

This allows us to to process $N \times N$ digital images while retaining the results of Theorems \ref{th:SE2inv} and \ref{th:SE2proj} as statements on finite frames (see e.g. \cite{CasazzaKutyniok}) since Plancherel's theorem and Fourier convolution theorem still hold. In particular, when computing numerically the inverse of $W_\psi$  using (i), Theorem \ref{th:SE2proj}, one has to choose an $R > 0$ so that Calder\'on's condition \eqref{eq:CalderonSE2} for
\begin{equation}\label{eq:CalderonDiscrete}
C_\psi(\xi) = \sum_{j = 0}^{M - 1}  e^{-\big|\xi + \binom{p\cos\theta_j}{p\sin\theta_j}\big|^2/2 s^2}
\end{equation}
holds with some $0 < A \leq B < \infty$ for all $\xi \in B_R = \{\xi \in \Z_N \times \Z_N : \xi_1^2 + \xi_2^2 < R^2 \}$. This is the injectivity condition on $\Hil_{R} = \{f \in \C^{N \times N} : \wh{f}(\xi) = 0 \, \forall \, \xi \notin B_R\}$ and, due to the finiteness of the space, now it can be achieved for all $R$, i.e. without bandlimiting. However, since this is equivalent to the frame inequalities \eqref{eq:frames}, the quantity $\sqrt{\frac{B}{A}}$ defines actually the condition number of $W_\psi $. Thus, in order to keep stability for the inversion, the ratio $\frac{B}{A}$ can not be arbitrarily large (see e.g. \cite{CasazzaKutyniok, Duits07}). Once the parameters $s,p,R$ are chosen in such a way that this ratio provides an acceptable numerical accuracy, one can then compute the projection $\Proj$ given by (ii), Theorem \ref{th:SE2proj}, on $F \in \C^{N \times N \times M}$, by making use of Fourier convolution theorem:
\begin{equation}\label{eq:RKdiscrete}
\wh{\Proj F} (\xi,j) = \frac{\chi_{B_R}(\xi)}{C_\psi(\xi)}\wh{\psi_{\theta_j}}(\xi) \sum_{\ell = 0}^{M - 1} \wh{F}(\xi,\ell) \ol{\wh{\psi_{\theta_\ell}}(\xi)}.
\end{equation}

We note at this point that this standard discretization, in general (for $M$ different from $2$ or $4$), retains all of the approach of \S \ref{s:continuous} but the overall semidirect group structure of $\R^2 \rtimes \Sone$.

Let us now consider the discretization of the problem \eqref{eq:problem} and denote the graph of $\Theta : \Z_N \times \Z_N \to \Z_M$ by $\mathscr{G}_\Theta = \{(x,j) \in \Z_N \times \Z_N \times \Z_M : j = \Theta(x)\}$. If we denote by $\Opr$ the selection operator that sets to zero all the components of an $F \in \C^{N \times N \times M}$ that do not belong to $\mathscr{G}_\Theta$, that is
$$
\Opr F(x,j) = \left\{
\begin{array}{cl}
F(x,j) & (x,j) \in \mathscr{G}_\Theta\\
0 & (x,j) \notin \mathscr{G}_\Theta
\end{array}
\right.
$$
we can see that this is now an orthogonal projection of $\C^{N \times N \times M}$. Hence, problem \eqref{eq:problem} can be rewritten in the present discrete setting as follows: given $f \in \Hil_R$, find $F \in \C^{N \times N \times M}$ that solves the linear problem
\begin{equation}\label{eq:SE2system}
\left\{
\begin{array}{rl}
\Proj F & = F\\
\Opr F & = \Opr W_\psi f.
\end{array}
\right.
\end{equation}
The meaning of \eqref{eq:SE2system} is indeed to recover $W_\psi f$ knowing only the values $W_\psi f (x,\Theta(x))$.

We propose to look for such a solution using the following iteration in $\C^{N \times N \times M}$: for $F_0 = \Opr W_\psi f$, compute
\begin{equation}\label{eq:iteration}
F_n = \Proj F_{n - 1} - \Opr \Proj F_{n - 1} + F_0 \ , \quad n = 1, 2, \dots
\end{equation}
The idea behind this iteration is elementary: we start with the values of $W_\psi f$ selected by $\Theta$, we project them on the RKHS defined by the image of $W_\psi$, and we replace the values on $\mathscr{G}_\Theta$ of the result with the known values of $W_\psi f$. The convergence of this iteration is discussed in the next section. Before that, we observe that \eqref{eq:iteration} can be seen as a linearized version of the Wilson, Cowan and Ermentrout equation \cite{WC72, EC80} since, denoting by $K = (\one - \Opr)\Proj$, it is a discretization of
$$
\frac{d}{dt} F_t = - F_t + K F_t + F_0.
$$
Apart from the absence of a sigmoid, this is indeed a classical model of population dynamics. Here, the ``kernel'' $K$ is not just the reproducing kernel $\Proj$, but it also contains the projection on a feature map $\Opr$. Returning to the model of V1, here the forcing term $F_0$ is the data collected by simple cells, while the stationary solution $F$, if it exists, is the full group representation of the visual stimulus defined as solution to the Volterra-type equation $F = K F + F_0$.\vspace{2ex}

\subsection{The project and replace iteration}

We show here that, whenever the problem \eqref{eq:SE2system} is solvable, the iteration \eqref{eq:iteration} converges to its solution. Since the argument is general, we will consider in this subsection the setting of an arbitrary finite dimensional vector space $V$ endowed with a scalar product and the induced norm, and two arbitrary orthogonal projections $P, Q$. For an orthogonal projection $P$ we will denote by $P^\bot = \one - P$ the complementary orthogonal projection.

We start with a basic observation, which is just a restatement of the solvability condition \eqref{eq:uniqueness} as that of a linear system defined by an orthogonal projection, in this case $Q$, on a subspace, in this case characterized as $\Ran(P)$. The simple proof is included for convenience.  
\begin{lem}
Let $P, Q$ be orthogonal projections of a finite dimensional vector space $V$. The system
\begin{equation}\label{eq:system}
\left\{
\begin{array}{rl}
P F & = F\\
Q F & = Q \wt{F}
\end{array}
\right.
\end{equation}
has a unique solution in $V$ for any $\wt{F} \in \Ran(P)$ if and only if
\begin{equation}\label{eq:condition}
\Ker(Q) \cap \Ran(P) = \{0\} .
\end{equation}
\end{lem}
\begin{proof}
The system \eqref{eq:system} is of course solved by $F = \wt{F}$, so we only need to prove that, for all $\wt{F} \in \Ran(P)$, this solution is unique if and only if \eqref{eq:condition} holds. In order to see this, let $S \in V$ be a solution to \eqref{eq:system}, and denote by $E = \wt{F} - S$. Then $E \in \Ran(P) \cap \Ker(Q)$. Hence, \eqref{eq:condition} holds if and only if $E = 0$, that is, the only solution to \eqref{eq:system} is $F = \wt{F}$.
\end{proof}

The problem posed by the system \eqref{eq:system} is a problem of linear algebra: if we know that a vector $F$ belongs to a given subspace $\Ran(P) \subset V$, and we know the projection of $F$ on a different subspace $\Ran(O) \subset V$, can we recover $F$? The next theorem shows that, if the system \eqref{eq:system} has a unique solution, such a solution can be obtained by the project and replace iteration \eqref{eq:iteration}.

\begin{theo}\label{th:main}
Let $V$ be finite dimensional vector space, and let $P, Q$ be orthogonal projections of $V$. Given $\wt{F} \in \Ran(P)$, set $F_0 = Q\wt{F}$, and let $\{F_n\}_{n \in \N}, \{H_n\}_{n \in \N} \subset V$ be the sequences defined by the project and replace iteration
\begin{equation}\label{eq:PRiteration}
\left\{
\begin{array}{rl}
H_n & = P F_{n - 1}\vspace{4pt}\\
F_n & = Q^\bot H_n + F_0
\end{array}
\right. \ , \quad n = 1, 2, \dots
\end{equation}
If condition \eqref{eq:condition} holds, then
$$
\lim_{n \to \infty} H_n = \lim_{n \to \infty} F_n = \wt{F}
$$
and the errors $\|\wt{F} - H_n\|$, $\|\wt{F} - F_n\|$ decay exponentially with the number of iterations $n$.
\end{theo}

\begin{proof}
Denoting by $H_1 = P F_0$, we can rewrite the iteration \eqref{eq:PRiteration} as two separate iterations, generating the two sequences $\{F_n\}_{n \in \N}, \{H_n\}_{n \in \N}$ as follows:
\begin{equation}\label{eq:series}
\left\{
\begin{array}{ccccccl}
H_{n+1} & = &  P Q^\bot H_n + H_1 & = & \ldots &
= & \displaystyle \sum_{k = 0}^n (P Q^\bot)^k H_1\vspace{4pt}\\
F_n & = & Q^\bot P F_{n - 1} + F_0 & = & \ldots &
= & \displaystyle \sum_{k = 0}^n (Q^\bot P)^k F_0 .
\end{array}
\right.
\end{equation}

If \eqref{eq:condition} holds then $\|P Q^\bot\| < 1$ and $\|Q^\bot P\| < 1$, because $\Ran(\Proj) \cap \Ran(\Opr^\bot) = \Ran(\Proj) \cap \Ker(\Opr) = \{0\}$. Hence, the existence of the limits $H = \lim_{n \to \infty} H_n$ and $F = \lim_{n \to \infty} F_n$ is due to the convergence of the Neumann series (see e.g. \cite{RieszNagy}). Thus, getting rid of the notations $F_0 = Q \wt{F}$ and $H_1 = P F_0$, we have that the limits $H$ and $F$ are the solutions to the equations
$$
\left\{
\begin{array}{rcl}
(\one - P Q^\bot) H & = & P Q \wt{F} \vspace{4pt}\\
(\one - Q^\bot P) F & = & Q \wt{F}.
\end{array}
\right.
$$

We can now see that $F = \wt{F}$, because this equation reads
$$
(\one - Q^\bot P)^{-1} Q \wt{F} = \wt{F}
$$
and, using that $\wt{F} \in \Ran(P)$, we can rewrite it as
$$
Q \wt{F} = (\one - Q^\bot P) \wt{F} = \wt{F} - Q^\bot \wt{F}
$$
which is true by definition of $Q^\bot = \one - Q$. Moreover, by taking the limit in the first equation of \eqref{eq:PRiteration}, and using again that $\wt{F} \in \Ran(P)$, we obtain
$$
H = P F = P \wt{F} = \wt{F}.
$$

Finally, the convergence of \eqref{eq:iteration} is exponential because, using the series expression \eqref{eq:series} for $\wt{F} = F$,
\begin{align*}
\|\wt{F} - F_n\| & = \| \sum_{k = 0}^\infty (Q^\bot P)^k F_0 - \sum_{k = 0}^n (Q^\bot P)^k F_0 \| = \| \sum_{k = n+1}^\infty (Q^\bot P)^k F_0\|\\
& = \| (Q^\bot P)^{n+1} \sum_{k = 0}^\infty (Q^\bot P)^k F_0\| = \| (Q^\bot P)^{n+1} F \| \leq \|Q^\bot P\|^{n+1} \|F\|
\end{align*}
and a similar argument applies to $\|\wt{F} - H_n\|$.
\end{proof}

\section{Numerical results}\label{s:numerics}

We present here the reconstruction results of the project and replace iteration on the restriction to feature maps of the $SE(2)$ transform of real images. We have chosen eight $512 \times 512$ pixels, 8-bit grayscale digital images $\{f_i\}_{i = 1}^8 \subset \{0,\dots,255\}^{512 \times 512}$, which are shown in Figure \ref{fig:dataset} together with their Fourier spectra. Note that, for processing, they have been bandlimited in order to formally mantain the structure described in \S \ref{s:continuous}. However, this bandlimiting has minimal effects, not visible to the eye, since the spectra have a strong decay: for this reason, the bandlimited images are not shown.

For the discrete $SE(2)$ transform we have chosen a discretization of $\Sone$ with 12 angles, so that, with respect to the notation of the previous section, we have $N = 512$ and $M = 12$. We have also chosen the parameter values $s = 51, p = 170, R = 252$. The mother wavelet $\psi$ and dual wavelet $\gamma$ produced by these parameters are shown in Figure \ref{fig:SE2}, top and center, on a crop of the full $512 \times 512$ space for better visualization. We stress that the stability of the transform and the numerical accuracy of the projection \eqref{eq:RKdiscrete} depend only on the behavior of the Calder\'on's function, while the accuracy of image representation under bandlimiting with radius $R$ depends on the decay of the spectra of the images. The Calder\'on's function $C_\psi$, computed as in \eqref{eq:CalderonDiscrete}, is shown in Figure \ref{fig:SE2}, bottom left: the chosen parameters define a ratio $B/A \approx 6 \cdot 10^3$, corresponding to a condition number for $W_\psi$ of less than $10^2$. In Figure \ref{fig:SE2}, bottom right, we have shown in $\log_{10}$ scale the multiplier $\chi_{B_R}/C_\psi$ that defines the dual wavelet $\gamma$ as in \eqref{eq:gamma}, and in particular we can see the bandlimiting radius, to be compared with the spectra of Figure \ref{fig:dataset}. For visualization purposes, in Figure \ref{fig:RK} we have shown in spatial coordinates the integral kernel defining the projection \eqref{eq:RKdiscrete}, which is the reproducing kernel for the discrete $SE(2)$ transform. Its real and imaginary parts are shown on the same crop used to display the wavelet of Figure \ref{fig:SE2}.

We have implemented the iteration \eqref{eq:iteration} for the restriction of this discrete $SE(2)$ transform to different types of feature maps $\Theta$, shown in Figure \ref{fig:pins}. We will comment below each one of these cases. We have chosen to illustrate the effect of that iteration as follows. We have computed the sequence $\{H_n\}_{n = 1}^{\nu}$ as in \eqref{eq:PRiteration}, for a number $\nu$ of iterations, and applied the inverse $SE(2)$ transform $W_\gamma^*$ to each $H_n$. This allows us to obtain real images that are directly comparable with the original ones. We then have shown $W_\gamma^* H_1$, representing the first image that can be directly retrieved from the feature parameters, and $W_\gamma^* H_\nu$, that is the image obtained when we stopped the iteration. Moreover, as a measure of reconstruction error, we have considered the following rescaling of the Euclidean norm, at each step $n \in \{1,\dots,\nu\}$:
\begin{equation}\label{eq:DELTA}
\Delta_n = 100*\left(\frac{1}{N^2}\sum_{x \in \Z_N \times \Z_N}|f_i(x) - W_\gamma^* H_n[f_i](x)|^2\right)^\frac12/255 = 100*\frac{\|f_i - W_\gamma^* H_n[f_i]\|}{255*N}.
\end{equation}
This adimensional quantity measures a \% error obtained as the average square difference by pixel of an image $f_i$ in the dataset from its $n$-steps reconstruction $W_\gamma^* H_n[f_i]$, divided by the size of the admissible pixel range for 8 bit images, which is $\{0,\dots,255\}$.\vspace{1ex}

\paragraph*{First feature map: purely random selection of angles.} The first map that we have considered is a $\Theta : \Z_N \times \Z_N \to \Z_M$ that, for each $x \in \Z_N \times \Z_N$, simply choses one value in $\Z_M$ as a uniformly distributed random variable. This map is shown in the first line of Figure \ref{fig:pins}, left, and in the first line of Figure \ref{fig:pins}, right, we have shown an enlargement to the same crop at which the wavelets in Figure \ref{fig:SE2} and the reproducing kernel in Figure \ref{fig:RK} are shown. In Figure \ref{fig:S1} we have shown the images resulting from $W_\gamma^* H_1$ and $W_\gamma^* H_{1000}$, and the evolution of the error \eqref{eq:DELTA} in $\log_{10}$ scale, respectively in the left, center and right column, for the first four images of Figure \ref{fig:dataset}. In this case we can see that the error $\Delta_n$ goes beyond $1$\%, indicated by $0$ on the $y$ axis, in just about 500 iterations. As a remark, this kind of feature maps are commonly encountered in rodents (see e.g. \cite{HoAngelucci21} and references therein). \vspace{1ex}

\paragraph*{Pinwheel-shaped feature maps.} 
Next, we present the results for three selection maps $\Theta$ that are pinwheel-shaped, as it is commonly observed for orientation and direction preference maps of V1 in primates and other mammals. These maps can be constructed as follows \cite{PetitotBook}: for $\rho \in \R^+$, let $\phi_\rho : \Z_N \times \Z_N \to \C$ be given by
$$
\phi_\rho(x) = \int_0^{2\pi} e^{i \big(\rho (x_1 \cos(\alpha) + x_2 \sin(\alpha)) + \Gamma(\alpha) \big)} d\alpha
$$
where $\Gamma$ is a purely random process with values in $[0,2\pi)$. The maps $\Theta_\rho : \Z_N \times \Z_N \to \Z_M$ that we have considered are obtained by
\begin{equation}\label{eq:generatepinwheels}
\Theta_\rho(x) = \left\lfloor \frac{M}{2\pi}\textnormal{angle}(\phi_\rho(x))\right\rfloor
\end{equation}
where angle$(z)$ is the phase of a complex number $z \in \C$, and $\lfloor t \rfloor$ is the integer part of a real number $t$. The resulting maps are quasiperiodic, with a characteristic correlation length which corresponds to the fact that the spectrum of $\phi_\rho$ is concentrated on a ring of radius $\frac{\rho}{2\pi}$. The main feature of those maps $\Theta_\rho$ is that they possess points, called pinwheel points, around which all angles are mapped, and these points are spaced, on average, by a distance of $2\pi/\rho$ (see e.g. \cite{PetitotBook} and references therein).

In the second, third and fourth line of Figure \ref{fig:pins}, left, we have shown the resulting maps $\Theta_\rho$ with $\rho$ respectively given by $\rho = 0.8$, $\rho = 0.4$ and $\rho = 0.06$. On the right column we have shown an enlargement to the same crop used before.

The results of the iteration are presented, as described above, in Figures \ref{fig:S2}, \ref{fig:S3}, \ref{fig:S4}, whose structure is the same as in Figure \ref{fig:S1} with the exception of the number of iteration, that is now larger.

In order to discuss these results, we first recall that the correlation length of orientation preference maps has been often related to the \emph{size} of receptive fields, as a ``coverage'' constraint to obtain a faithful representation of the visual stimulus (see \cite{Bosking02, Keil11, BCS14} and references therein). However, no mathematical proof of this principle in terms of image reconstruction has been given so far, nor has the word \emph{size} received a more quantitative meaning within the models of type \eqref{eq:receptive}, and we have not tried to give any more specific meaning either. However, for the three cases that we present, by comparing the crops of the maps $\Theta_\rho$ in Figure \ref{fig:pins} with the wavelet $\psi$ in Figure \ref{fig:SE2}, we can see that for $\rho = 0.8$ the correlation length of the map $\Theta_\rho$ is approximately similar to what we could call effective support of the receptive field, while for $\rho = 0.4$ we have that the area of influence of the receptive field does not include two different pinwheel points, and for $\rho = 0.06$ the two scales are very different. Heuristically, one could then be led to think that the reconstruction properties in the three cases may present qualitative differences. For example, that condition \eqref{eq:uniqueness}, or its discrete counterpart \eqref{eq:condition}, may hold in the first case and may not hold in the third case.

What we can see by the numerical results of the proposed algorithm are that, actually, there is a difference in the behavior of the decay. For larger values of the parameter $\rho$, when the map $\Theta_\rho$ is more similar to the purely random selection described above, the decay of the error is faster, while for smaller values of $\rho$ the decay is slower, but nevertheless the error appears to be monotonically decreasing. In the presented cases, for $\rho = 0.8$, we can see in the right column of Figure \ref{fig:S2} that in about 2000 iterations the error decay appears to enter an exponential regime, which is rectilinear in the $\log_{10}$ scale. We see in Figure \ref{fig:S3} that it takes roughly twice as many iterations for $\rho = 0.4$ to enter the same regime. On the other hand, for $\rho = 0.06$ we can see in Figure \ref{fig:S4} that after a relatively small number of iterations the decay becomes very small, and does not seem to become exponential even after 10000 iterations. However, visual inspection of the results (which ``measures'' the error in a different way than the Euclidean norm) in this last case, show that the starting image appears to be qualitatively highly corrupted, while the image obtained after the iteration was stopped is remarkably true to the original one and does not display evident artifacts away from the boundaries. \vspace{1ex}

\paragraph*{Selection of a single orientation: deconvolution.}
The surprisingly good performance in the reconstruction problem for the las map $\Theta_{0.06}$ has motivated a performance test for an additional feature selection map, given by a function $\Theta$ that is independent of $x$, hence selecting just one angle in the $SE(2)$ transform. In the same fashion as seeing the purely random distribution as a limiting case for large $\rho$ of the pinwheel-shaped maps, this can be considered as a limiting case for small $\rho$. However, keeping only the values of $W_\psi f$ for a single angle $\theta$ concretely corresponds to performing a convolution with one function $\psi_\theta$, and aiming to reconstruct $f$ is actually a deconvolution problem. In this case, the frequency behavior of the convolution filter is that of a Gaussian centered away from the origin, and its shape can be observed in the Calder\'on's function of Figure \ref{fig:SE2}, which clearly shows the sum of 12 such Gaussians. We have shown the reconstruction results for this problem, using now 15000 iterations, in Figure \ref{fig:S5}. The error decay seems to share, qualitatively, much of the same properties of the case $\rho = 0.06$.

\section{Conclusions}
In this paper we have proposed an elementary iterative technique to address a problem inspired by visual perception and related to the reconstruction of images from a fixed reduced set of values of its $SE(2)$ group wavelet transform.

A possible interpretation of this iteration with a kernel defined by the $SE(2)$ group as a neural computation in V1 comes from the modeling of the neural connectivity as a kernel operation \cite{WC72, EC80, CS15, MCSx}, especially if considered in the framework of a neural system that aims to learn group invariant representations of visual stimuli \cite{PA14, Anselmi20}. A direct comparison of the proposed technique with kernel techniques recently introduced with radically different purposes in \cite{MBCSx, MCSx} shows however two main differences at the level of the kernel that is used: here we need the dual wavelet to build the projection kernel, and the iteration kernel effectively contains the feature maps. On the other hand, a possible application is the inclusion of a solvability condition such as \eqref{eq:uniqueness} as iterative steps within learning frameworks such as those of \cite{AERP19, Anselmi20}.

We would like to observe also that, since the proof of convergence of this technique is general, it could be applied to other problems with a similar structure. The computational cost essentially relies on the availability of efficient methods to implement the two projections that define the problem in the discrete setting, as it happens to be the case for the setting studied in this paper. In particular, similar arguments could be applied to other wavelet transforms based on semidirect product groups $\R^d \rtimes G$, with $G$ a subgroup of $GL_d(\R)$ that defines what in this paper are sometimes called local features, and to sampling projections obtained for example, but not only, from other types of feature maps $\Theta : \R^d \to G$. From the dimensional point of view, in the discrete setting the selection of local features with a feature map can be seen as a downsampling that allows one to mantain in the transformed space the same dimension of the input vector. This is often a desirable property and it is already commonly realized e.g. by the MRA decomposition algorith of classical wavelets or by the pooling operation in neural networks. Moreover, its apparent stability of convergence seem to suggest that this operation can be performed a priori, without the need of a previous study of solvability of the problem.

Several questions remain open after this study. Probably the most fundamental one is the characterization of those maps $\Theta$ that, for a given mother wavelet $\psi$, satisfy the solvability condition \eqref{eq:uniqueness}. In terms of the study of the convergence of the project and replace iteration, moreover, it is plausible that one could obtain convergence under weaker conditions than \eqref{eq:condition}, even if maybe to a different solution, as it appears to happen in some of the numerical simulations presented.

\newpage

\section*{Figures}

\begin{figure}[H]
\begin{center}
\includegraphics[width=.24\textwidth]{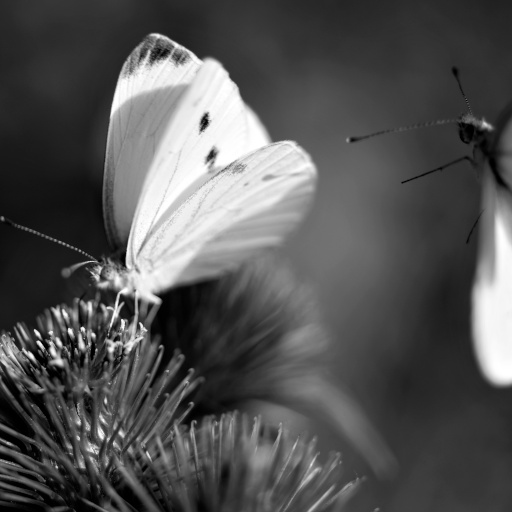}
\includegraphics[width=.24\textwidth]{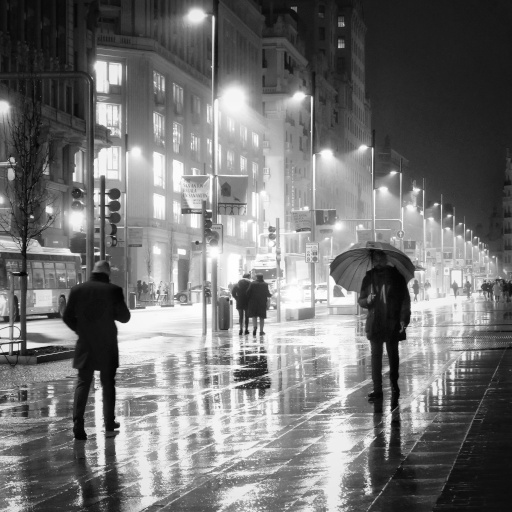}
\includegraphics[width=.24\textwidth]{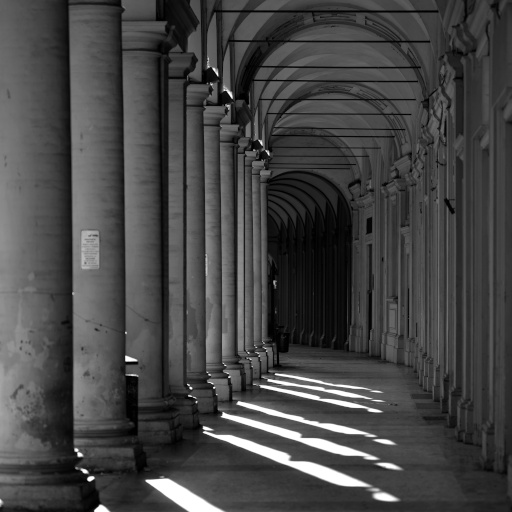}
\includegraphics[width=.24\textwidth]{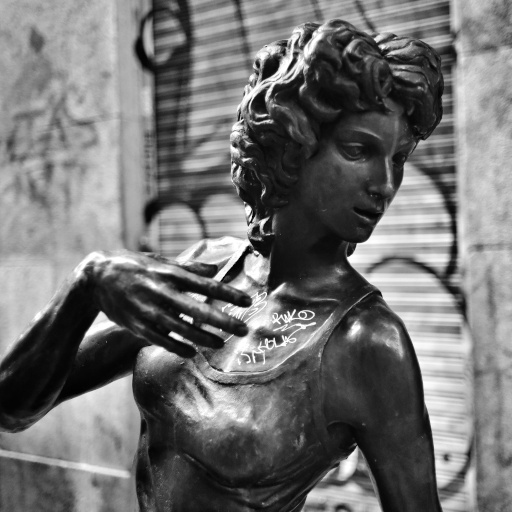}\vspace{4pt}\\
\includegraphics[width=.24\textwidth]{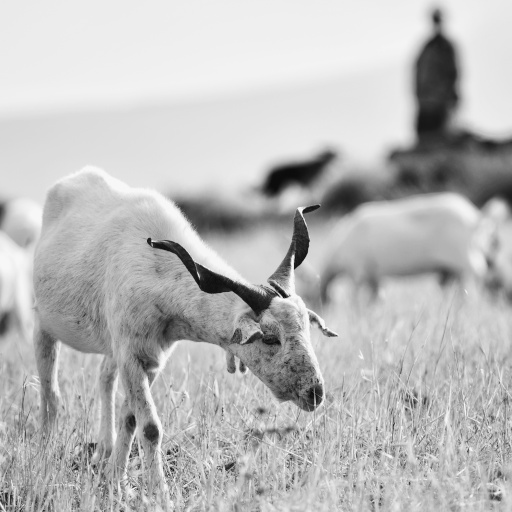}
\includegraphics[width=.24\textwidth]{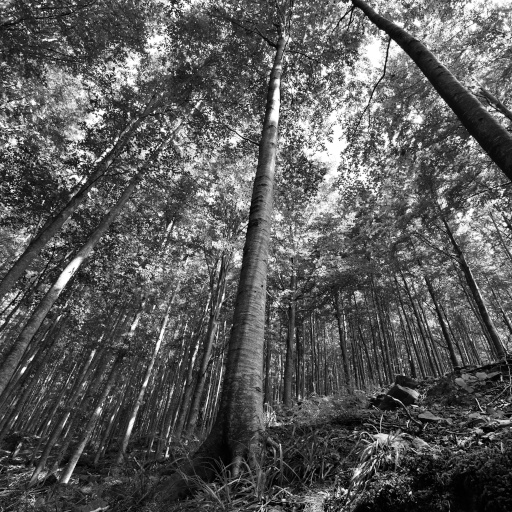}
\includegraphics[width=.24\textwidth]{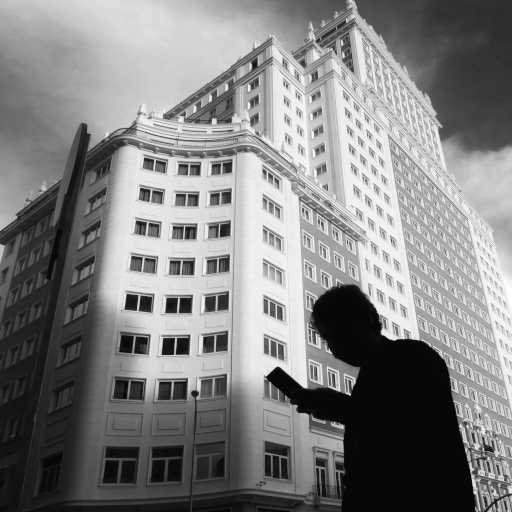}
\includegraphics[width=.24\textwidth]{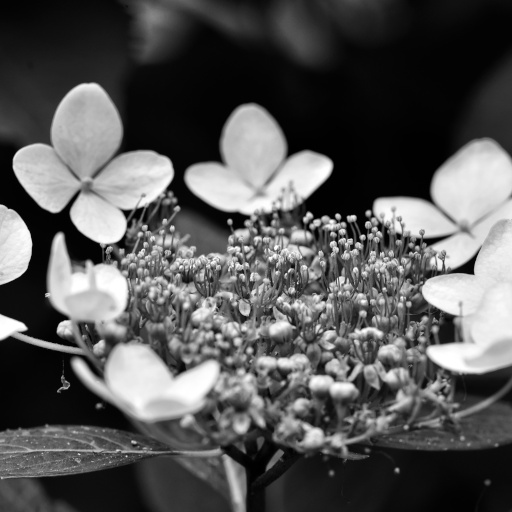}\vspace{2ex}\\
\includegraphics[width=.24\textwidth]{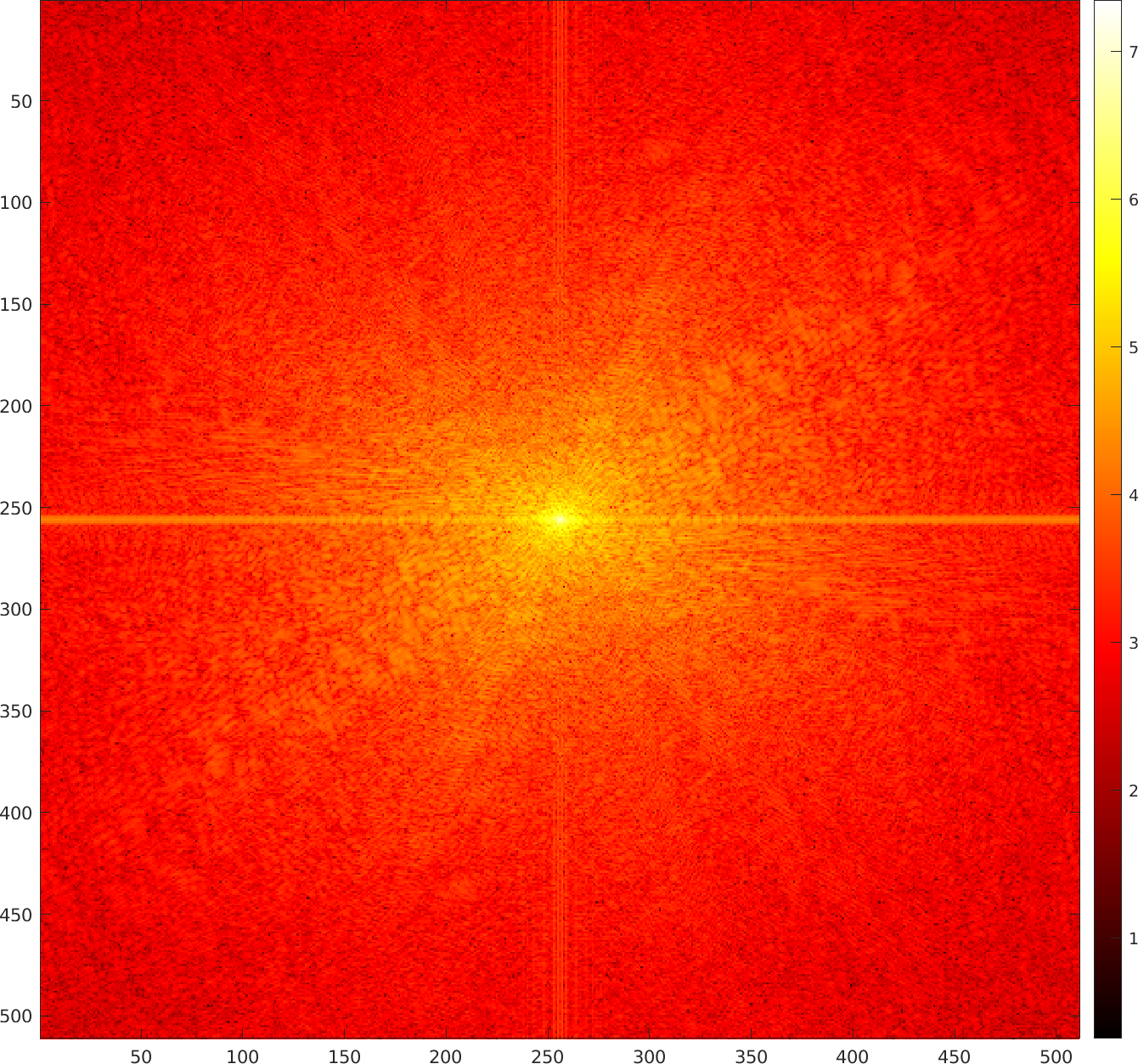}
\includegraphics[width=.24\textwidth]{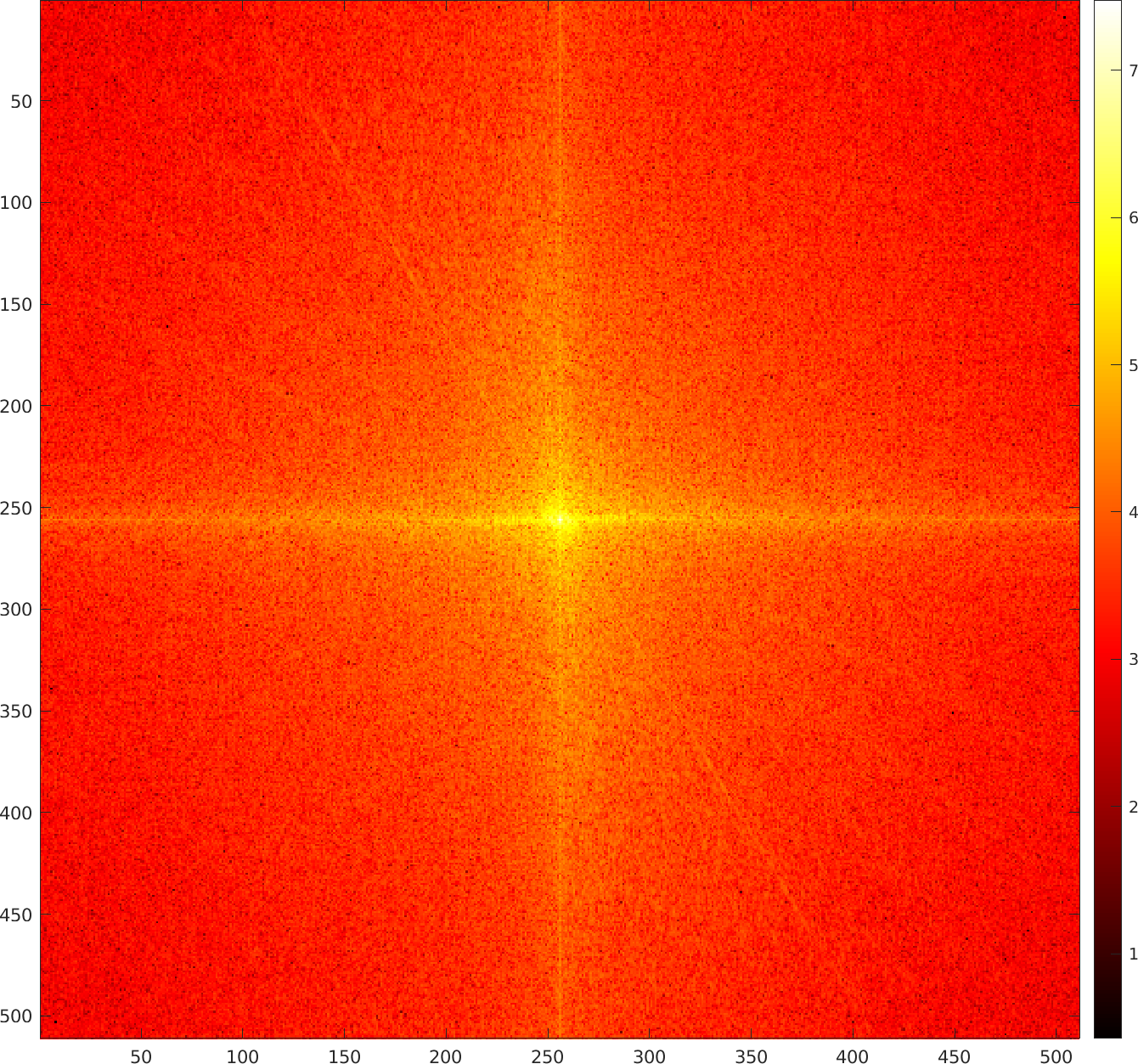}
\includegraphics[width=.24\textwidth]{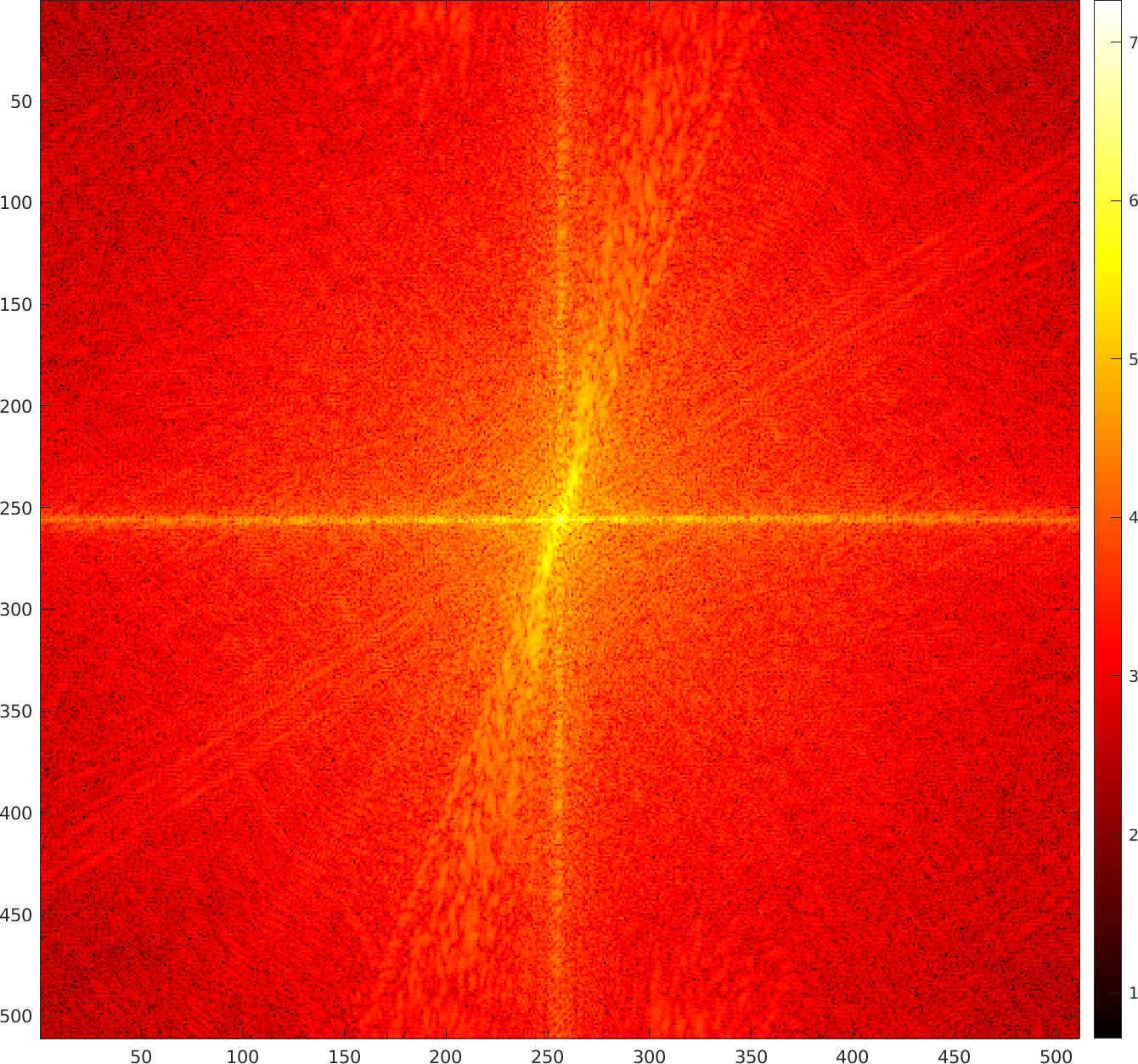}
\includegraphics[width=.24\textwidth]{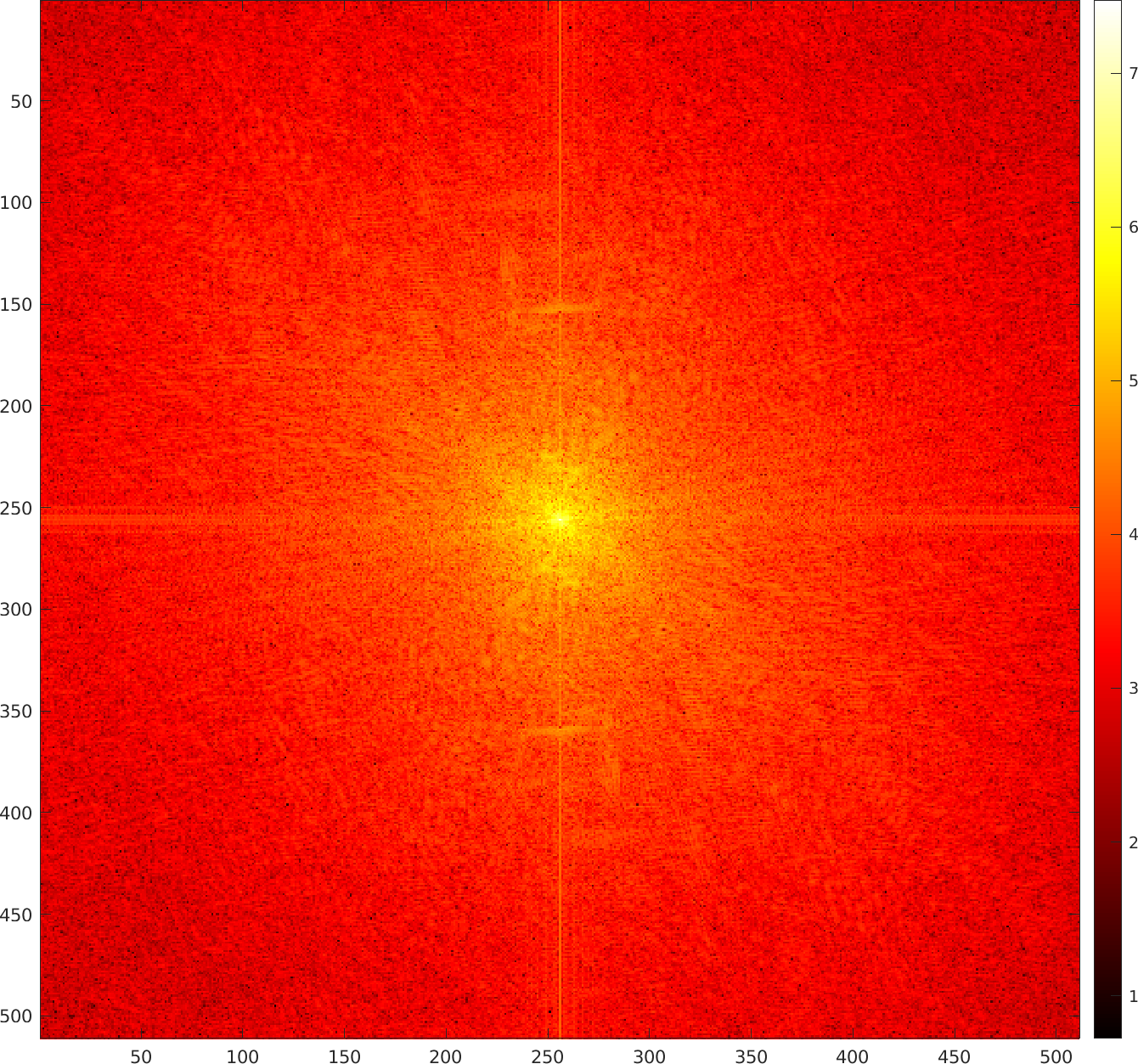}\vspace{4pt}\\
\includegraphics[width=.24\textwidth]{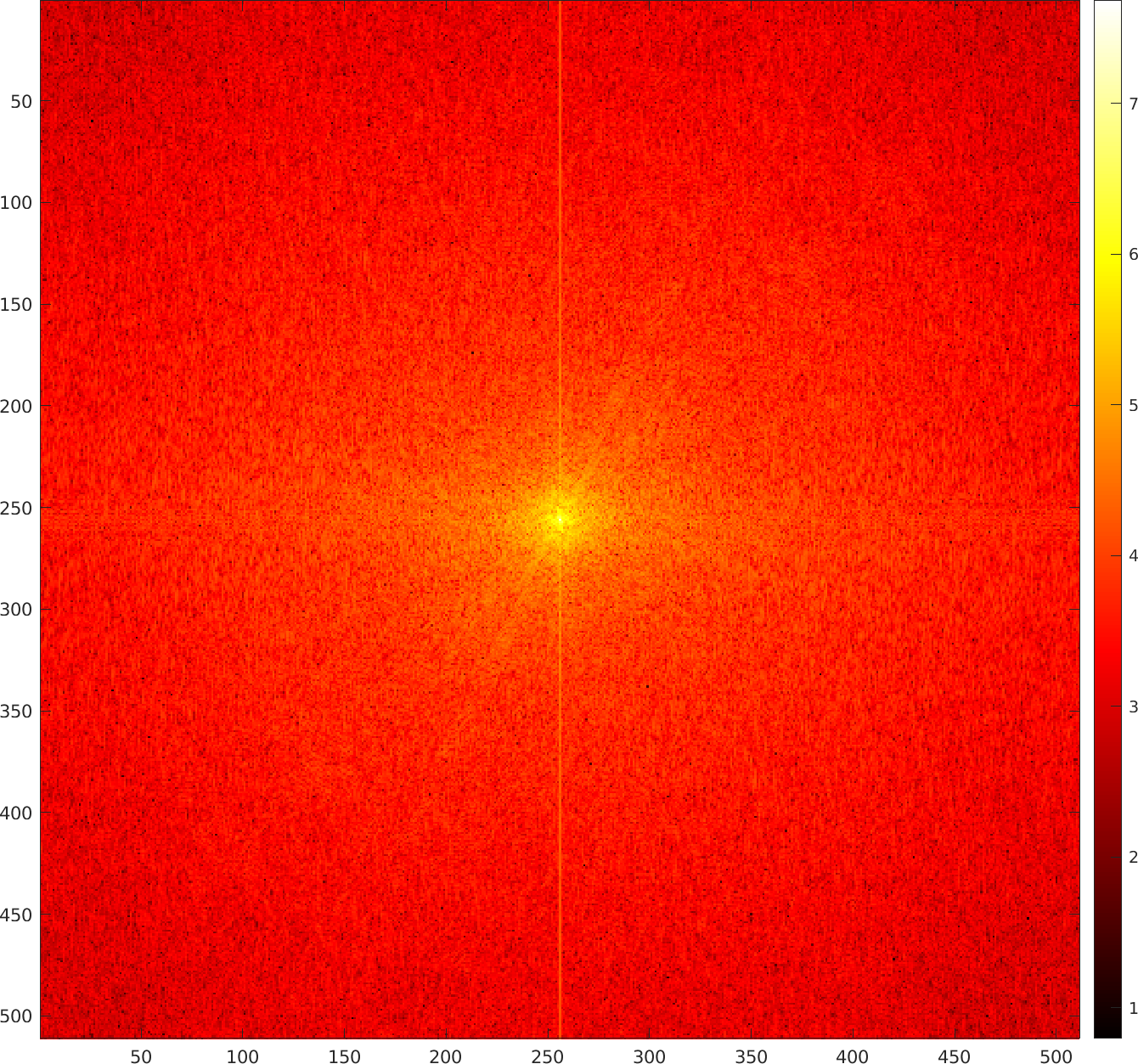}
\includegraphics[width=.24\textwidth]{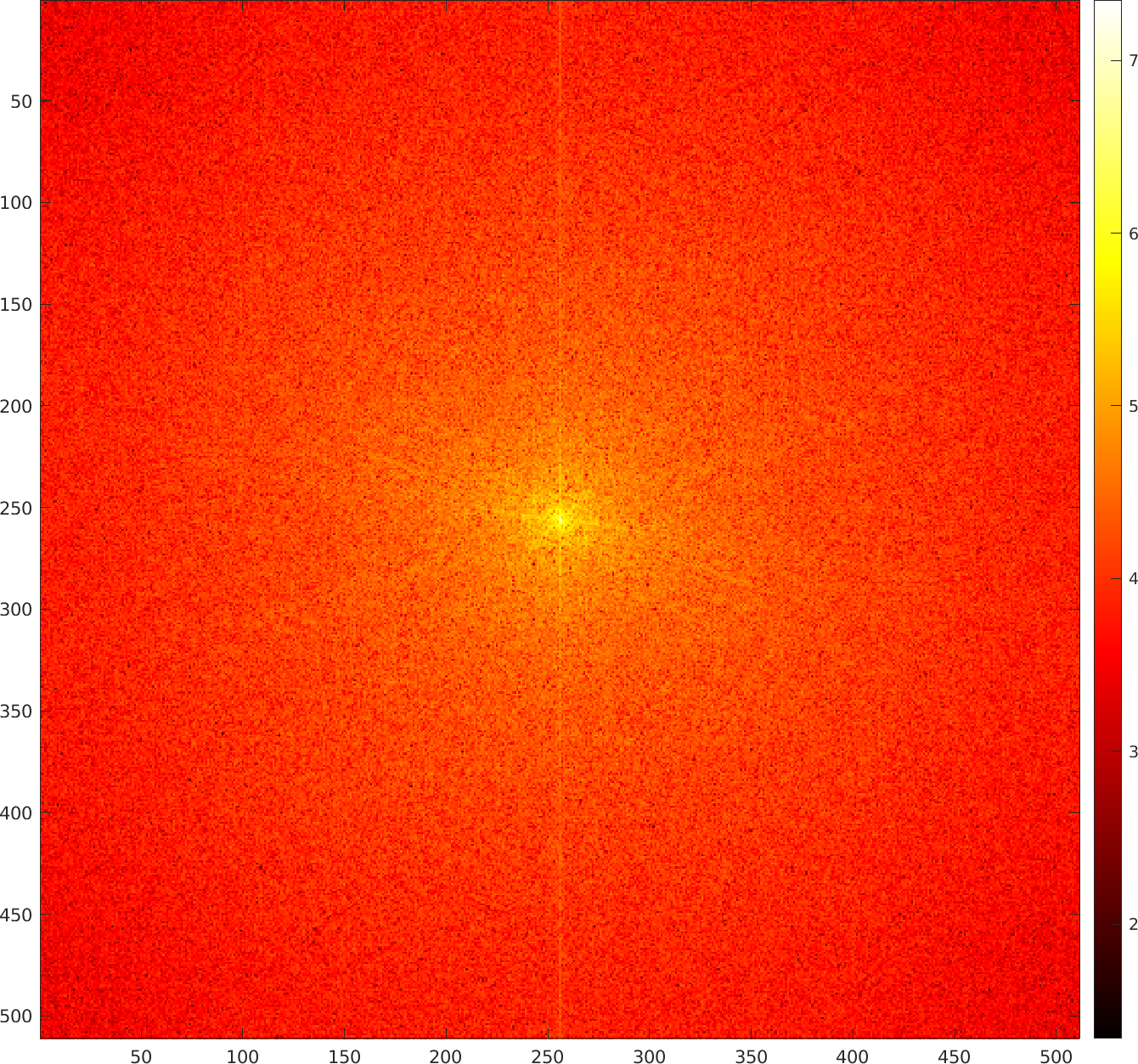}
\includegraphics[width=.24\textwidth]{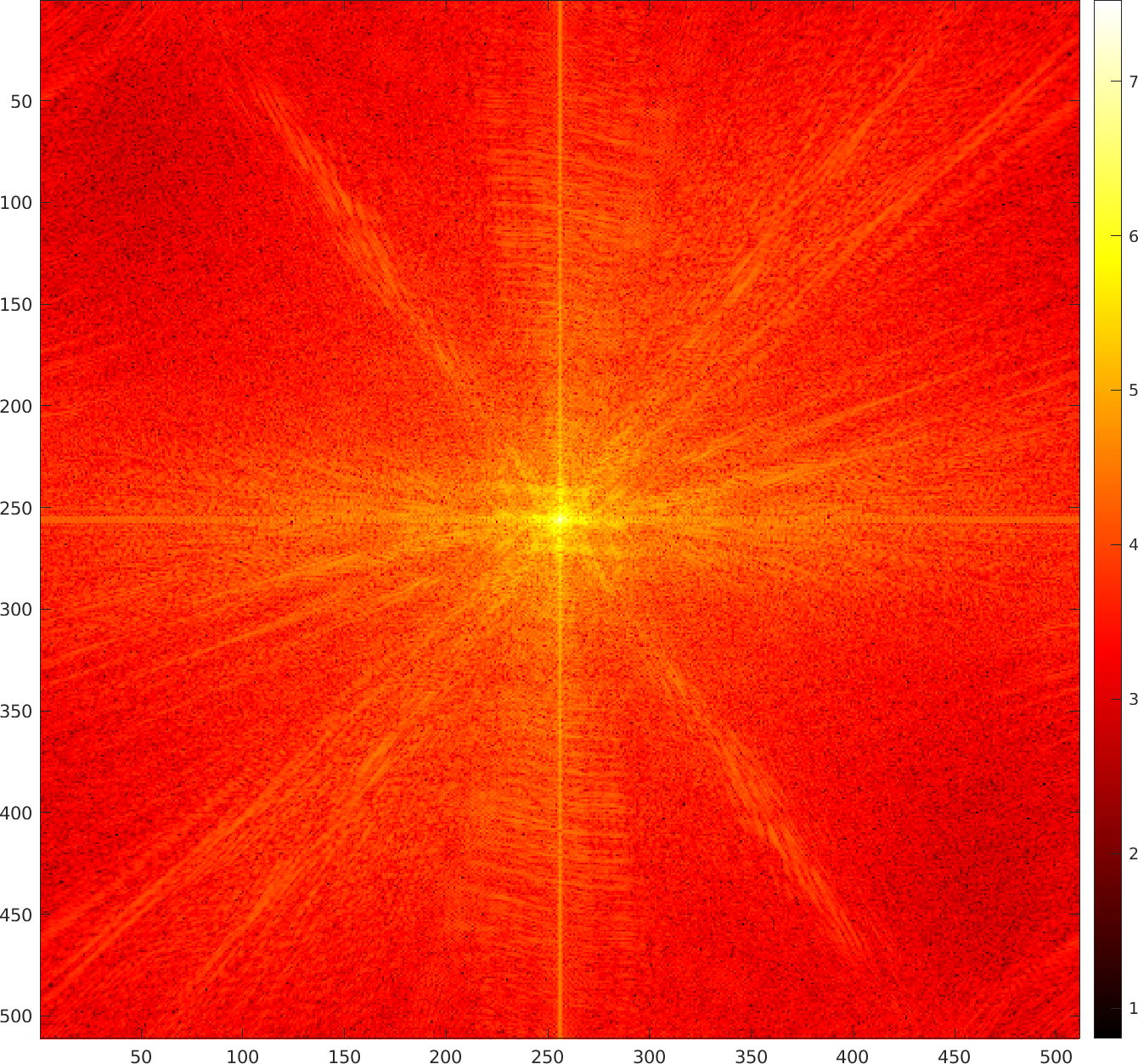}
\includegraphics[width=.24\textwidth]{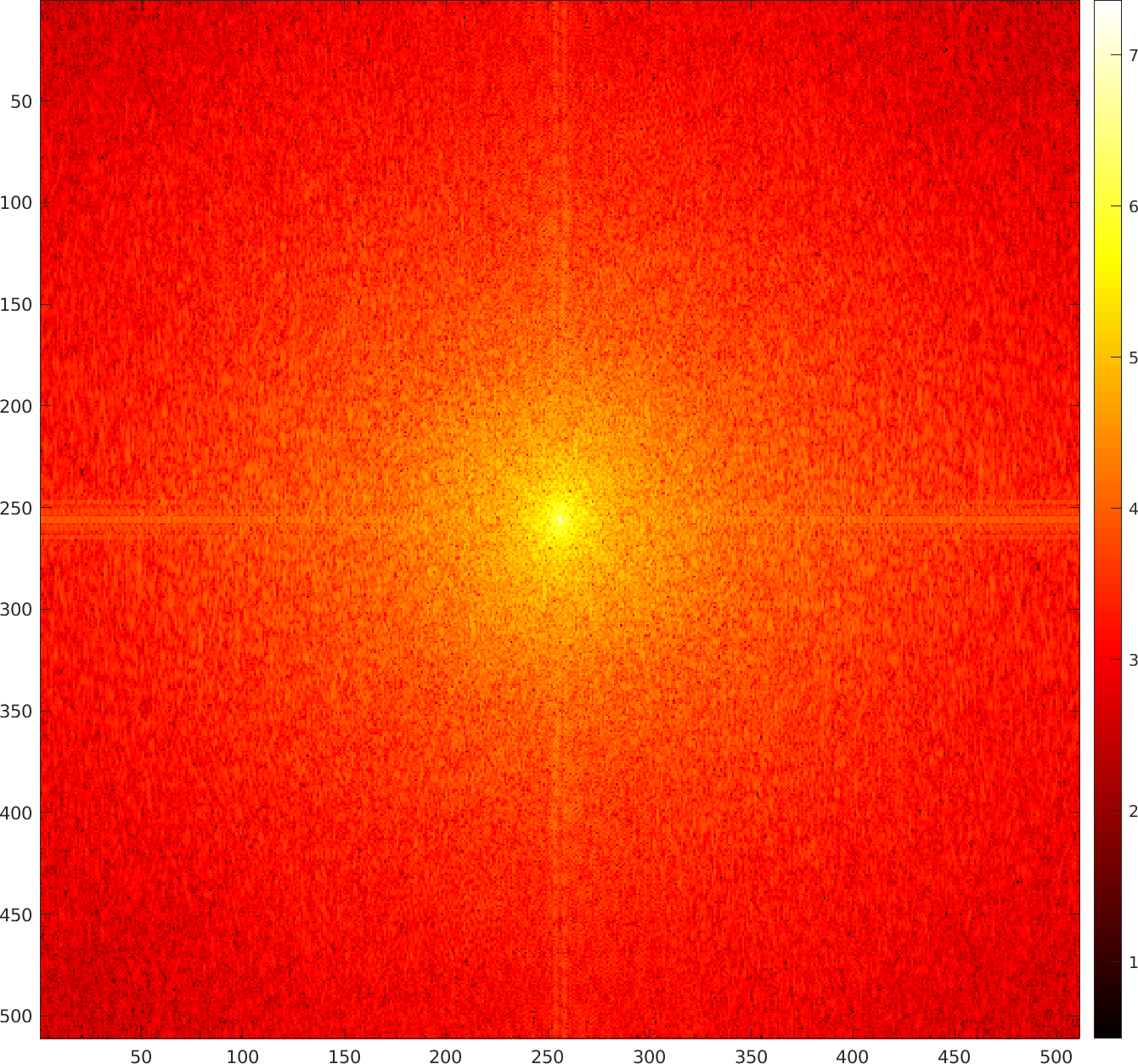}
\end{center}
\caption{Original test images (top) and their Fourier spectra in $\log_{10}$ scale (bottom).}\label{fig:dataset}
\end{figure}

\newpage

\begin{figure}[H]
\begin{center}
\includegraphics[height=.9\textheight]{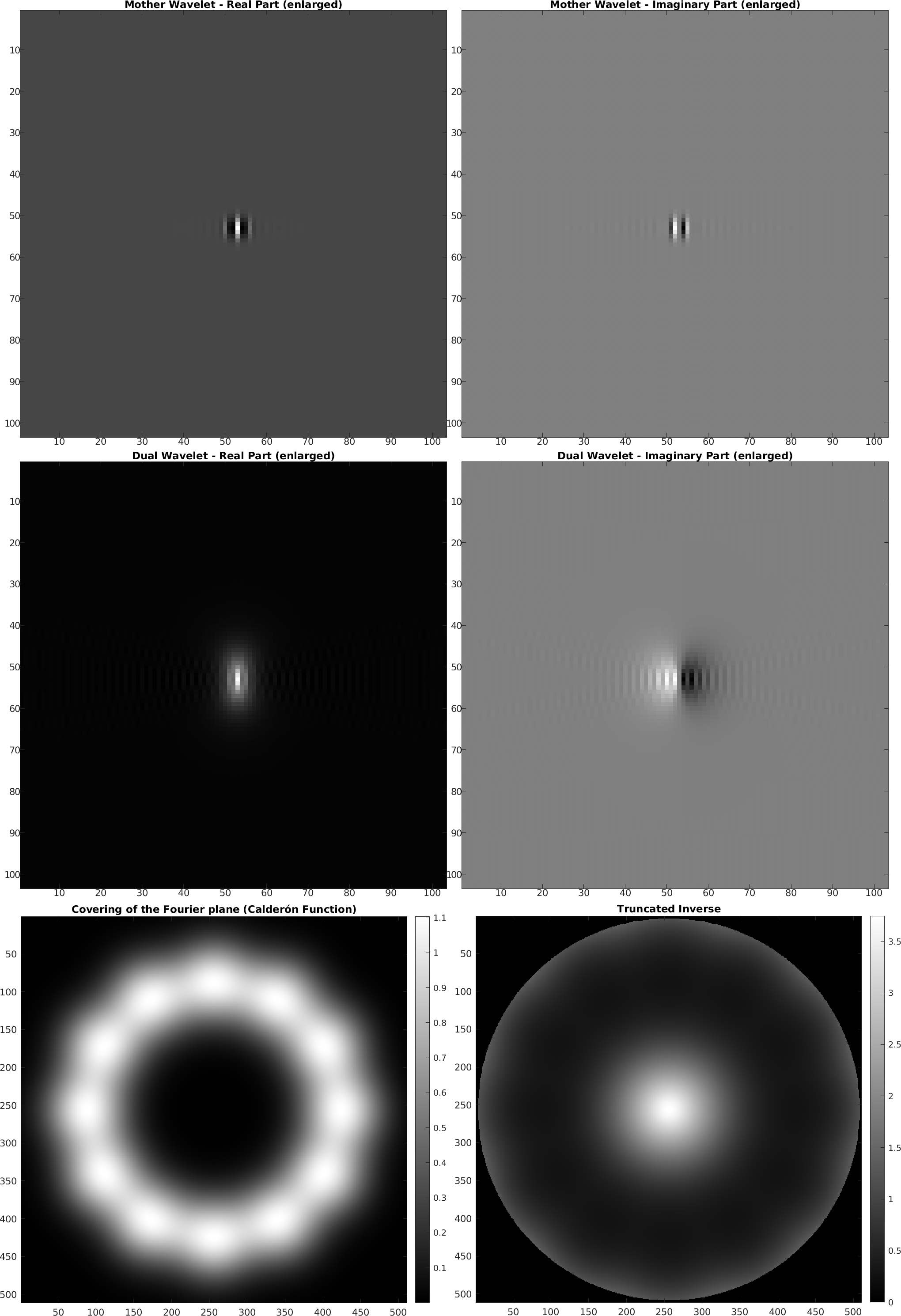}
\end{center}
\caption{Top: mother wavelet $\psi$. Center: dual wavelet $\gamma$. Bottom, left: Calder\'on's function. Bottom, right: inverse of the Calder\'on's function in $\log_{10}$ scale, bandlimited with $R = 252$.}\label{fig:SE2}
\end{figure}

\newpage

\begin{figure}[H]
\begin{center}
\includegraphics[width=.98\textwidth]{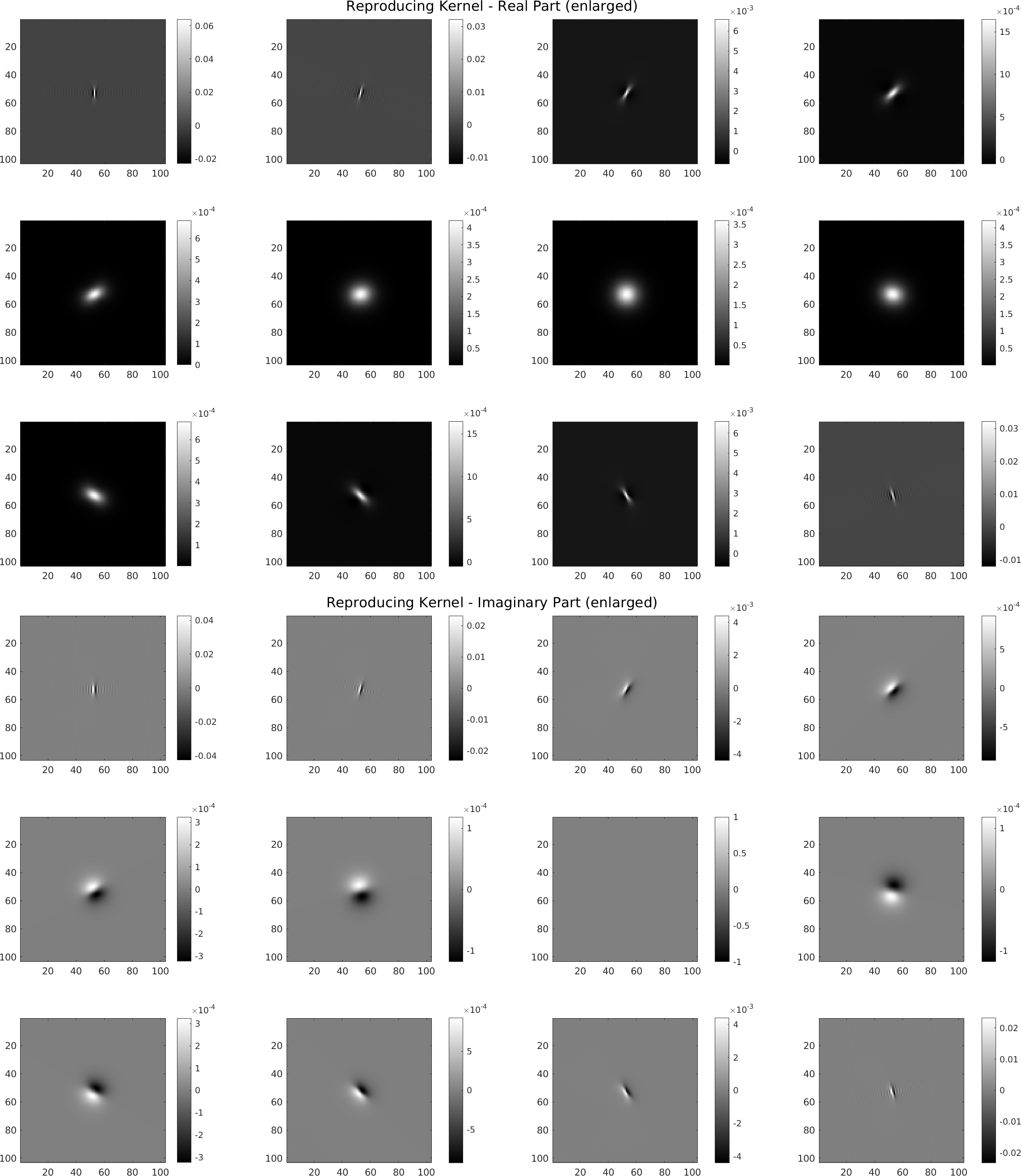}
\end{center}
\caption{Reproducing kernel for the wavelet of Figure \ref{fig:SE2}. Top: real part for the 12 angles. Bottom: imaginary part for the 12 angles.}\label{fig:RK}
\end{figure}

\newpage

\begin{figure}[H]
\begin{center}
\includegraphics[height=.9\textheight]{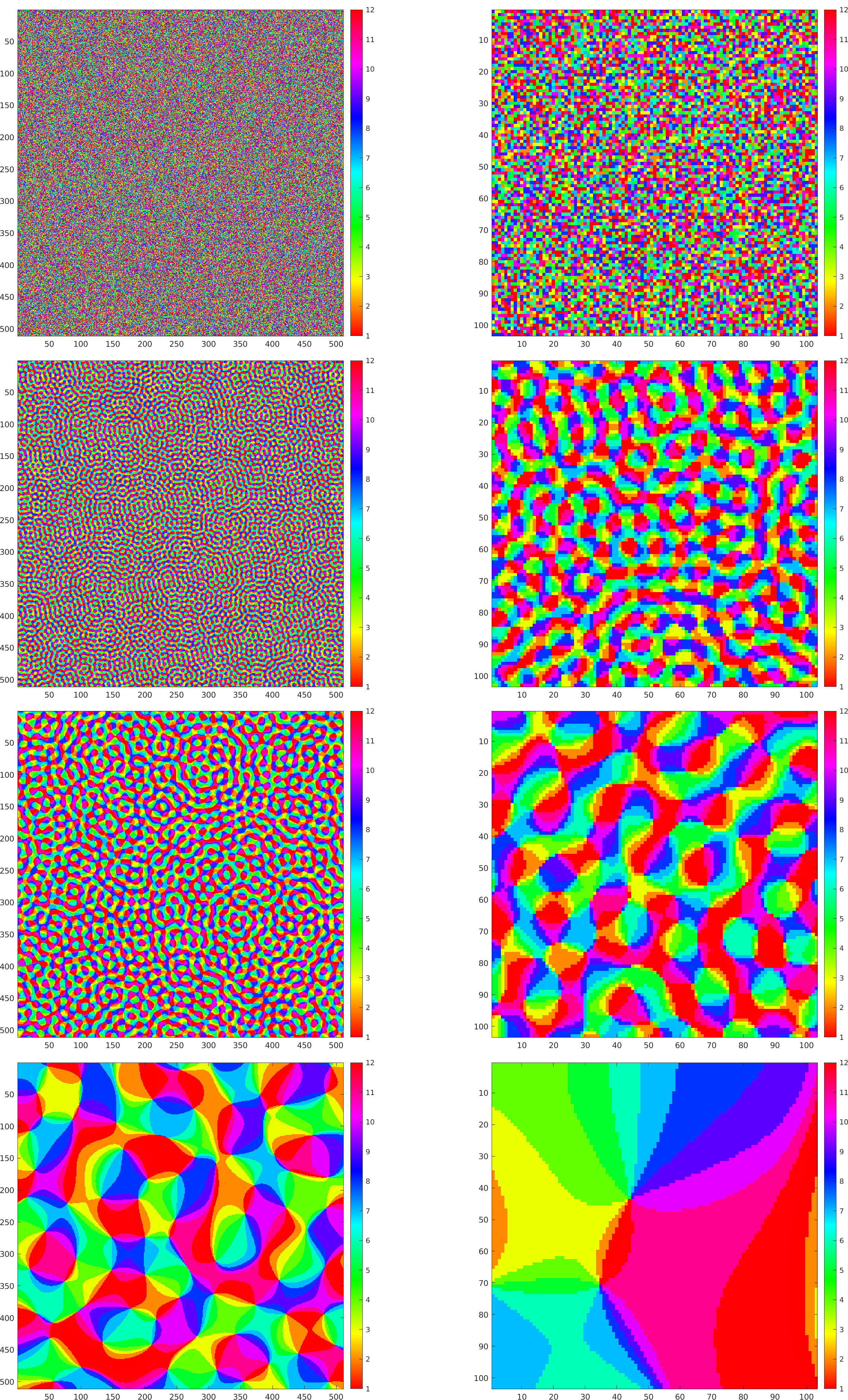}
\end{center}
\caption{Left column: maps $\Theta$ for the simulations in Figures \ref{fig:S1}, \ref{fig:S2}, \ref{fig:S3}, \ref{fig:S4}. Right column: enlargements of the same maps. First line: purely random $\Theta$. Second, third and fourth line: maps $\Theta_\rho$ generated according to \eqref{eq:generatepinwheels} with $\rho$ respectively given by $\rho = 0.8$, $\rho = 0.4$ and $\rho = 0.06$.}\label{fig:pins}
\end{figure}

\begin{figure}[H]
\begin{center}
\includegraphics[width=.325\textwidth]{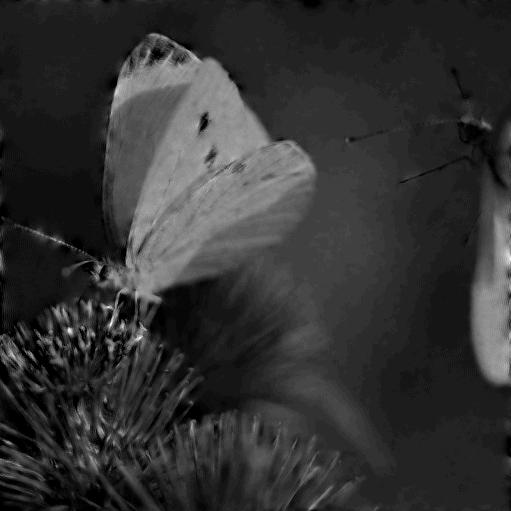}
\includegraphics[width=.325\textwidth]{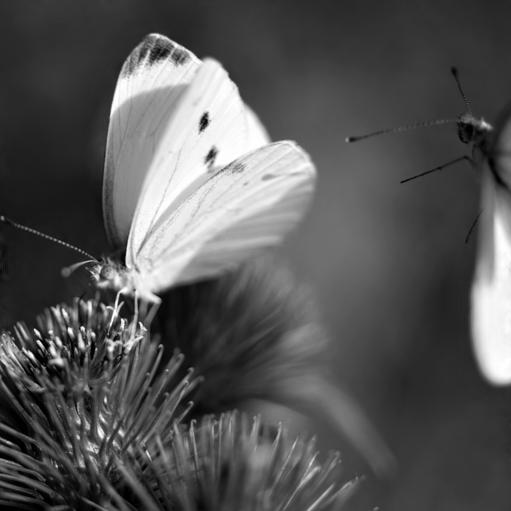}
\includegraphics[width=.33\textwidth]{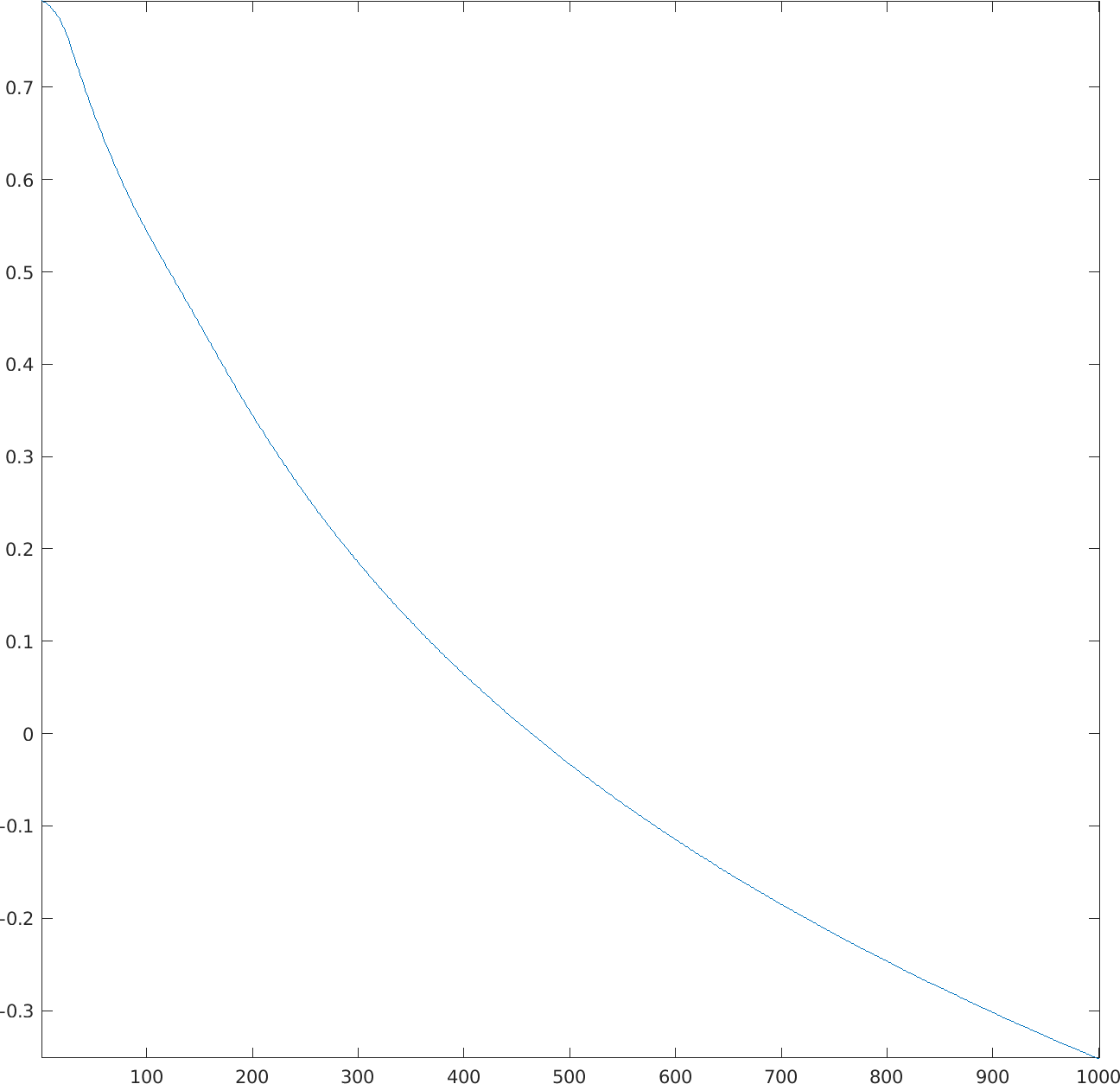}\vspace{4pt}\\
\includegraphics[width=.325\textwidth]{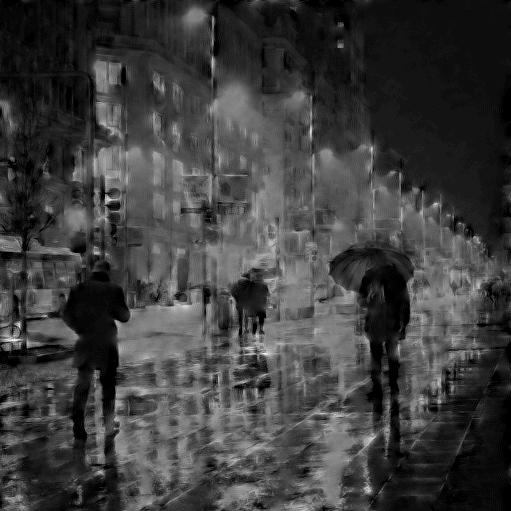}
\includegraphics[width=.325\textwidth]{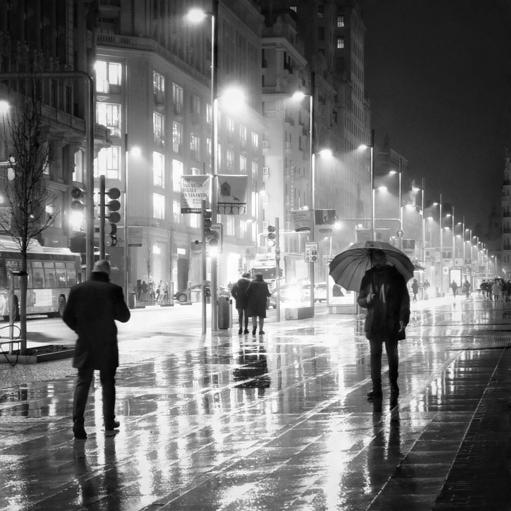}
\includegraphics[width=.33\textwidth]{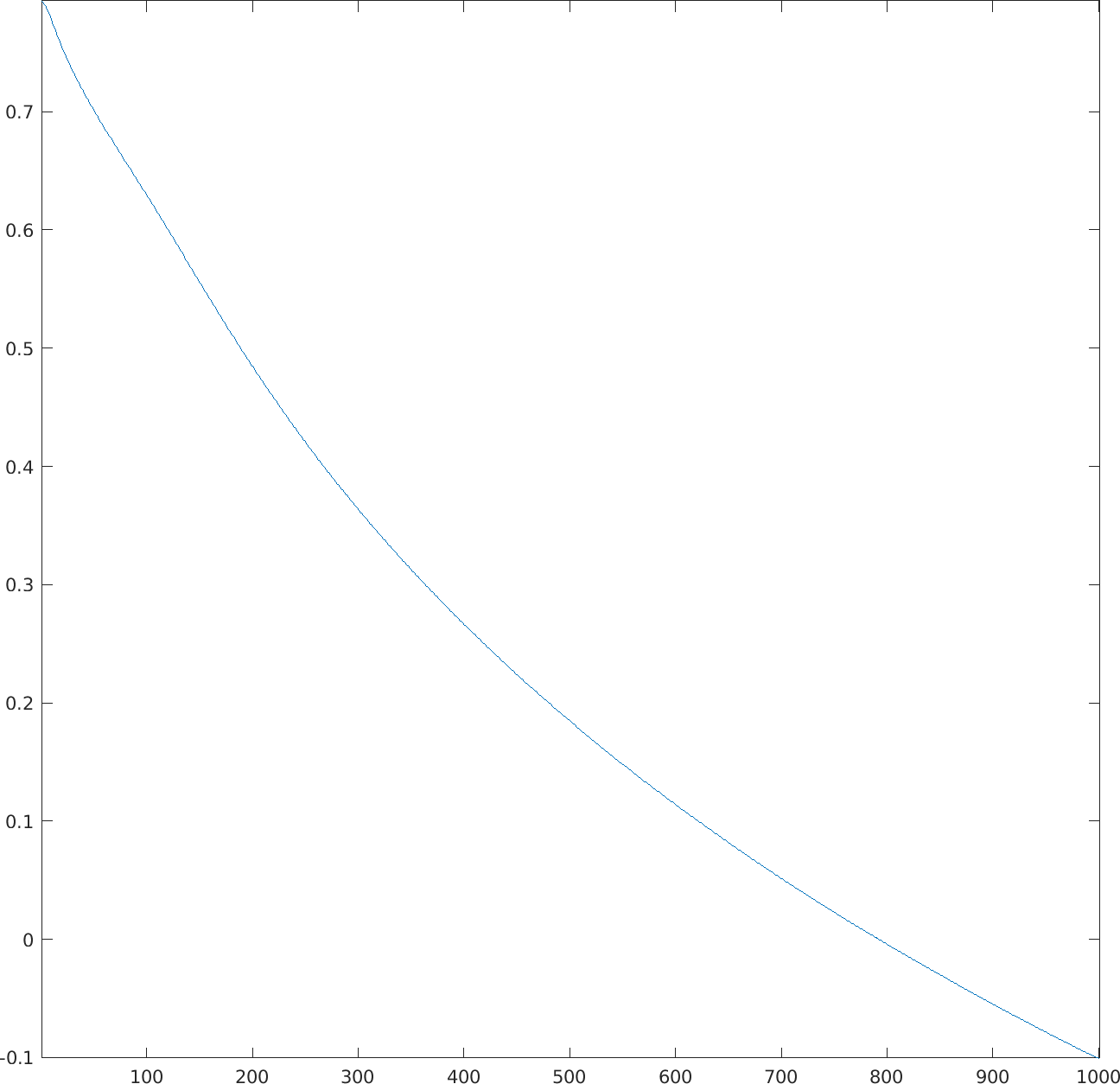}\vspace{4pt}\\
\includegraphics[width=.325\textwidth]{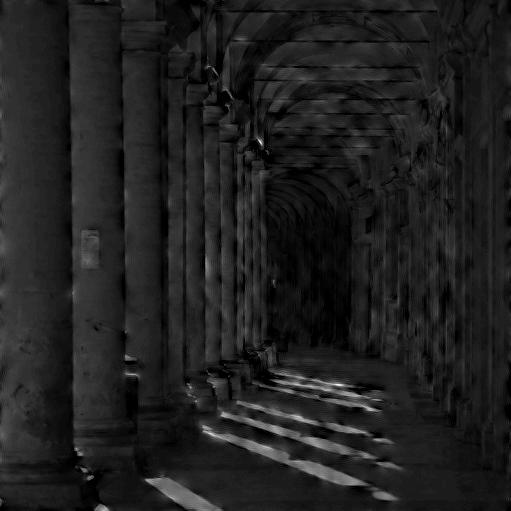}
\includegraphics[width=.325\textwidth]{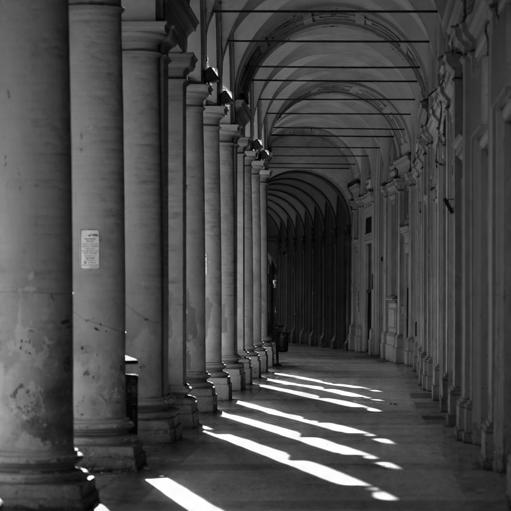}
\includegraphics[width=.33\textwidth]{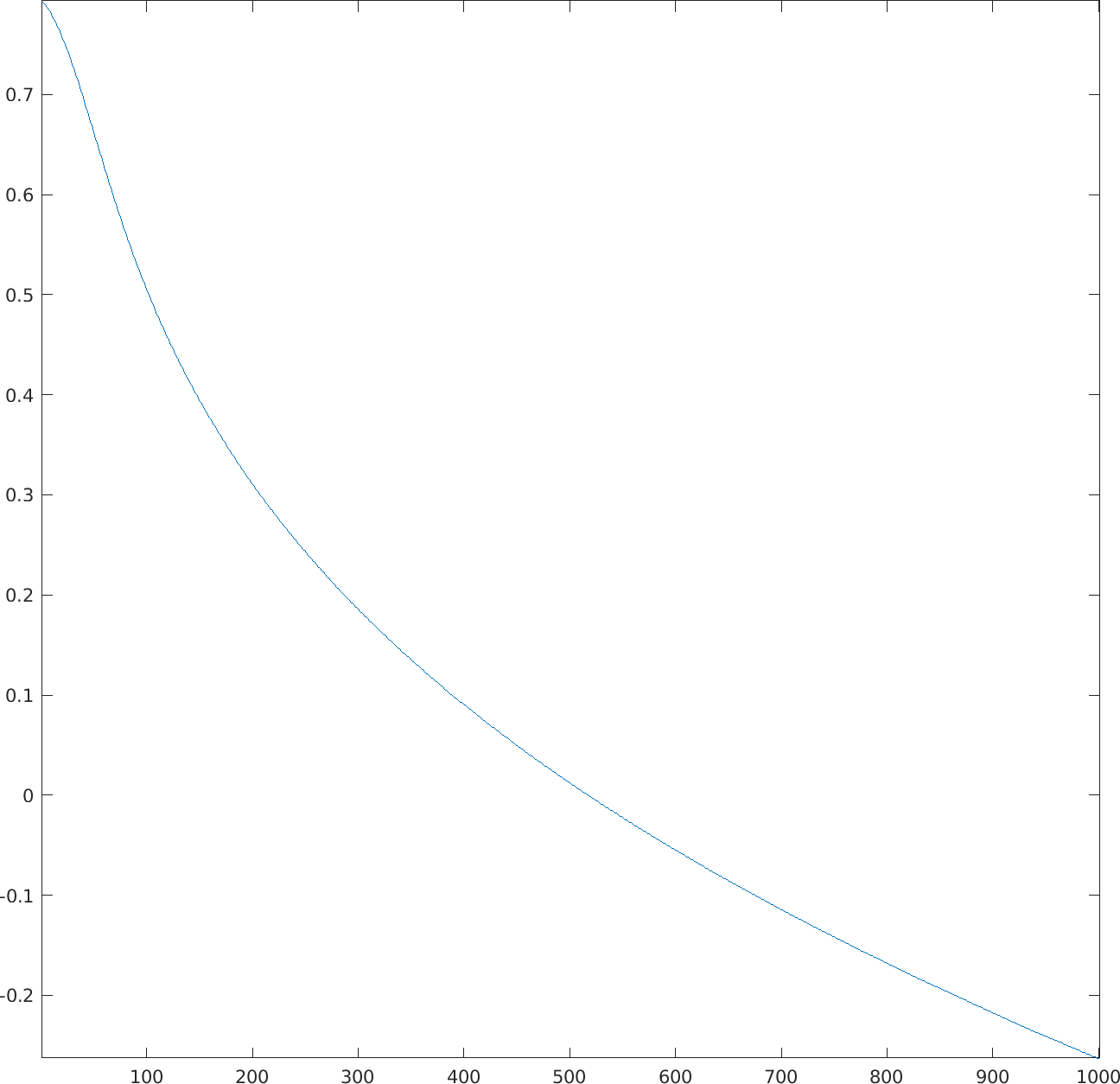}\vspace{4pt}\\
\includegraphics[width=.325\textwidth]{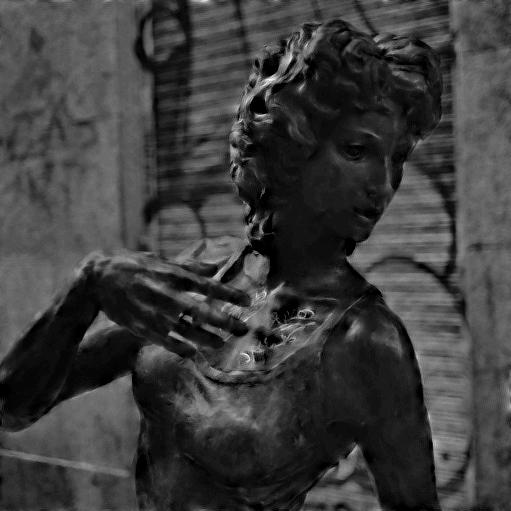}
\includegraphics[width=.325\textwidth]{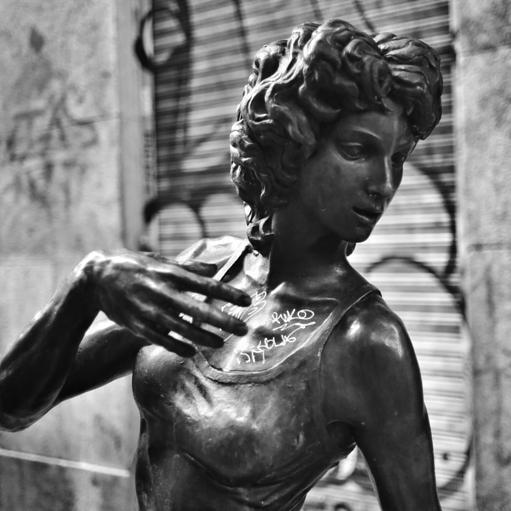}
\includegraphics[width=.33\textwidth]{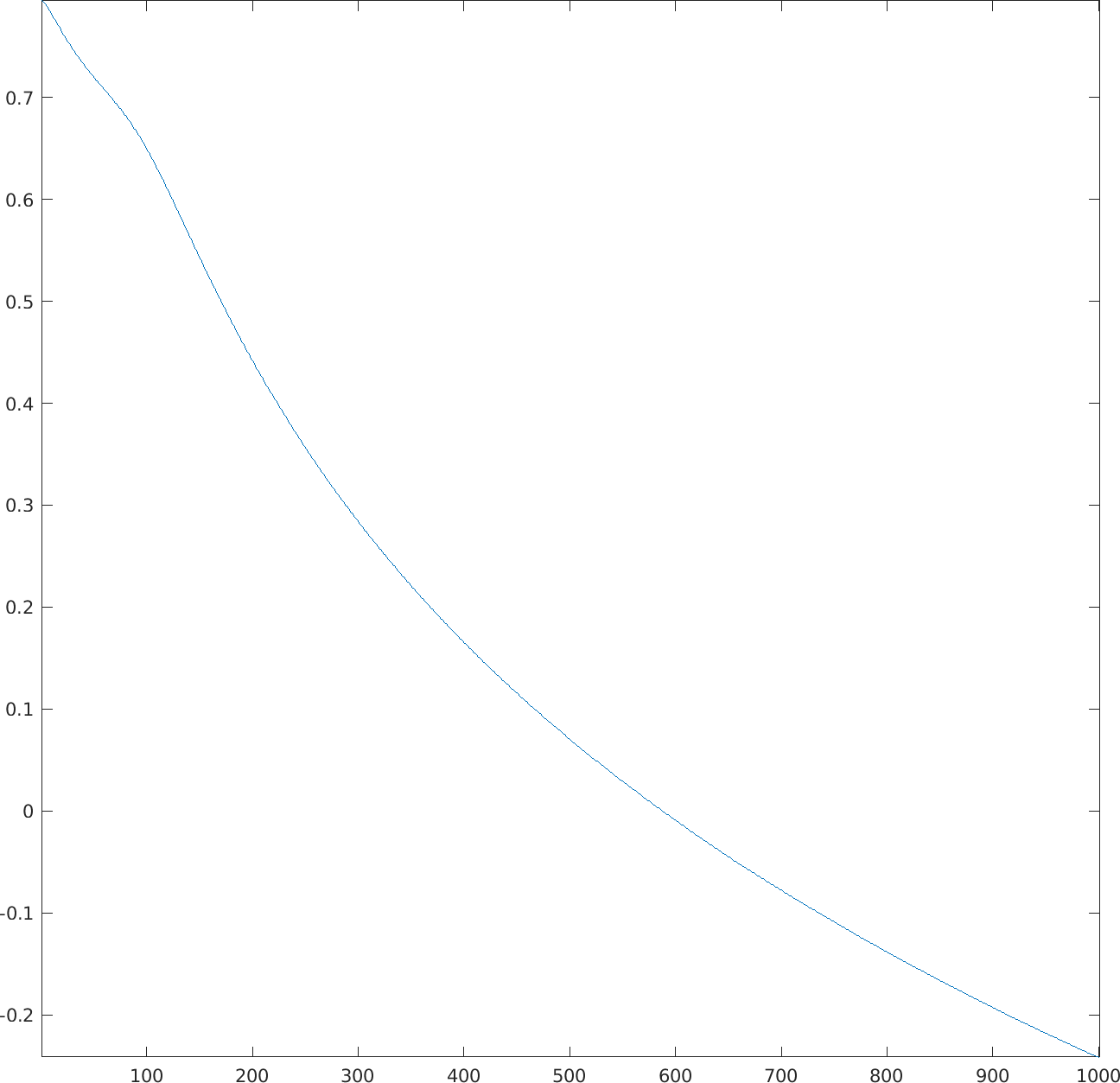}
\end{center}
\caption{Iteration on the images 1 to 4 of Figure \ref{fig:dataset} for the purely random $\Theta$ shown in the first line of Figure \ref{fig:pins}. Left: first step of the iteration. Center: after 1000 iterations. Right: $\log_{10}(\Delta_n)$, for the error \eqref{eq:DELTA}.}\label{fig:S1}
\end{figure}

\newpage

\begin{figure}[H]
\begin{center}
\includegraphics[width=.325\textwidth]{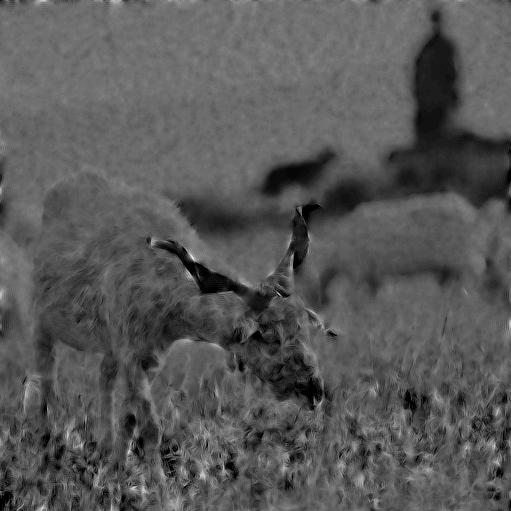}
\includegraphics[width=.325\textwidth]{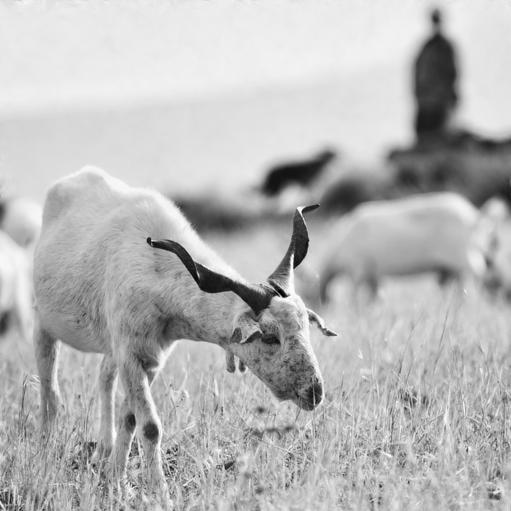}
\includegraphics[width=.33\textwidth]{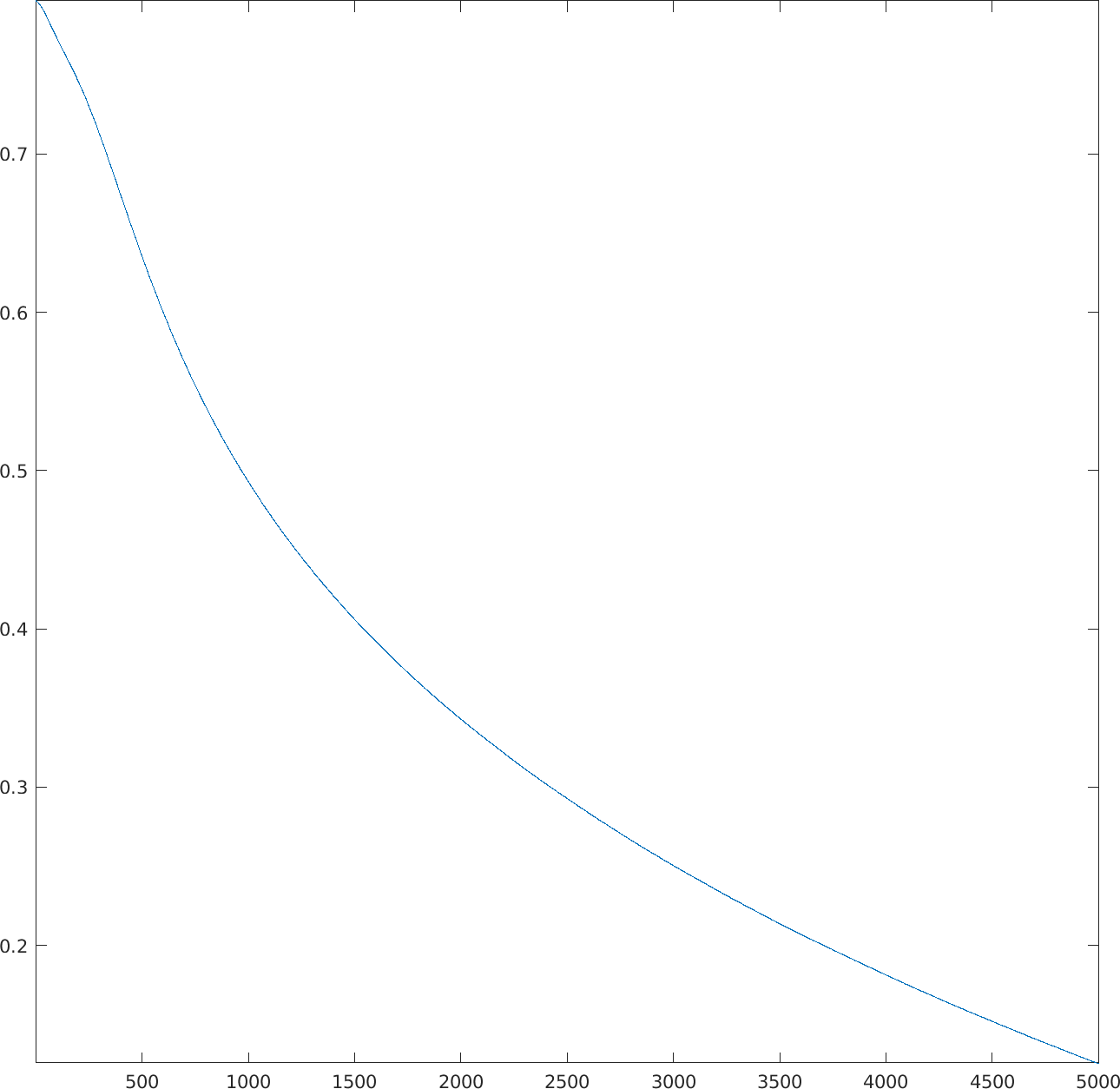}\vspace{4pt}\\
\includegraphics[width=.325\textwidth]{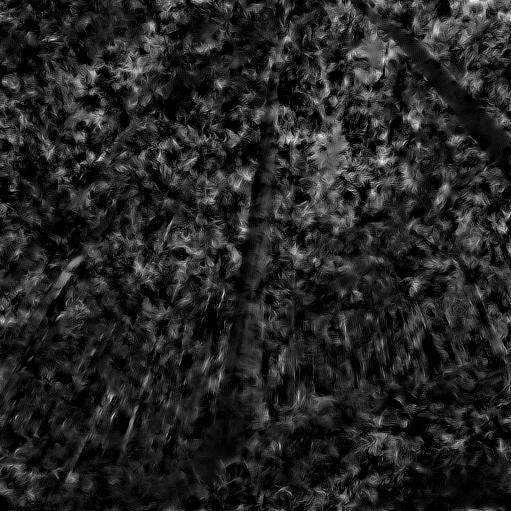}
\includegraphics[width=.325\textwidth]{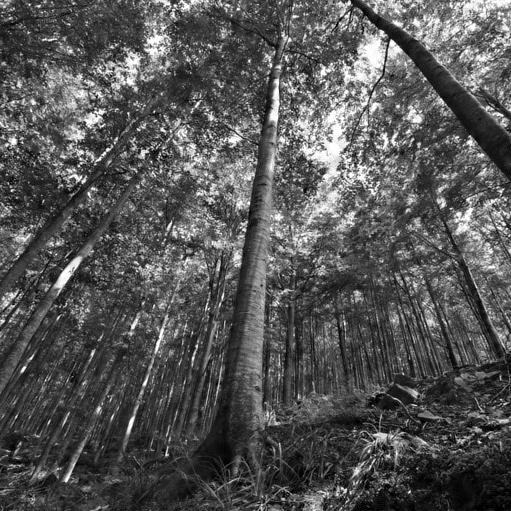}
\includegraphics[width=.33\textwidth]{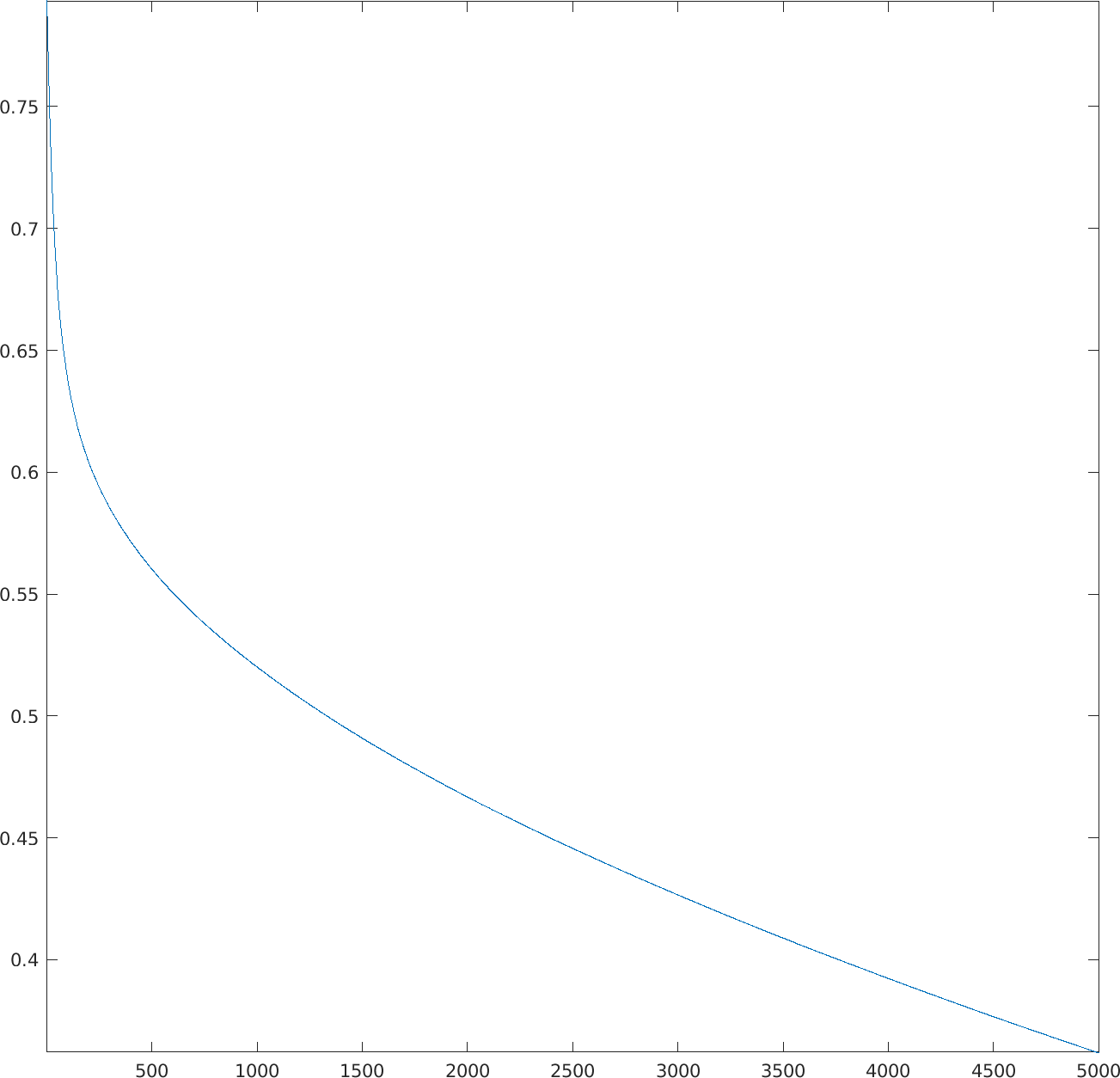}\vspace{4pt}\\
\includegraphics[width=.325\textwidth]{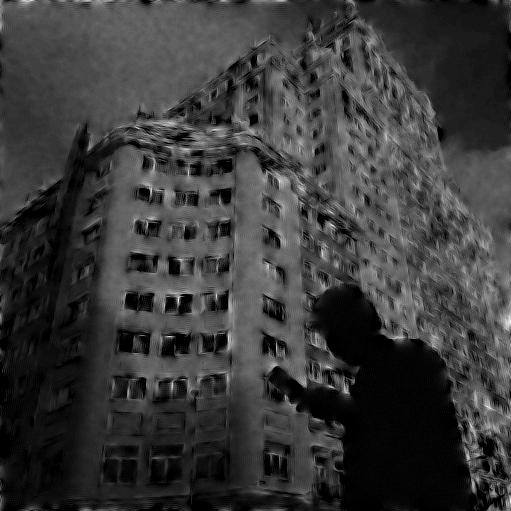}
\includegraphics[width=.325\textwidth]{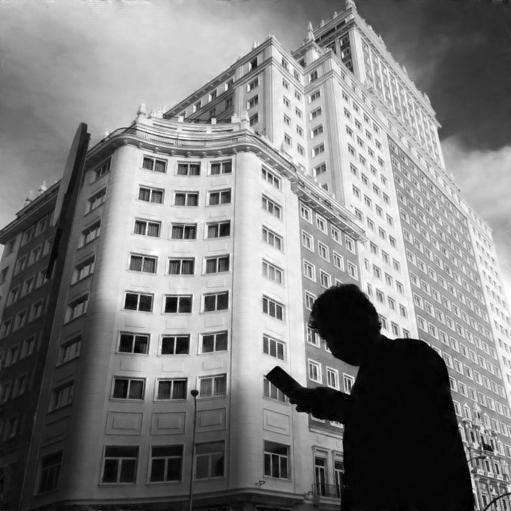}
\includegraphics[width=.33\textwidth]{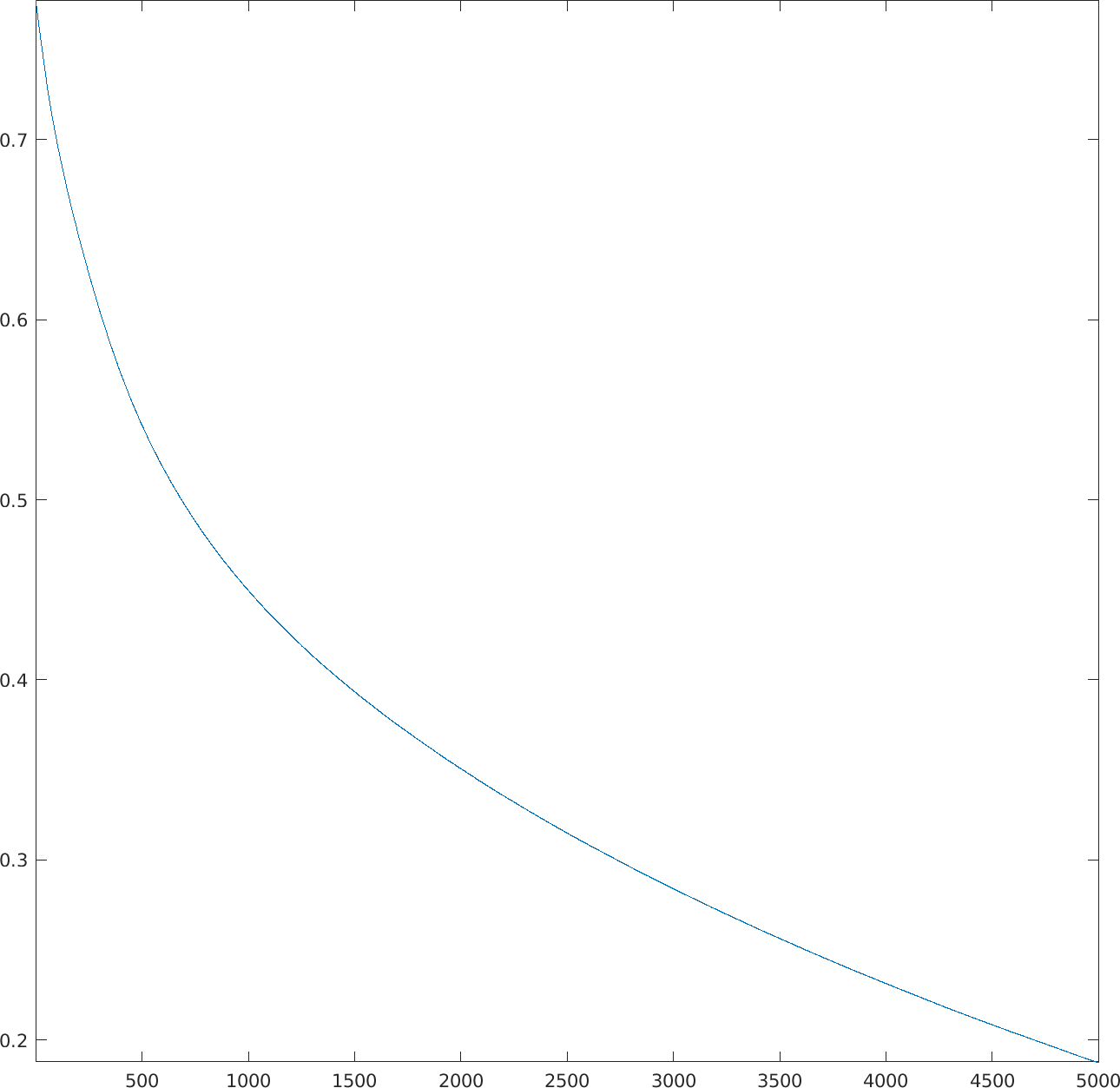}\vspace{4pt}\\
\includegraphics[width=.325\textwidth]{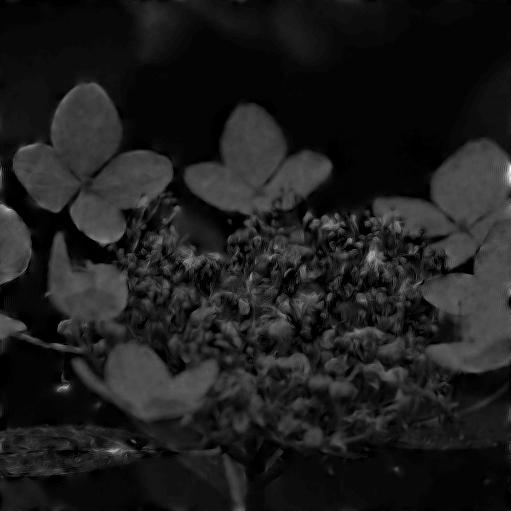}
\includegraphics[width=.325\textwidth]{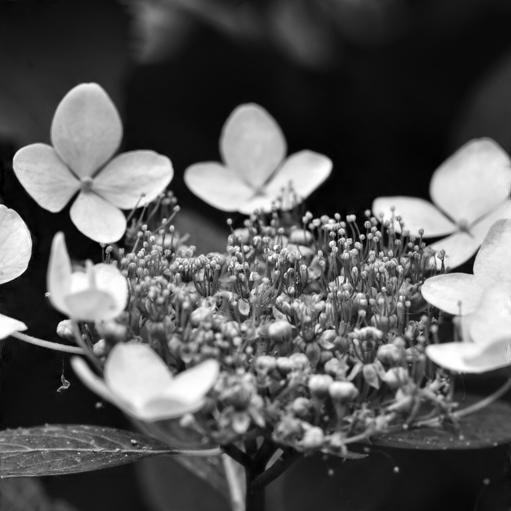}
\includegraphics[width=.33\textwidth]{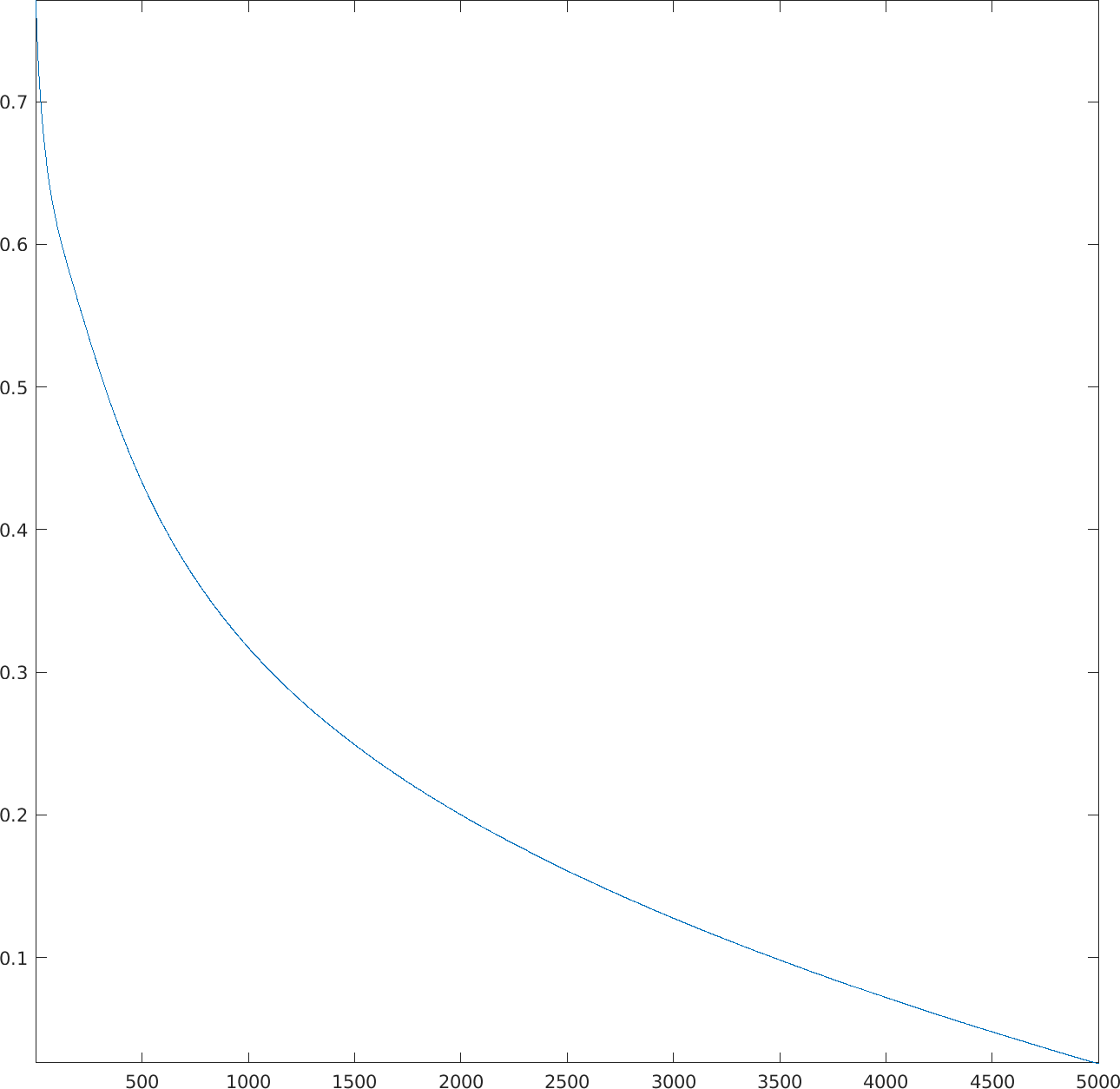}
\end{center}
\caption{Iteration on the images 5 to 8 of Figure \ref{fig:dataset} for $\Theta_{0.8}$ shown in the second line of Figure \ref{fig:pins}. Left: first step of the iteration. Center: after 5000 iterations. Right: $\log_{10}(\Delta_n)$, for the error \eqref{eq:DELTA}.}\label{fig:S2}
\end{figure}

\newpage

\begin{figure}[H]
\begin{center}
\includegraphics[width=.325\textwidth]{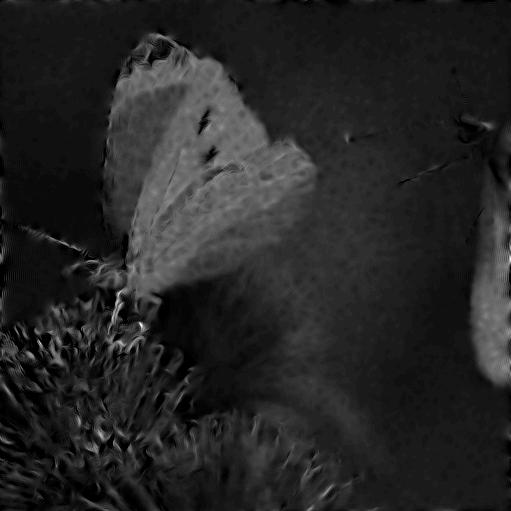}
\includegraphics[width=.325\textwidth]{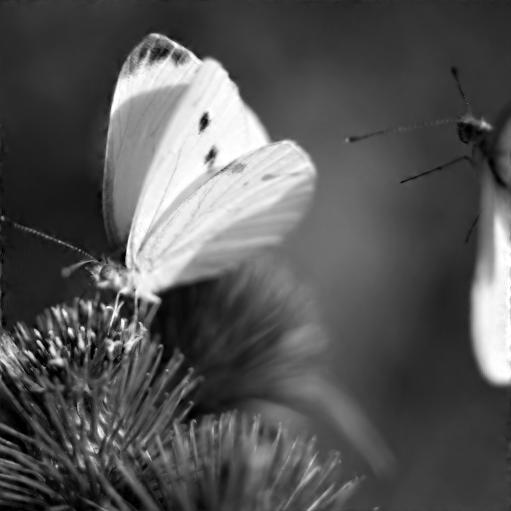}
\includegraphics[width=.33\textwidth]{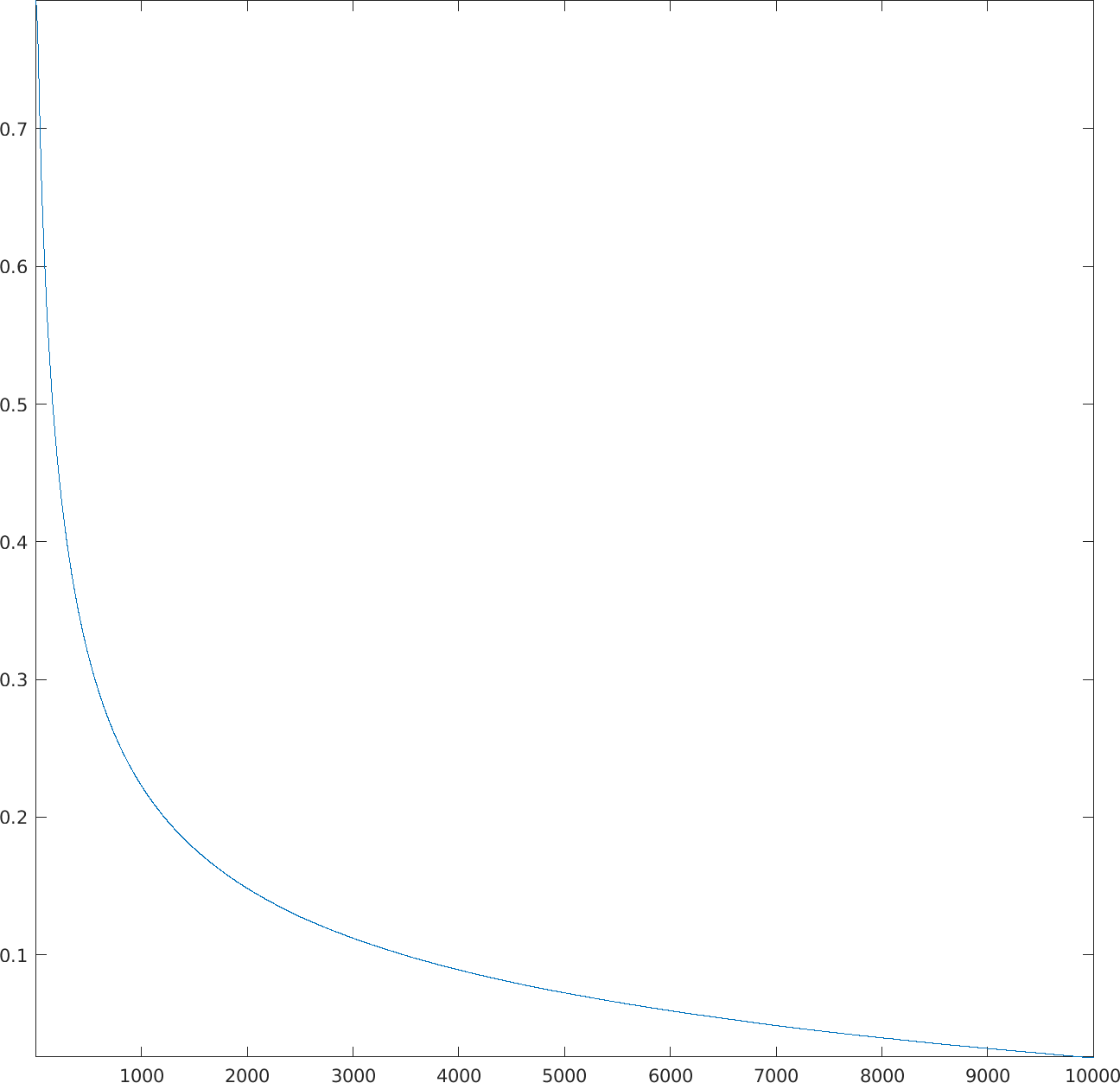}\vspace{4pt}\\
\includegraphics[width=.325\textwidth]{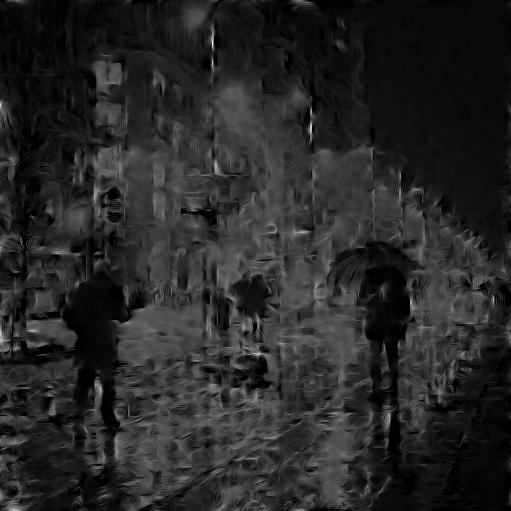}
\includegraphics[width=.325\textwidth]{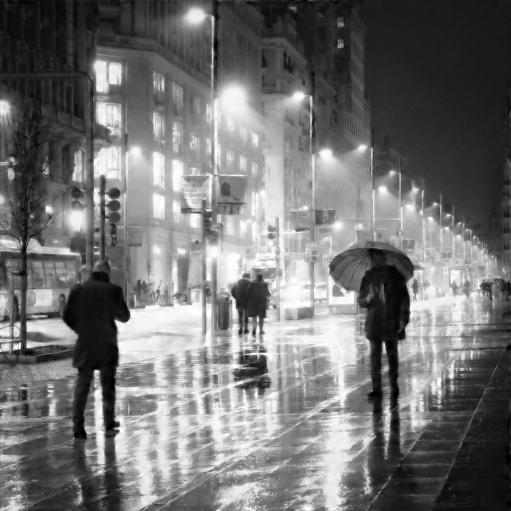}
\includegraphics[width=.33\textwidth]{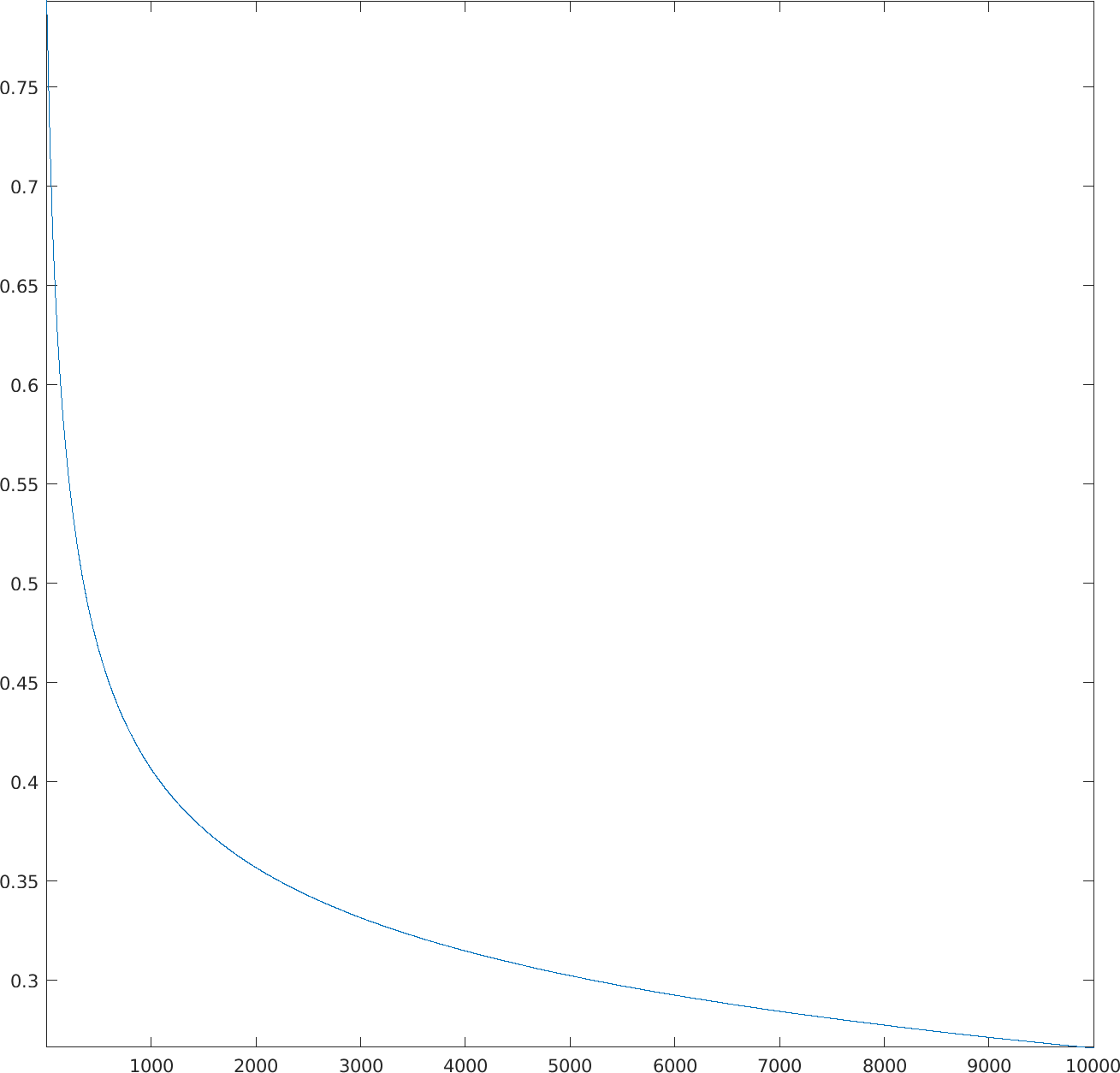}\vspace{4pt}\\
\includegraphics[width=.325\textwidth]{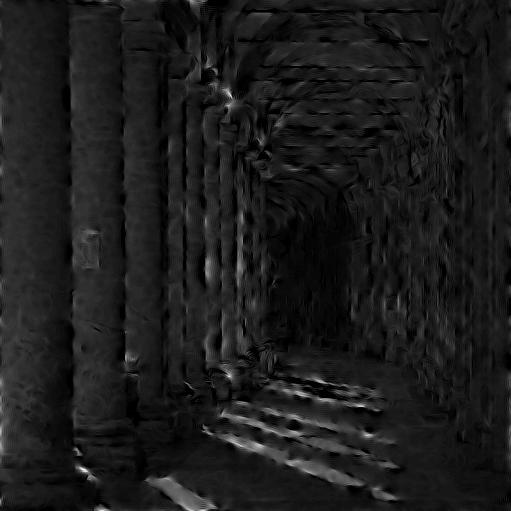}
\includegraphics[width=.325\textwidth]{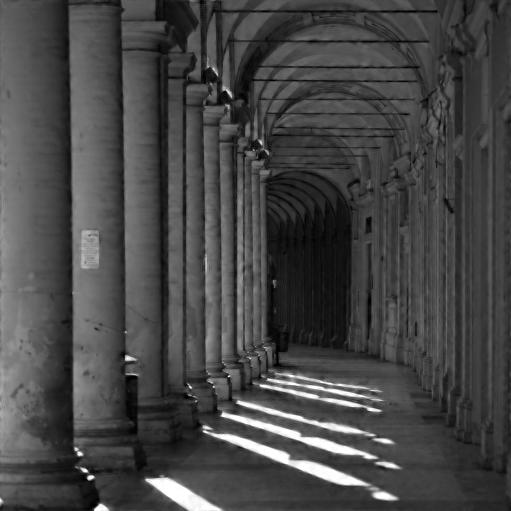}
\includegraphics[width=.33\textwidth]{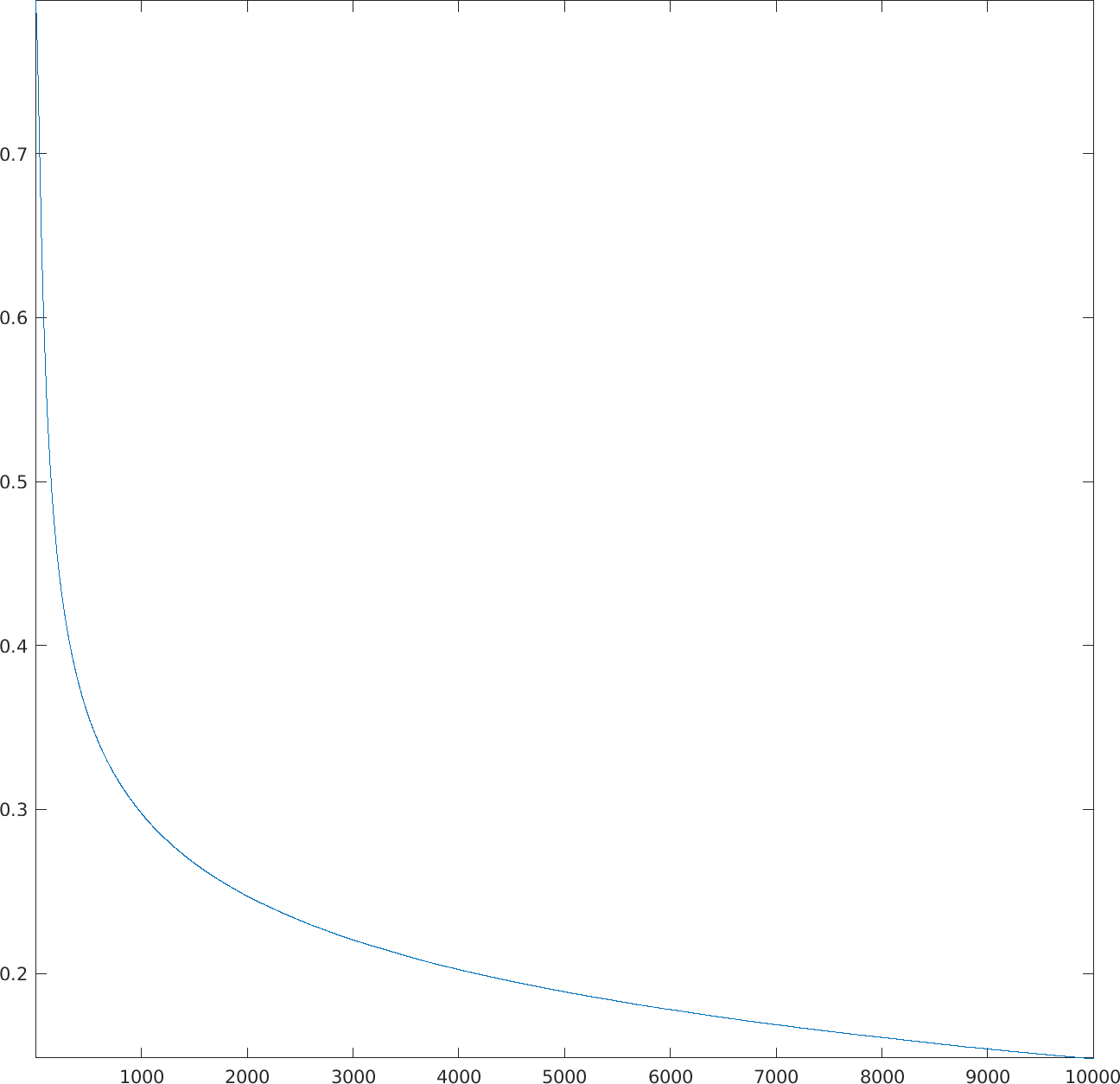}\vspace{4pt}\\
\includegraphics[width=.325\textwidth]{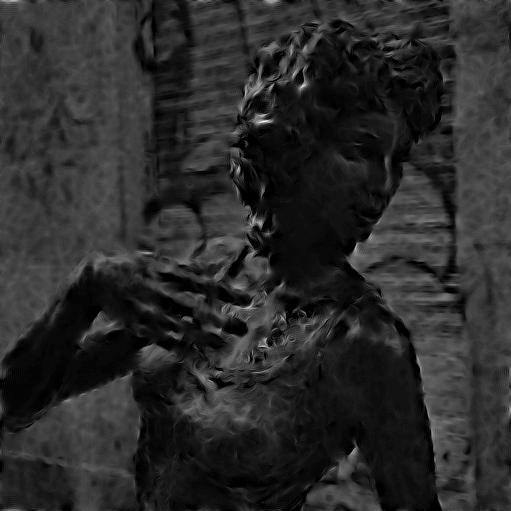}
\includegraphics[width=.325\textwidth]{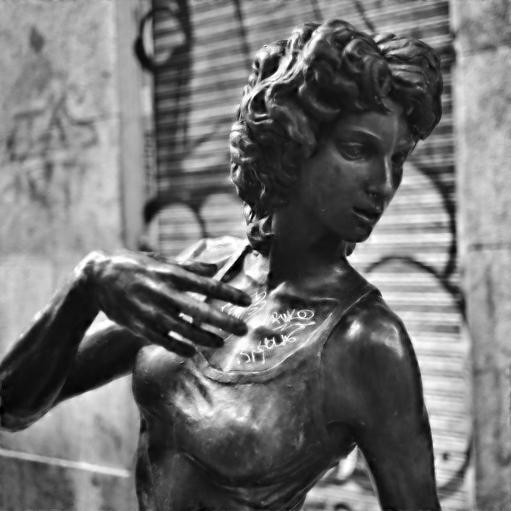}
\includegraphics[width=.33\textwidth]{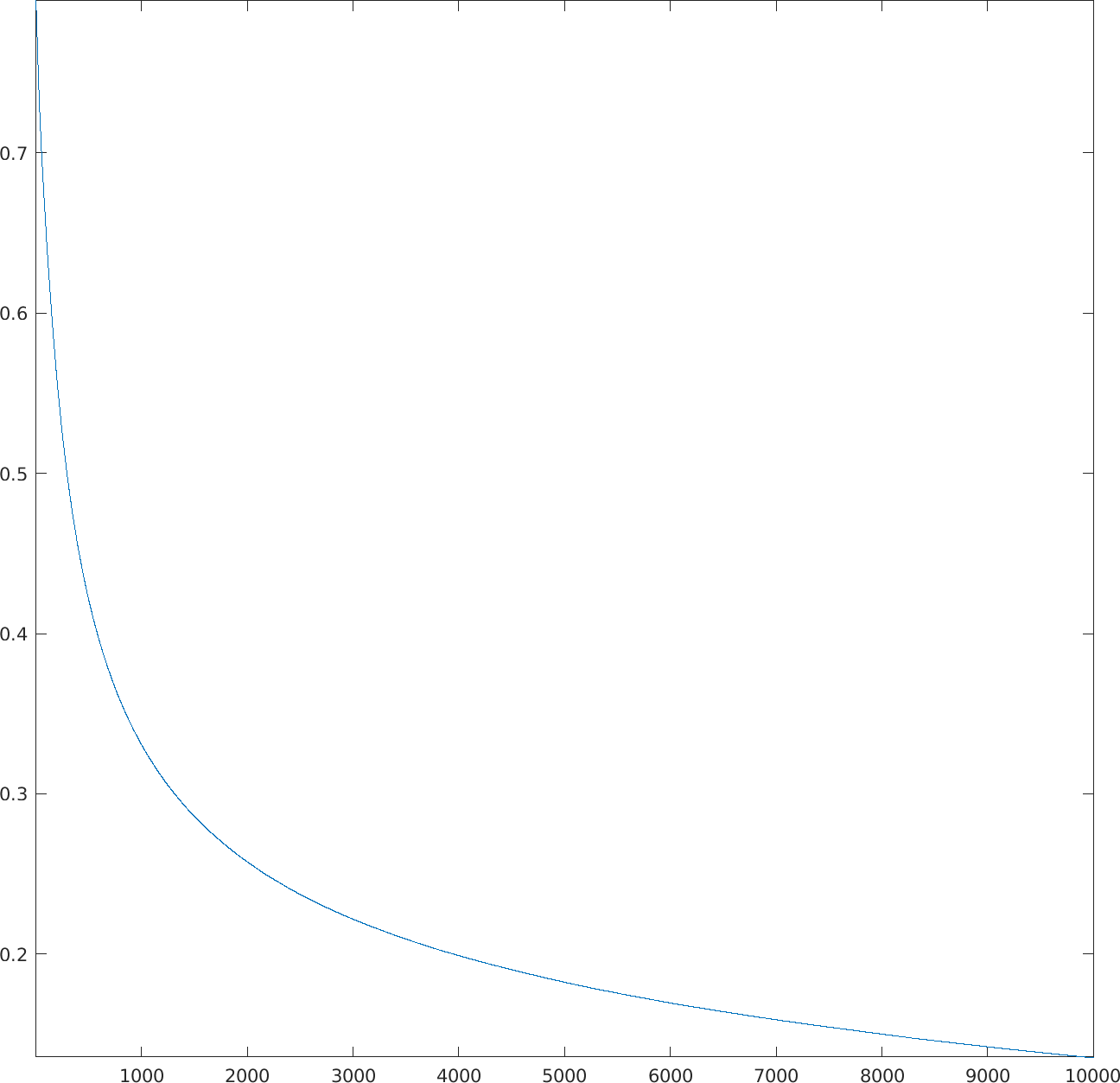}
\end{center}
\caption{Iteration on the images 1 to 4 of Figure \ref{fig:dataset} for $\Theta_{0.4}$ shown in the third line of Figure \ref{fig:pins}. Left: first step of the iteration. Center: after 10000 iterations. Right: $\log_{10}(\Delta_n)$, for the error \eqref{eq:DELTA}.}\label{fig:S3}
\end{figure}

\newpage

\begin{figure}[H]
\begin{center}
\includegraphics[width=.325\textwidth]{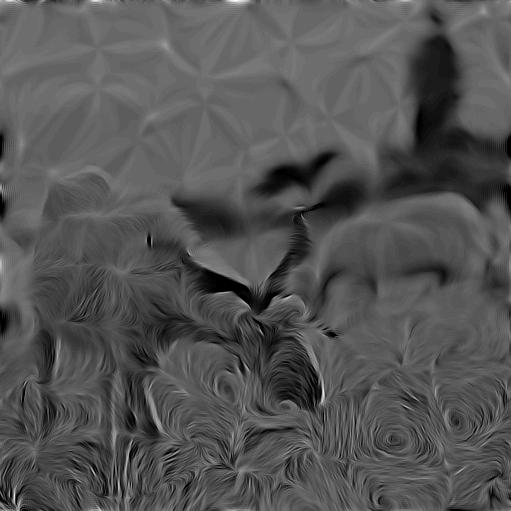}
\includegraphics[width=.325\textwidth]{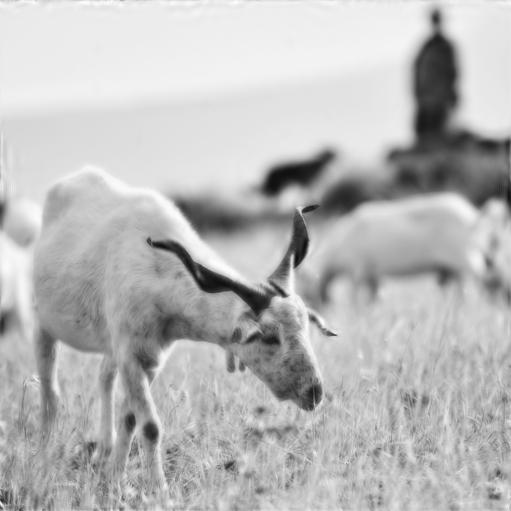}
\includegraphics[width=.33\textwidth]{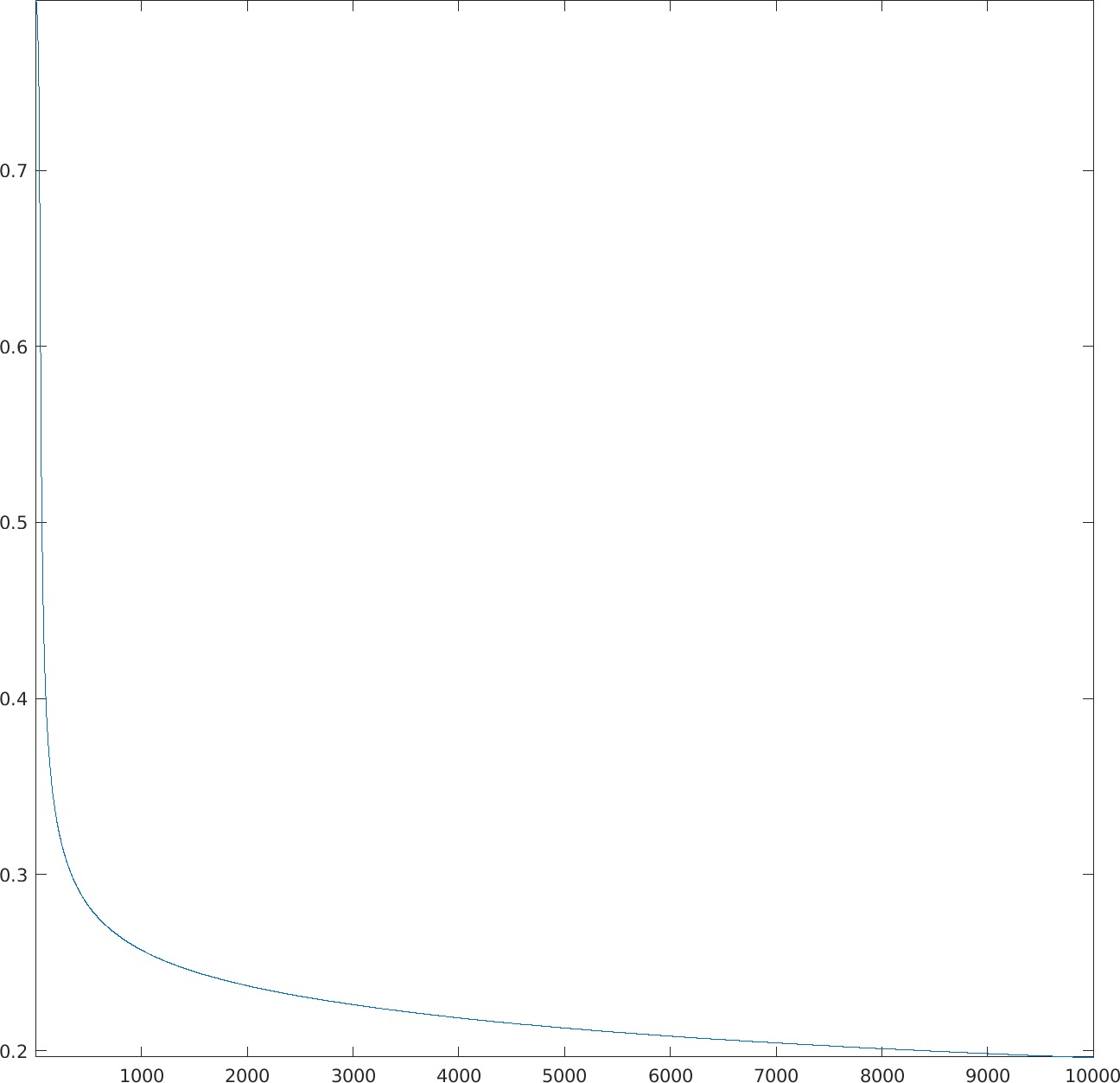}\vspace{4pt}\\
\includegraphics[width=.325\textwidth]{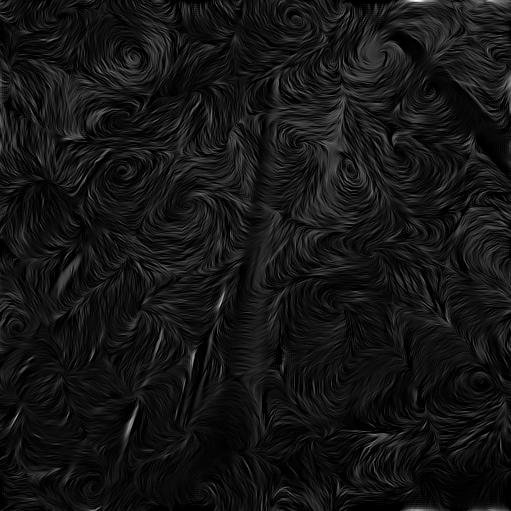}
\includegraphics[width=.325\textwidth]{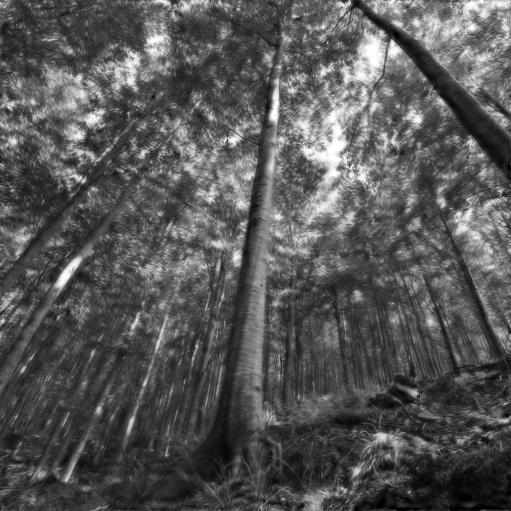}
\includegraphics[width=.33\textwidth]{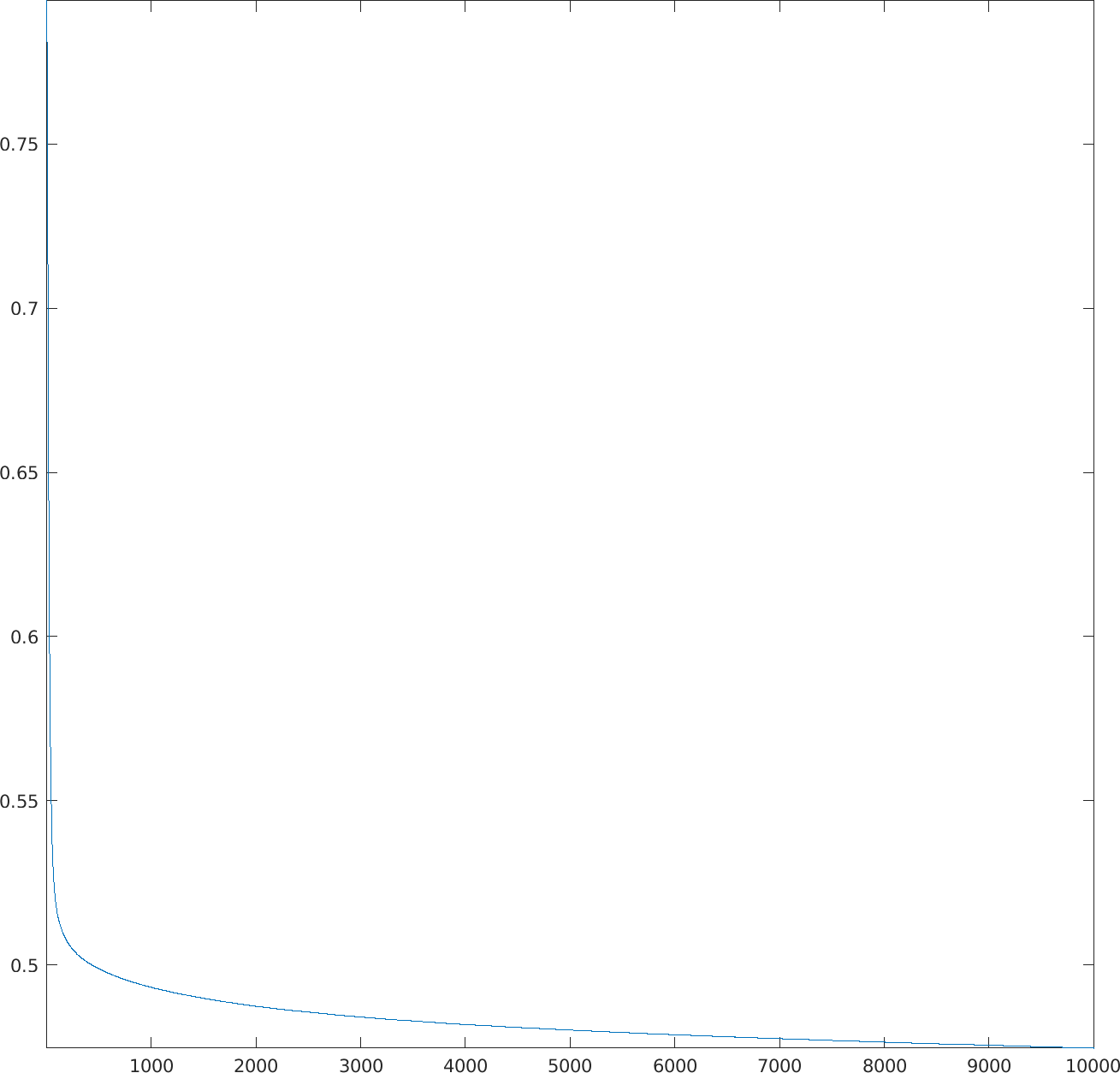}\vspace{4pt}\\
\includegraphics[width=.325\textwidth]{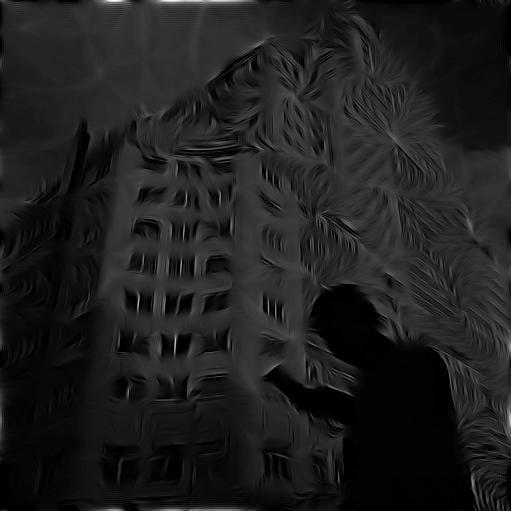}
\includegraphics[width=.325\textwidth]{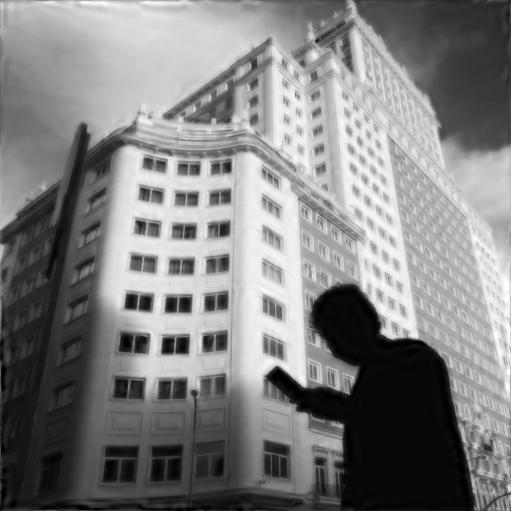}
\includegraphics[width=.33\textwidth]{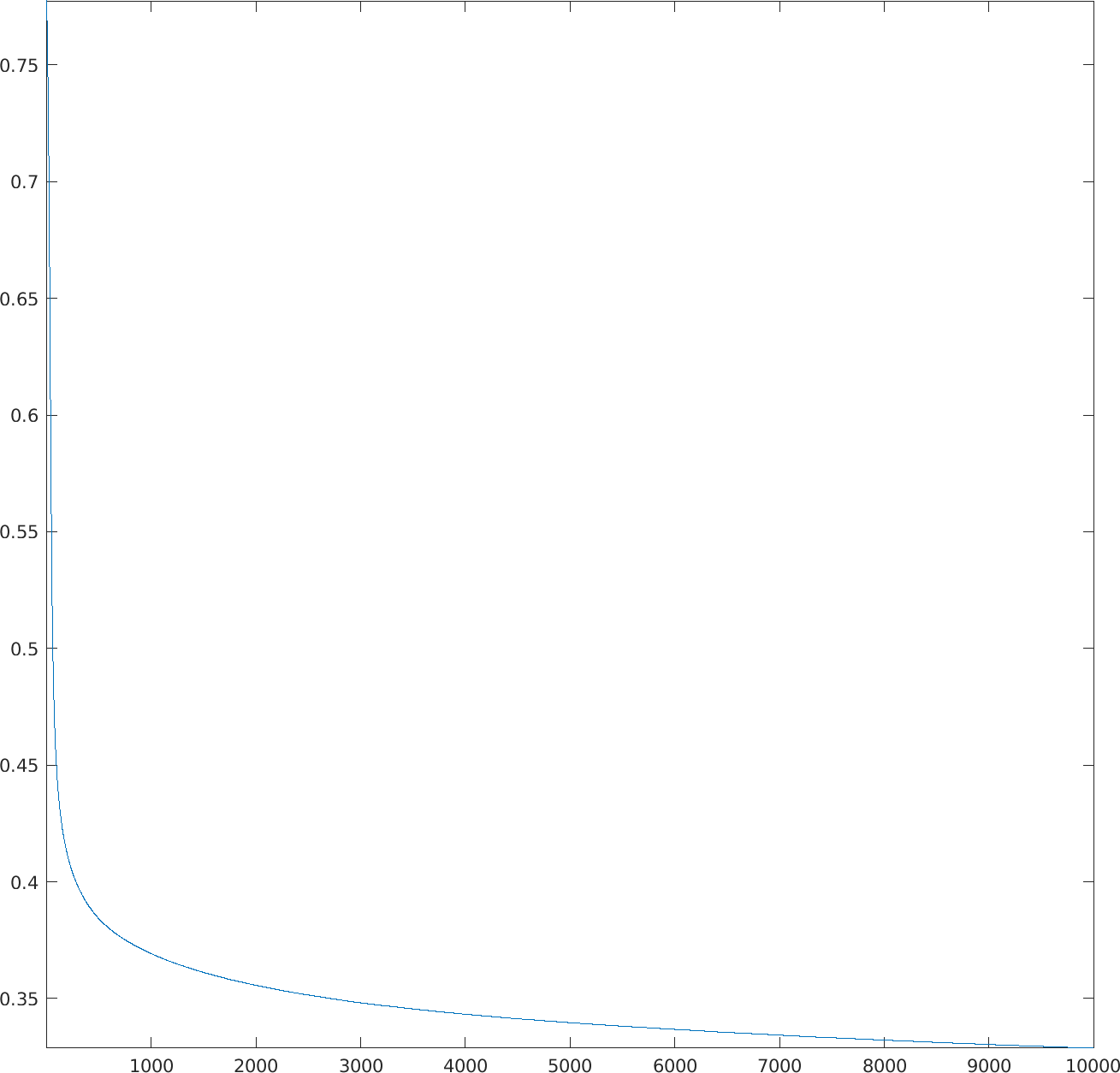}\vspace{4pt}\\
\includegraphics[width=.325\textwidth]{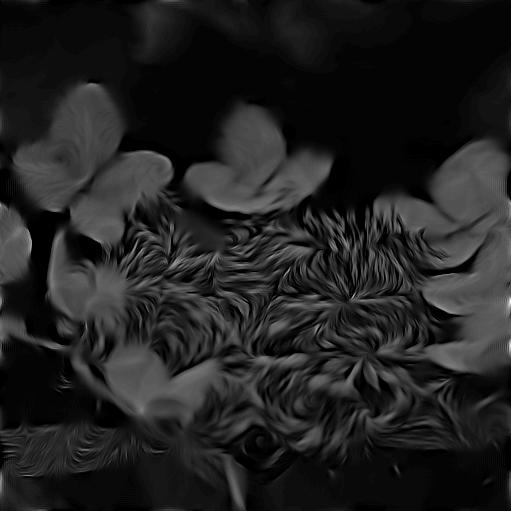}
\includegraphics[width=.325\textwidth]{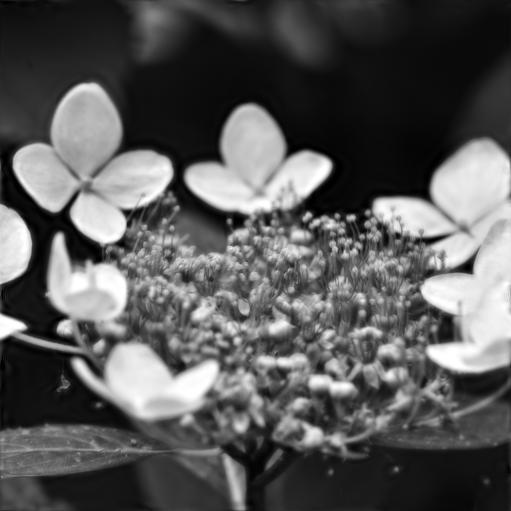}
\includegraphics[width=.33\textwidth]{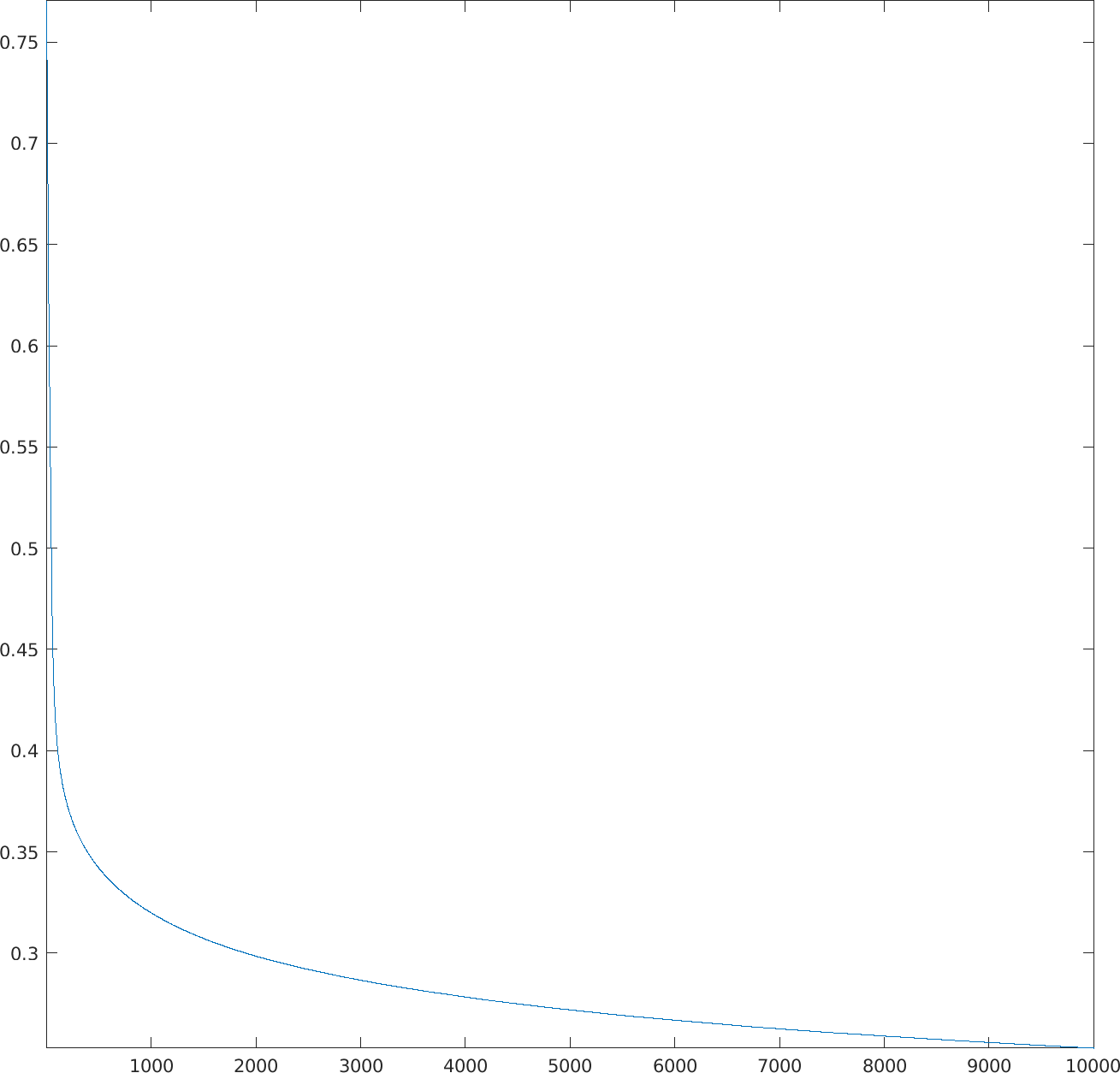}
\end{center}
\caption{Iteration on the images 5 to 8 of Figure \ref{fig:dataset} for $\Theta_{0.06}$, shown in the fourth line of Figure \ref{fig:pins}. Left: first step of the iteration. Center: after 10000 iterations. Right: $\log_{10}(\Delta_n)$, for the error \eqref{eq:DELTA}.}\label{fig:S4}
\end{figure}

\newpage

\begin{figure}[H]
\begin{center}
\includegraphics[width=.325\textwidth]{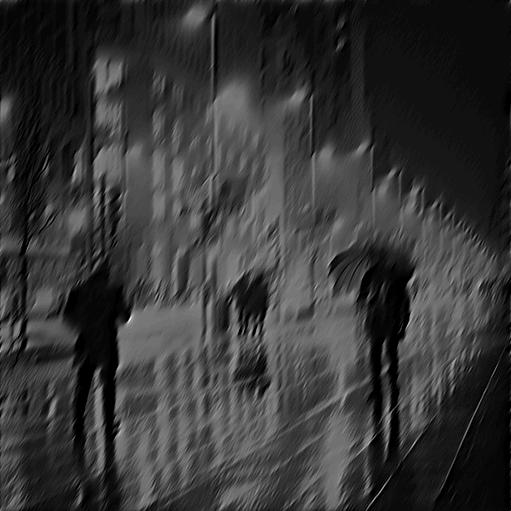}
\includegraphics[width=.325\textwidth]{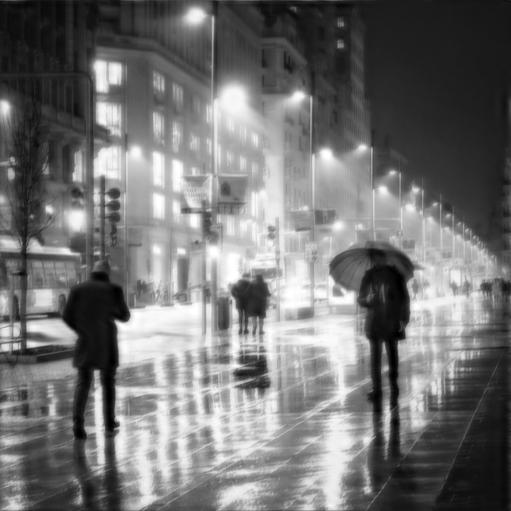}
\includegraphics[width=.33\textwidth]{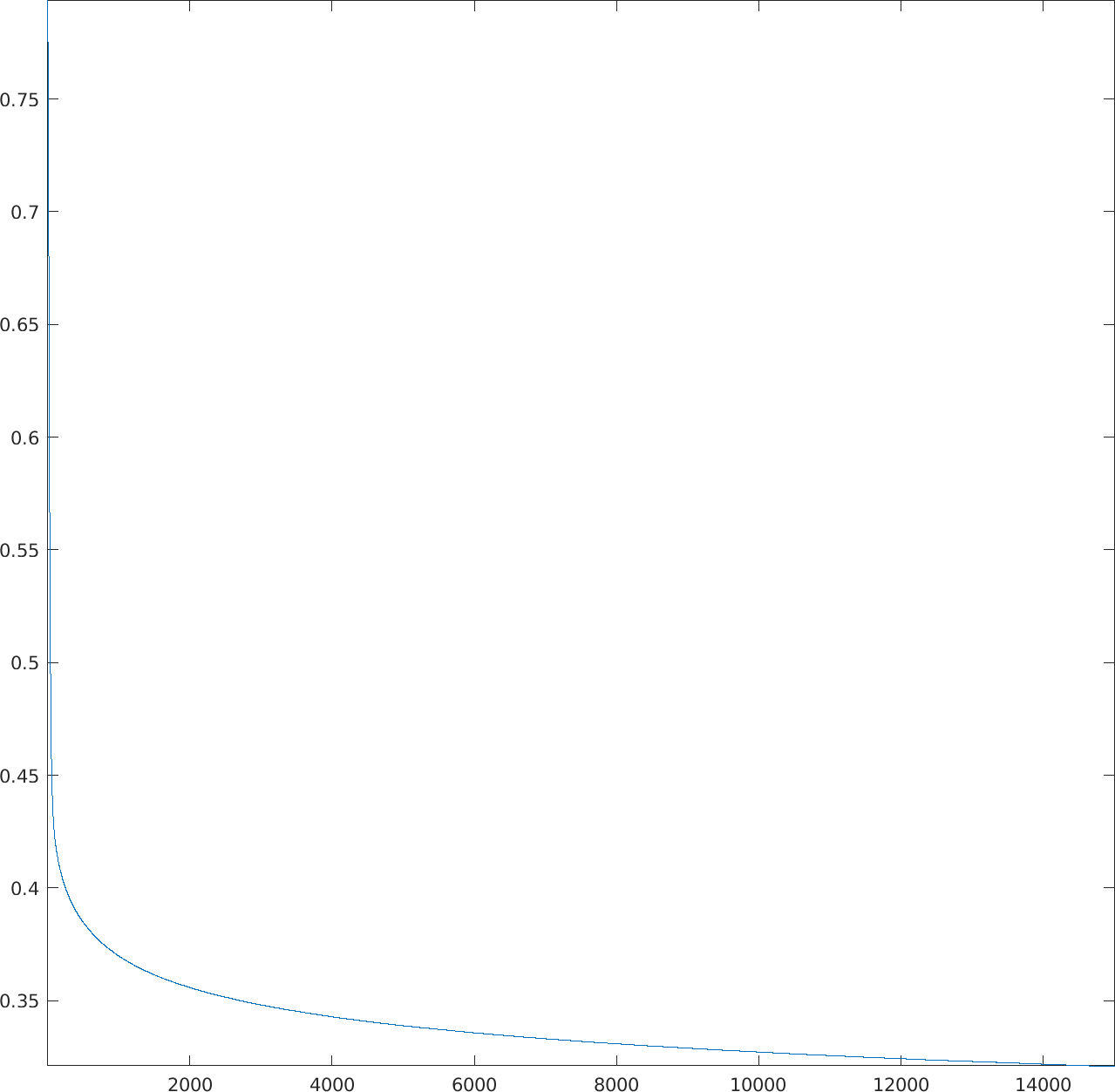}\vspace{4pt}\\
\includegraphics[width=.325\textwidth]{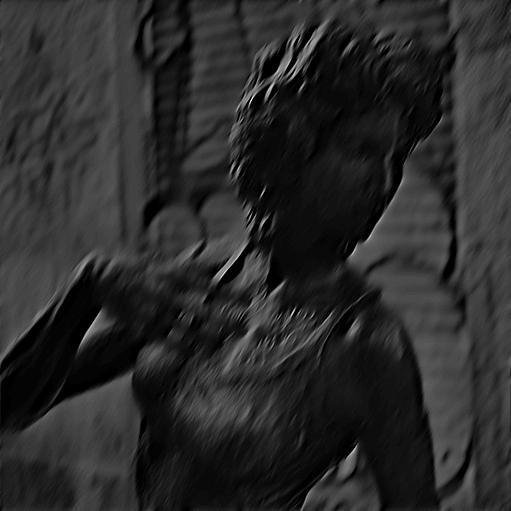}
\includegraphics[width=.325\textwidth]{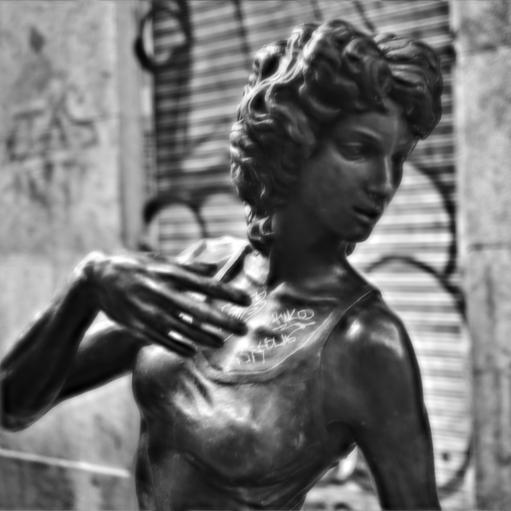}
\includegraphics[width=.33\textwidth]{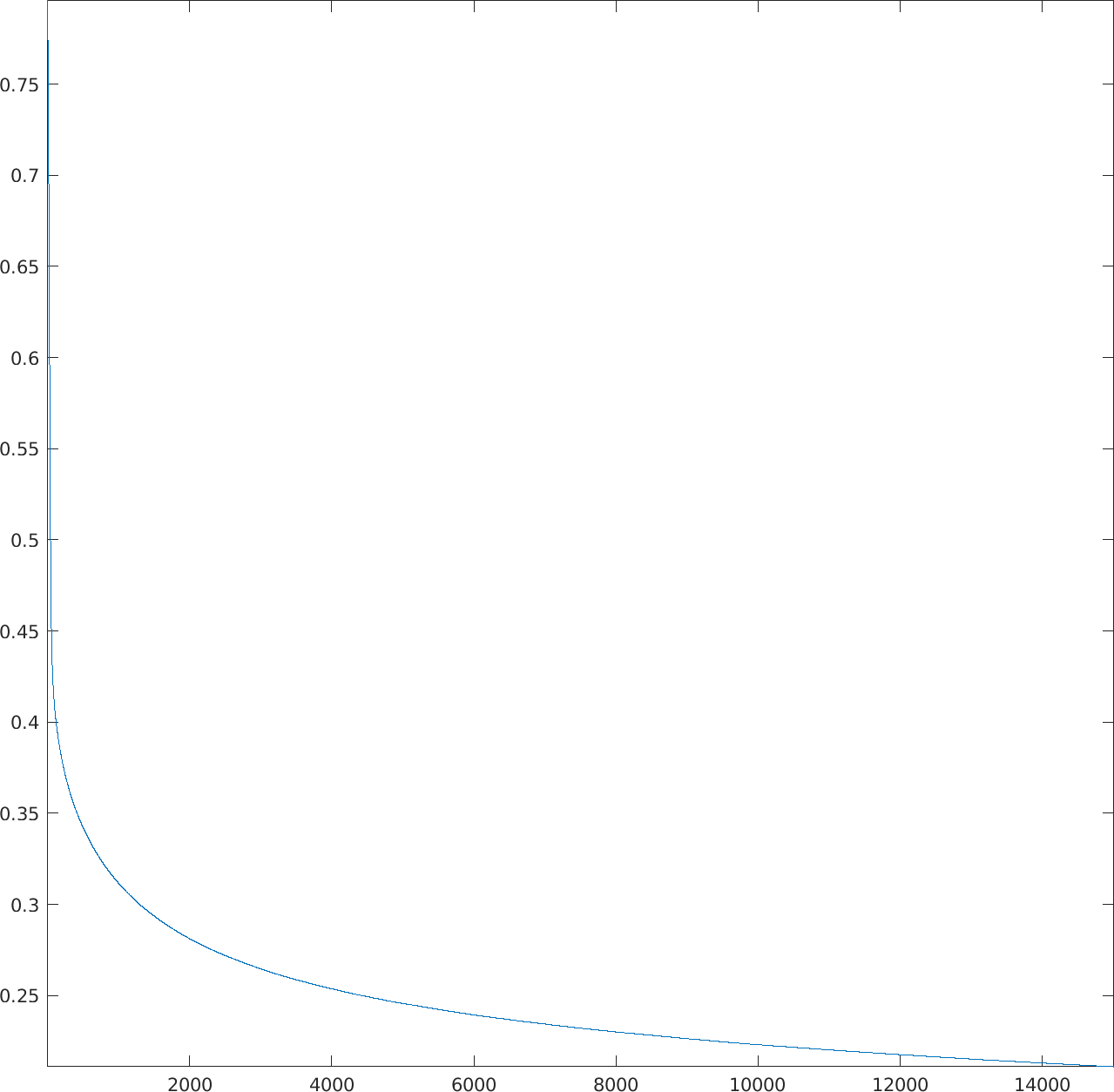}\vspace{4pt}\\
\includegraphics[width=.325\textwidth]{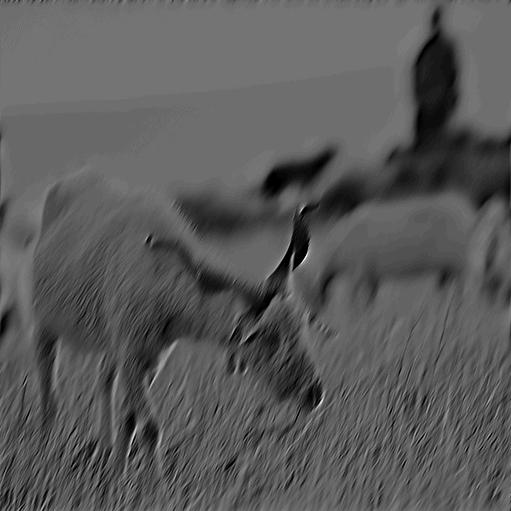}
\includegraphics[width=.325\textwidth]{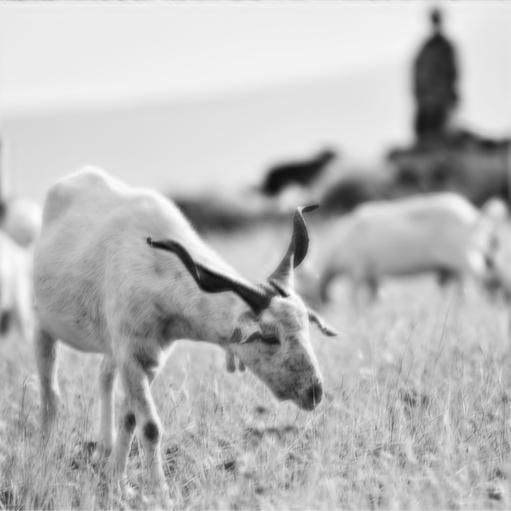}
\includegraphics[width=.33\textwidth]{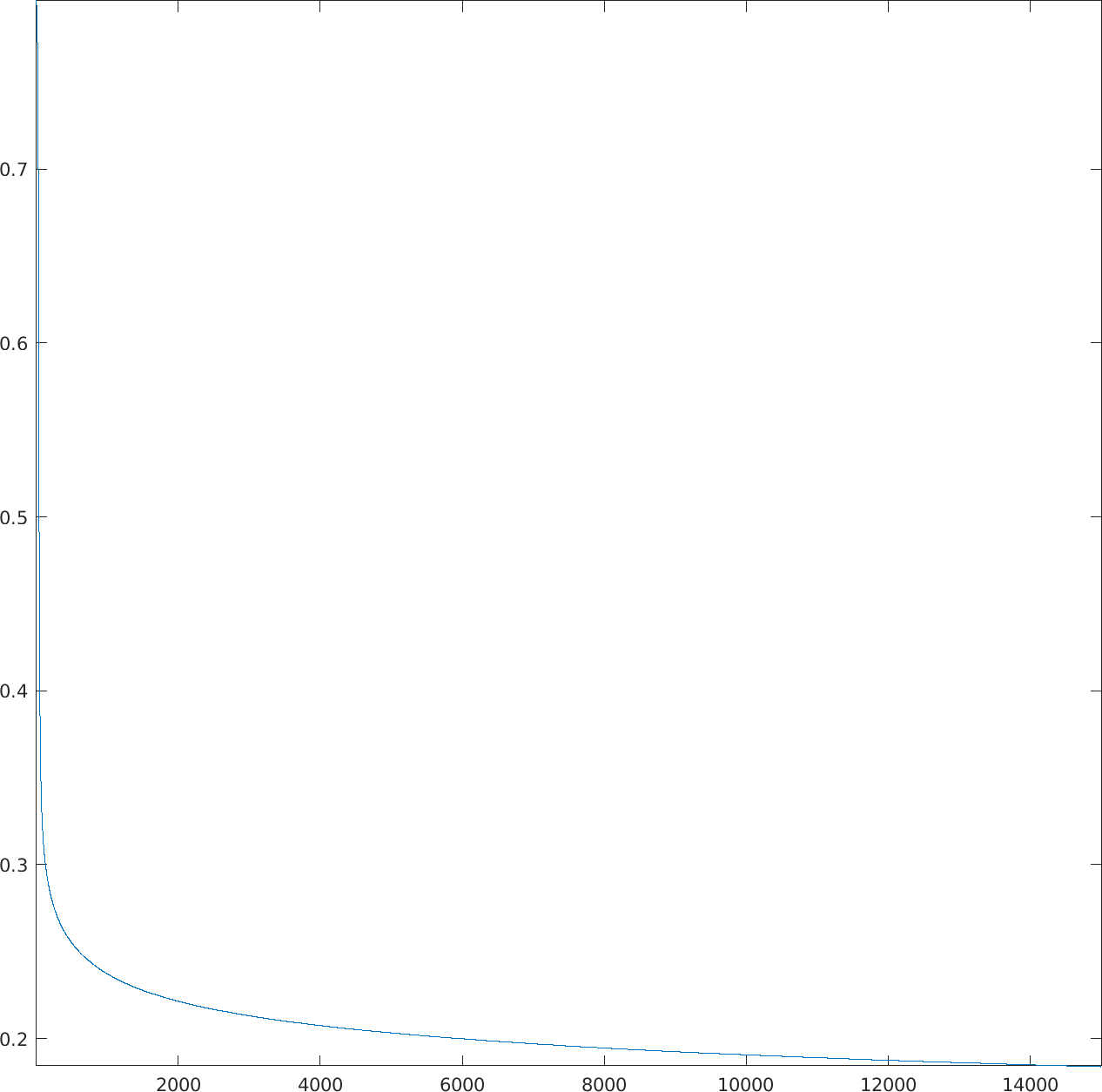}\vspace{4pt}\\
\includegraphics[width=.325\textwidth]{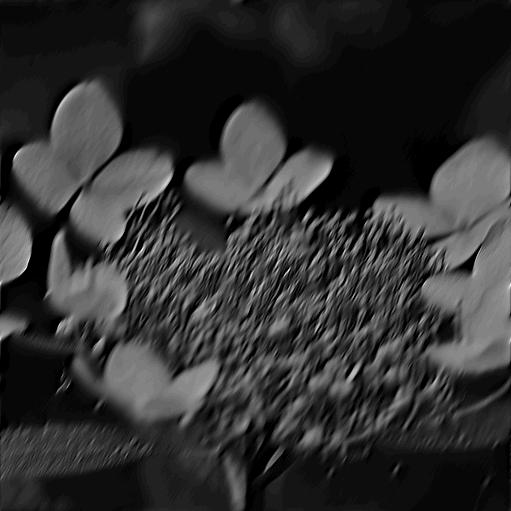}
\includegraphics[width=.325\textwidth]{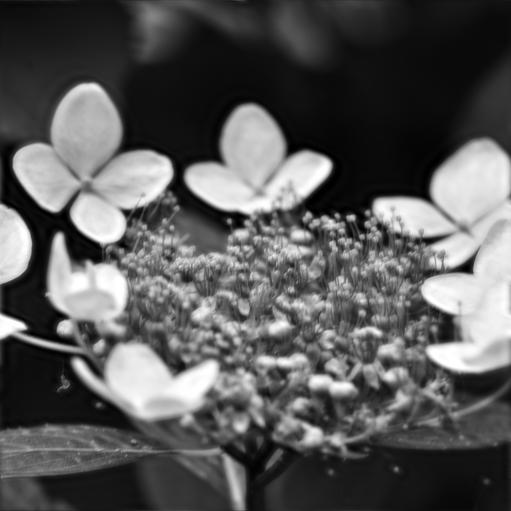}
\includegraphics[width=.33\textwidth]{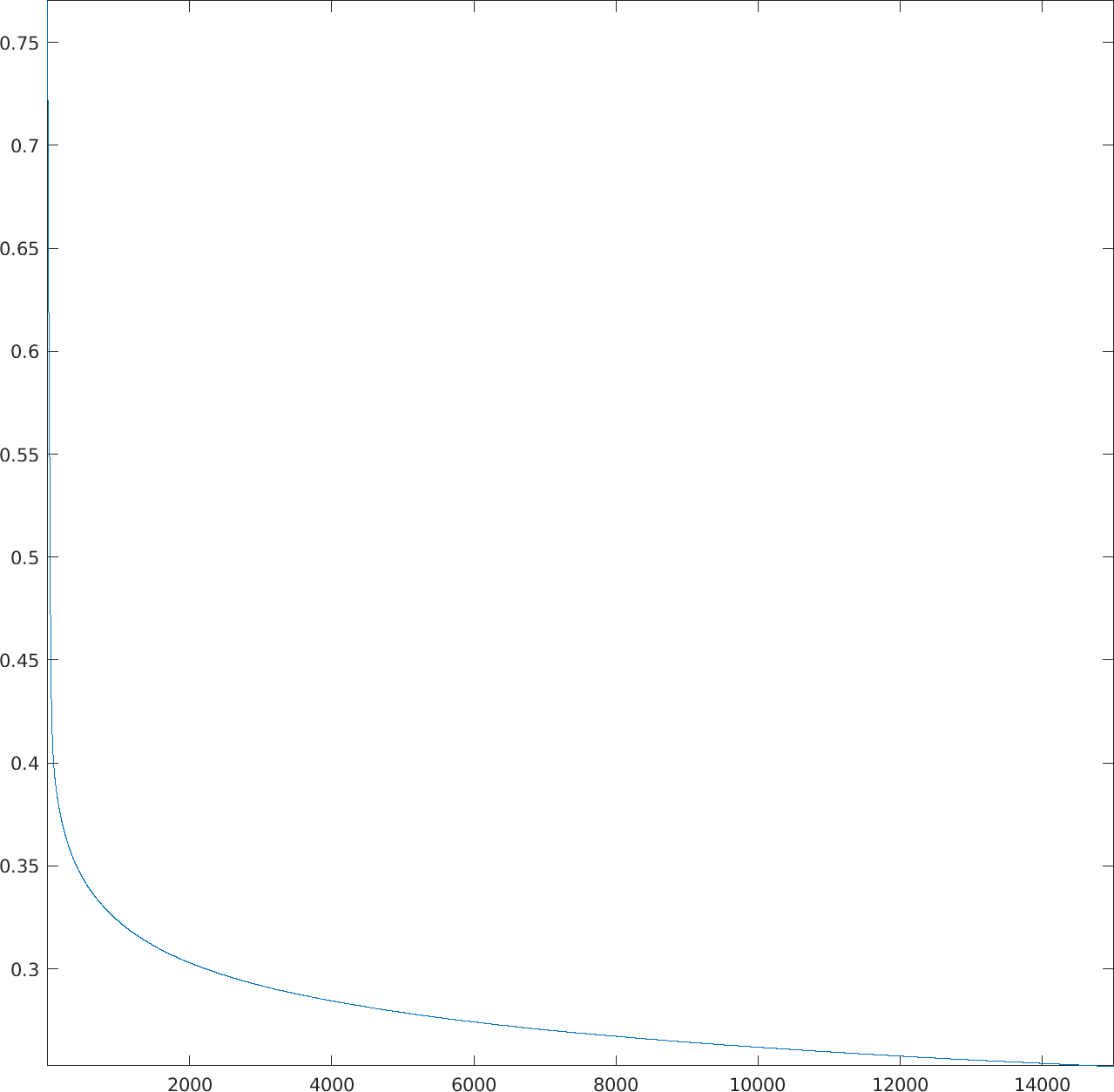}
\end{center}
\caption{Iteration for the $SE(2)$ deconvolution by $\psi$ on images of Figure \ref{fig:dataset}. Left: first step of the iteration. Center: after 15000 iterations. Right: $\log_{10}(\Delta_n)$, for the error \eqref{eq:DELTA}.}\label{fig:S5}
\end{figure}

\newpage

\bibliographystyle{plain}
\bibliography{frontiers.bib}

\begin{thebibliography}{10}

\bibitem{AACM19}
Elona Agora, Jorge Antezana, Carlos Cabrelli, and Basarab Matei.
\newblock Existence of quasicrystals and universal stable sampling and
  interpolation in lca groups.
\newblock {\em Trans. Amer. Math. Soc.}, 372:4647--4674, 2019.

\bibitem{AAG91}
S~T Ali, J~P Antoine, and J~P Gazeau.
\newblock Square integrability of group representations on homogeneous spaces.
  i. reproducing triples and frames.
\newblock {\em Annales de l'I. H. P., section A}, 55:829--855, 1991.

\bibitem{AAG93}
S~T Ali, J~P Antoine, and J~P Gazeau.
\newblock Continuous frames in hilbert space.
\newblock {\em Annals of Physics}, 222:1--37, 1993.

\bibitem{Angelucci02}
A~Angelucci, J~B Levitt, E~J~S Watson, J~M Hupe\'e, J~Bullier, and J~S Lund.
\newblock Circuits for local and global signal integration in primary visual
  cortex.
\newblock {\em J. Neurosci.}, 22:8633--8646, 2002.

\bibitem{AERP19}
F~Anselmi, G~Evangelopoulos, L~Rosasco, and T~Poggio.
\newblock Symmetry-adapted representation learning.
\newblock {\em Pattern Recognition}, 86:201--208, 2019.

\bibitem{Anselmi20}
F~Anselmi, A~Patel, and L~Rosasco.
\newblock Neurally plausible mechanisms for learning selective and invariant
  representations.
\newblock {\em J. Math. Neurosc.}, 10:12, 2020.

\bibitem{PA14}
F~Anselmi and T~A Poggio.
\newblock Representation learning in sensory cortex: a theory.
\newblock {\em CBMM Memo Series}, 026, 2014.

\bibitem{Antoine2D}
J~P Antoine, Romain Murenzi, Pierre Vandergheynst, and Syed~Twareque Ali.
\newblock {\em Two-Dimensional Wavelets and their Relatives}.
\newblock Cambridge University Press, 2004.

\bibitem{BCS14}
D~Barbieri, G~Citti, and A~Sarti.
\newblock How uncertainty bounds the shape index of simple cells.
\newblock {\em J. Math. Neurosc.}, 4:5, 2014.

\bibitem{Barbieri15}
Davide Barbieri.
\newblock Geometry and dimensionality reduction of feature spaces in primary
  visual cortex.
\newblock {\em Proc. SPIE, Wavelets and Sparsity XVI}, 9597:95970J, 2015.

\bibitem{BonhofferGrinvald91}
T~Bonhoeffer and A~Grinvald.
\newblock Iso-orientation domains in cat visual cortex are arranged in
  pinwheel-like patterns.
\newblock {\em Nature}, 353:429--431, 1991.

\bibitem{Bosking02}
William~H Bosking, Justin~C Crowley, and David Fitzpatrick.
\newblock Spatial coding of position and orientation in primary visual cortex.
\newblock {\em Nature Neuroscience}, 5:874--882, 2002.

\bibitem{BoskingFitzpatrick97}
William~H Bosking, Ying Zhang, Brett Schofield, and David Fitzpatrick.
\newblock Orientation selectivity and the arrangement of horizontal connections
  in tree shrew striate cortex.
\newblock {\em Journal of Neuroscience}, 17:2112--2127, 1997.

\bibitem{CandesTao06}
E~J Candes, J~Romberg, and T~Tao.
\newblock Robust uncertainty principles: exact signal reconstruction from
  highly incomplete frequency information.
\newblock {\em IEEE Transactions on Information Theory}, 52:489--509, 2006.

\bibitem{CarandiniHeeger12}
M~Carandini and D~J Heeger.
\newblock Normalization as a canonical neural computation.
\newblock {\em Nature Rev. Neurosci.}, 13:51--62, 2012.

\bibitem{Carandini05}
Matteo Carandini, Jonathan~B. Demb, Valerio Mante, David~J. Tolhurst, Yang Dan,
  Bruno~A. Olshausen, Jack~L. Gallant, and Nicole~C. Rust.
\newblock Do we know what the early visual system does?
\newblock {\em J. Neurosci.}, 25:10577--10597, 2005.

\bibitem{CasazzaKutyniok}
P~G Casazza, G~Kutyniok, and (Editors).
\newblock {\em Finite frames}.
\newblock Birkh\"auser, 2013.

\bibitem{CS06}
G~Citti and A~Sarti.
\newblock A cortical based model of perceptual completion in the
  roto-translation space.
\newblock {\em J. Math. Imag. Vis.}, 24:307--326, 2006.

\bibitem{CS15}
G~Citti and A~Sarti.
\newblock The constitution of visual perceptual units in the functional
  architecture of v1.
\newblock {\em J Comput Neurosci}, 38:285–300, 2015.

\bibitem{Cocci12}
G~Cocci, D~Barbieri, and A~Sarti.
\newblock Spatiotemporal receptive fields of cells in v1 are optimally shaped
  for stimulus velocity estimation.
\newblock {\em J. Opt. Soc. Am. A}, 29:130--138, 2012.

\bibitem{DDGL}
Stephan Dahlke, Filippo~De Mari, Philipp Grohs, and Demetrio Labate.
\newblock {\em Harmonic and Applied Analysis. From Groups to Signals}.
\newblock Birkh\"auser, 2015.

\bibitem{Daubechies}
Ingrid Daubechies.
\newblock {\em Ten lectures on wavelets}.
\newblock SIAM, 1992.

\bibitem{DeAngelis95}
G~C DeAngelis, I~Ohzawa, and R~D Freeman.
\newblock Receptive-field dynamics in the central visual pathways.
\newblock {\em Trends Neurosci.}, 18:451--458, 1995.

\bibitem{DeitmarEchterhoff}
Anton Deitmar and Siegfried Echterhoff.
\newblock {\em Principles of Harmonic Analysis, Second Edition}.
\newblock Springer, 2014.

\bibitem{DonohoStark89}
David~L Donoho and Philip~B Stark.
\newblock Uncertainty principles and signal recovery.
\newblock {\em SIAM J. Appl. Math.}, 49:906--931, 1989.

\bibitem{Duits04}
Remco Duits.
\newblock Perceptual organization in image analysis.
\newblock {\em Ph.D. thesis, Eindhoven University of Technology, Department of
  Biomedical Engineering, The Netherlands}, 2005.

\bibitem{Duits07}
Remco Duits, Michael Felsberg, Gosta Granlund, and Bart Ter~Haar Romeny.
\newblock Image analysis and reconstruction using a wavelet transform
  constructed from a reducible representation of the euclidean motion group.
\newblock {\em Int J Comput Vision}, 72:79--102, 2007.

\bibitem{Duits10}
Remco Duits and Erik Franken.
\newblock Left-invariant parabolic evolutions on $se(2)$ and contour
  enhancement via invertible orientation scores part i: Linear left-invariant
  diffusion equations on $se(2)$.
\newblock {\em Quart. Appl. Math.}, 68:255--292, 2010.

\bibitem{EC80}
G~B Ermentrout and J~D Cowan.
\newblock Large scale spatially organized activity in neural nets.
\newblock {\em SIAM J. Appl. Math.}, 38:1--21, 1980.

\bibitem{Fitzpatrick00}
D~Fitzpatrick.
\newblock Seeing beyond the receptive field in primary visual cortex.
\newblock {\em Curr. Opin. Neurobiol.}, 10:438--443, 2000.

\bibitem{Fuhr}
H~F\"uhr.
\newblock {\em Abstract Harmonic Analysis of Continuous Wavelet Transforms}.
\newblock Springer, 2005.

\bibitem{FGHKR17}
Hartmut Fuhr, Karlheinz Grochenig, Antti Haimi, Andreas Klotz, and Jose~Luis
  Romero.
\newblock Density of sampling and interpolation in reproducing kernel hilbert
  spaces.
\newblock {\em J. London Math. Soc.}, 96:663--686, 2017.

\bibitem{GRS18}
K~Grochenig, J~L Romero, and J~Stockler.
\newblock Sampling theorems for shift-invariant spaces, gabor frames, and
  totally positive functions.
\newblock {\em Invent. Math.}, 211:1119–1148, 2018.

\bibitem{Heil07}
Christopher Heil.
\newblock History and evolution of the density theorem for gabor frames.
\newblock {\em The Journal of Fourier Analysis and Applications}, 13:113--166,
  2007.

\bibitem{HoAngelucci21}
Chun Lum~Andy Ho, Robert Zimmermann, Juan~Daniel {Florez Weidinger}, Mario
  Prsa, Manuel Schottdorf, Sam Merlin, Tsuyoshi Okamoto, Koji Ikezoe, Fabien
  Pifferi, Fabienne Aujard, Alessandra Angelucci, Fred Wolf, and Daniel Huber.
\newblock Orientation preference maps in microcebus murinus reveal
  size-invariant design principles in primate visual cortex.
\newblock {\em Current Biology}, 31:733--741.e7, 2021.

\bibitem{HubelWiesel62}
D.~H. Hubel and T.~N. Wiesel.
\newblock Receptive fields, binocular interaction and functional architecture
  in the cat's visual cortex.
\newblock {\em The Journal of Physiology}, 160(1):106--154, 1962.

\bibitem{Keil11}
W~Keil and F~Wolf.
\newblock Coverage, continuity, and visual cortical architecture.
\newblock {\em Neural Syst Circ}, 1:17, 2011.

\bibitem{Duits21}
Maxime~W. Lafarge, Erik~J. Bekkers, Josien~P.W. Pluim, Remco Duits, and Mitko
  Veta.
\newblock Roto-translation equivariant convolutional networks: Application to
  histopathology image analysis.
\newblock {\em Medical Image Analysis}, 68:101849, 2021.

\bibitem{LeCun10}
Yann LeCun, Koray Kavukcuoglu, and Clement Farabet.
\newblock Convolutional networks and applications in vision.
\newblock In {\em Proceedings of 2010 IEEE International Symposium on Circuits
  and Systems}, pages 253--256, 2010.

\bibitem{Marcelja80}
S~Marcelja.
\newblock Mathematical description of the responses of simple cortical cells.
\newblock {\em J. Opt. Soc. Am.}, 70:1297--1300, 1980.

\bibitem{Marr}
David Marr.
\newblock {\em Vision}.
\newblock W. H. Freeman, San Francisco, 1980.

\bibitem{MateiMeyer2010}
Basarab Matei and Yves Meyer.
\newblock Simple quasicrystals are sets of stable sampling.
\newblock {\em Complex Var. Elliptic Equ.}, 55:947--964, 2010.

\bibitem{MCSx}
N~Montobbio, L~Bonnasse-Gahot, G~Citti, and A~Sarti.
\newblock From receptive profiles to a metric model of v1. (preprint).
\newblock {\em arXiv:1803.05783v3}, 2018.

\bibitem{MBCSx}
N~Montobbio, L~Bonnasse-Gahot, G~Citti, and A~Sarti.
\newblock Kercnn: Biologically inspired lateral connections for classification
  of corrupted images. (preprint).
\newblock {\em arXiv:1910.08336v1}, 2019.

\bibitem{PT99}
J~Petitot and Y~Tondut.
\newblock Vers une neurog\'eom\'etrie. fibrations corticales, structures de
  contact et contours subjectifs modaux.
\newblock {\em Math. Sci. Hum.}, 145:5--101, 1999.

\bibitem{PetitotBook}
Jean Petitot.
\newblock {\em Elements of Neurogeometry}.
\newblock Springer, 2017.

\bibitem{RieszNagy}
Frigyes Riesz and Bela Sz.-Nagy.
\newblock {\em Functional Analysis}.
\newblock Dover, 1990.

\bibitem{Ringach02}
Dario Ringach.
\newblock Spatial structure and symmetry of simple cells receptive fields in
  macaque primary visual cortex.
\newblock {\em J. Neurophysiol.}, 88:455--463, 2002.

\bibitem{SCP08}
A~Sarti, G~Citti, and J~Petitot.
\newblock The symplectic structure of the primary visual cortex.
\newblock {\em Biol. Cybern.}, 98:33--48, 2008.

\bibitem{Weiss01}
G.~Weiss and E.~N. Wilson.
\newblock The mathematical theory of wavelets.
\newblock In James~S. Byrnes, editor, {\em Twentieth Century Harmonic Analysis
  --- A Celebration}, pages 329--366. Springer, 2001.

\bibitem{Welicky96}
M~Weliky, W~Bosking, and David Fitzpatrick.
\newblock A systematic map of direction preference in primary visual cortex.
\newblock {\em Nature}, 379:725--728, 1996.

\bibitem{White07}
Leonard~E. White and David Fitzpatrick.
\newblock Vision and cortical map development.
\newblock {\em Neuron}, 56(2):327--338, 2007.

\bibitem{WC72}
H~R Wilson and J~D Cowan.
\newblock Excitatory and inhibitory interactions in localized populations of
  model neurons.
\newblock {\em Biophys. J.}, 12:1--24, 1972.

\bibitem{Duits16}
Jiong Zhang, Behdad Dashtbozorg, Erik Bekkers, Josien Pluim, Remco Duits, and
  Bart ter Haar~Romeny.
\newblock Robust retinal vessel segmentation via locally adaptive derivative
  frames in orientation scores.
\newblock {\em IEEE Transactions on Medical Imaging}, 35:2631--2644, 2016.

\end{thebibliography}

\newpage

\section*{Funding}
This project has received funding from the European Union's Horizon 2020 research and innovation programme under the Marie Sk\l odowska-Curie grant agreement No 777822, and from Grant PID2019-105599GB-I00, Ministerio de Ciencia, Innovaci\'on y Universidades, Spain.

\section*{Acknowledgments}
The author would like to thank G. Citti, A. Sarti, E. Hern\'andez, J. Antezana and F. Anselmi for inspiring discussion on topics related to the present work.

\section*{Address}
Davide Barbieri, Universidad Aut\'onoma de Madrid\\
\tt{davide.barbieri@uam.es}

\end{document}